\documentclass[11pt]{article}

\usepackage{enumerate,xspace}
\usepackage{amsmath,amssymb,wasysym}
\usepackage[all]{xy}
\usepackage{proof}
\usepackage[svgnames]{xcolor}
\usepackage{tikz}
\usepackage{scalerel}

\usepackage{stmaryrd} %need for special interpretation bracket
\usepackage{mathtools}
\usepackage{latexsym}
\usepackage{dsfont}
\usepackage{multicol}
\usepackage{cmll}
\usepackage{multirow}
\usepackage{longtable}

\usepackage{lscape}
\usepackage{array}

\newcommand\abs[1]{\left|#1\right|}

\usepackage{hyperref} %hyperlink package
\hypersetup{
    colorlinks,
    citecolor=red,
    filecolor=red,
    linkcolor=blue,
    urlcolor=red
}

\newtheorem{observation}{Remark}[section]
\newtheorem{lemma}[observation]{Lemma}  %%share counter with remark
\newtheorem{theorem}[observation]{Theorem}
\newtheorem{definition}[observation]{Definition}
\newtheorem{example}[observation]{Example}

\newtheorem{proposition}[observation]{Proposition} 
\newtheorem{corollary}[observation]{Corollary}

\makeatletter

\newdimen\w@dth

\def\setw@dth#1#2{\setbox\z@\hbox{\scriptsize $#1$}\w@dth=\wd\z@
\setbox\@ne\hbox{\scriptsize $#2$}\ifnum\w@dth<\wd\@ne \w@dth=\wd\@ne \fi
\advance\w@dth by 1.2em}

\def\t@^#1_#2{\allowbreak\def\n@one{#1}\def\n@two{#2}\mathrel
{\setw@dth{#1}{#2}
\mathop{\hbox to \w@dth{\rightarrowfill}}\limits
\ifx\n@one\empty\else ^{\box\z@}\fi
\ifx\n@two\empty\else _{\box\@ne}\fi}}
\def\t@@^#1{\@ifnextchar_ {\t@^{#1}}{\t@^{#1}_{}}}

\def\t@left^#1_#2{\def\n@one{#1}\def\n@two{#2}\mathrel{\setw@dth{#1}{#2}
\mathop{\hbox to \w@dth{\leftarrowfill}}\limits
\ifx\n@one\empty\else ^{\box\z@}\fi
\ifx\n@two\empty\else _{\box\@ne}\fi}}
\def\t@@left^#1{\@ifnextchar_ {\t@left^{#1}}{\t@left^{#1}_{}}}

\def\two@^#1_#2{\def\n@one{#1}\def\n@two{#2}\mathrel{\setw@dth{#1}{#2}
\mathop{\vcenter{\hbox to \w@dth{\rightarrowfill}\kern-1.7ex
                 \hbox to \w@dth{\rightarrowfill}}%
       }\limits
\ifx\n@one\empty\else ^{\box\z@}\fi
\ifx\n@two\empty\else _{\box\@ne}\fi}}
\def\tw@@^#1{\@ifnextchar_ {\two@^{#1}}{\two@^{#1}_{}}}

\def\tofr@^#1_#2{\def\n@one{#1}\def\n@two{#2}\mathrel{\setw@dth{#1}{#2}
\mathop{\vcenter{\hbox to \w@dth{\rightarrowfill}\kern-1.7ex
                 \hbox to \w@dth{\leftarrowfill}}%
       }\limits
\ifx\n@one\empty\else ^{\box\z@}\fi
\ifx\n@two\empty\else _{\box\@ne}\fi}}
\def\t@fr@^#1{\@ifnextchar_ {\tofr@^{#1}}{\tofr@^{#1}_{}}}

\newdimen\W@dth
\def\setW@dth#1#2{\setbox\z@\hbox{$#1$}\W@dth=\wd\z@
\setbox\@ne\hbox{$#2$}\ifnum\W@dth<\wd\@ne \W@dth=\wd\@ne \fi
\advance\W@dth by 1.2em}

\def\T@^#1_#2{\allowbreak\def\N@one{#1}\def\N@two{#2}\mathrel
{\setW@dth{#1}{#2}
\mathop{\hbox to \W@dth{\rightarrowfill}}\limits
\ifx\N@one\empty\else ^{\box\z@}\fi
\ifx\N@two\empty\else _{\box\@ne}\fi}}
\def\T@@^#1{\@ifnextchar_ {\T@^{#1}}{\T@^{#1}_{}}}

\def\T@left^#1_#2{\def\N@one{#1}\def\N@two{#2}\mathrel{\setW@dth{#1}{#2}
\mathop{\hbox to \W@dth{\leftarrowfill}}\limits
\ifx\N@one\empty\else ^{\box\z@}\fi
\ifx\N@two\empty\else _{\box\@ne}\fi}}
\def\T@@left^#1{\@ifnextchar_ {\T@left^{#1}}{\T@left^{#1}_{}}}

\def\Tofr@^#1_#2{\def\N@one{#1}\def\N@two{#2}\mathrel{\setW@dth{#1}{#2}
\mathop{\vcenter{\hbox to \W@dth{\rightarrowfill}\kern-1.7ex
                 \hbox to \W@dth{\leftarrowfill}}%
       }\limits
\ifx\N@one\empty\else ^{\box\z@}\fi
\ifx\N@two\empty\else _{\box\@ne}\fi}}
\def\T@fr@^#1{\@ifnextchar_ {\Tofr@^{#1}}{\Tofr@^{#1}_{}}}

\def\Two@^#1_#2{\def\N@one{#1}\def\N@two{#2}\mathrel{\setW@dth{#1}{#2}
\mathop{\vcenter{\hbox to \W@dth{\rightarrowfill}\kern-1.7ex
                 \hbox to \W@dth{\rightarrowfill}}%
       }\limits
\ifx\N@one\empty\else ^{\box\z@}\fi
\ifx\N@two\empty\else _{\box\@ne}\fi}}
\def\Tw@@^#1{\@ifnextchar_ {\Two@^{#1}}{\Two@^{#1}_{}}}

\def\to{\@ifnextchar^ {\t@@}{\t@@^{}}}
\def\from{\@ifnextchar^ {\t@@left}{\t@@left^{}}}
\def\tofro{\@ifnextchar^ {\t@fr@}{\t@fr@^{}}}
\def\To{\@ifnextchar^ {\T@@}{\T@@^{}}}
\def\From{\@ifnextchar^ {\T@@left}{\T@@left^{}}}
\def\Two{\@ifnextchar^ {\Tw@@}{\Tw@@^{}}}
\def\Tofro{\@ifnextchar^ {\T@fr@}{\T@fr@^{}}}

\makeatother

\title{Linearizing Combinators}
\author{Robin Cockett and Jean-Simon Pacaud Lemay}
\begin{document}
\allowdisplaybreaks

\maketitle
%\section{}
%\subsection{}

\begin{abstract} In 2017, Bauer, Johnson, Osborne, Riehl, and Tebbe (BJORT) showed that the abelian functor calculus provides an example of a Cartesian differential category.  The definition of a Cartesian differential category is based on a differential combinator which directly formalizes the total derivative from multivariable calculus.  However, in the aforementioned  work the authors used techniques from Goodwillie's functor calculus to establish a linearization process from which they then derived a differential combinator.  This raised the question of what the precise relationship between linearization and having a differential combinator might be.  

In this paper, we introduce the notion of a {\em linearizing combinator\/} which abstracts linearization in the abelian functor calculus.  We then use it to provide an alternative axiomatization of a Cartesian differential category.  Every Cartesian differential category comes equipped with a canonical linearizing combinator obtained by differentiation at zero.  Conversely, a differential combinator can be constructed {\`a} la BJORT when one has a system of {\em partial\/} linearizing combinators in each context.  Thus, while linearizing combinators do provide an alternative axiomatization of Cartesian differential categories, an explicit notion of partial linearization is required.  This is in contrast to the situation for differential combinators where partial differentiation is automatic in the presence of total differentiation.  The ability to form a system of partial linearizing combinators from a total linearizing combinator, while not being possible in general, is possible when the setting is Cartesian closed.
\end{abstract}

\noindent {\small \textbf{Acknowledgements:} The authors would like to thank Kristine Bauer for her help on this project, as well as Brenda Johnson and Sarah Yeakel for useful discussions at the 2018 Canadian Mathematical Society Summer Meeting which initiated this research project. The first author is partially supported by NSERC (Canada). The second author would like to thank Kellogg College, the Clarendon Fund, and the Oxford Google-DeepMind Graduate Scholarship for financial support for this project. }

\newpage
\tableofcontents

%%%%%%%%%%%%%%%%%%%%%%%%%%%%%%%%%%%%%%%%%%%%%%%%%%%%%%%%%%%%%%%%%%%%%

\section{Introduction}

Cartesian differential categories, introduced by Blute, Cockett, and Seely in \cite{blute2009cartesian}, are left additive categories which are equipped with a differential combinator $\mathsf{D}$ which formalizes the derivative from multivariable calculus over Euclidean spaces.  For every map $f: A \to B$, the differential combinator produces its derivatives $\mathsf{D}[f]: A \times A \to B$, which is \emph{linear} in its second argument. The notion of linearity in a Cartesian differential category is defined with respect to the differential combinator and often coincides with the classical notion from linear algebra. In particular, linearity in a Cartesian differential category always implies additivity. That said, there are examples of Cartesian differential categories where a map may be additive yet not linear. There is no shortage of examples of Cartesian differential categories: the category of Euclidean spaces and real smooth functions between them, the Lawvere Theory of polynomials over a commutative rig, any category with finite biproducts, cofree Cartesian differential categories \cite{cockett2011faa,lemay2018tangent}, the coKleisli category of a differential category \cite{Blute2019,blute2006differential} which include models such as convenient vector spaces \cite{blute2010convenient,kriegl1997convenient,manzyuk2012tangent}, and categorical models of the differential $\lambda$-calculus \cite{bucciarelli2010categorical,Cockett-2019,EHRHARD20031,manzonetto_2012}. 

abelian functor calculus was developed by Johnson and McCarthy in \cite{johnson2004deriving}, based on Goodwillie's functor calculus \cite{goodwillie1990calculus,goodwillie1991calculus,goodwillie2003calculus}. In \cite{bauer2018directional}, Bauer, Johnson, Osborne, Riehl, and Tebbe (BJORT) showed that, using the abelian functor calculus, the homotopy category of the category of abelian categories is a Cartesian differential category. The differential combinator $\nabla(-)$ (referred to as the \emph{directional derivative} in \cite[Section 6]{bauer2018directional}) is defined as \cite[Definition 6.1]{bauer2018directional} $\nabla F (X,V) := D_1(F(X \oplus -))(V)$\footnote{Here the second argument is the linear argument.}, where $D_1(G)$ is the \emph{linearization} (or \emph{linear approximation}) of a functor $G$ \cite[Section 5]{bauer2018directional}. 

From the Cartesian differential category perspective, the BJORT construction is backwards.  In any Cartesian differential category it is always possible to define the notion of a linear map and, indeed, to linearize a map using the differential combinator.   However, BJORT constructed their differential combinator using an already established notion of linear map and linearization. The goal of this paper is to reverse engineer BJORT's construction by abstracting the notion of linear approximation $\mathsf{D}_1$ from the (abelian) functor calculus. To this end, we introduce the notion of a \textbf{linearizing combinator} and show that every Cartesian differential category comes equipped with a canonical system of linearizing combinators built from the differential combinator. Furthermore, we show that the differential combinator can be reconstructed {\`a} la BJORT using such a system of linearizing combinators.  In this manner, we show that linearizing combinators do, in fact, provide an alternative axiomatization of Cartesian differential categories. 

To better understand the BJORT construction, let us consider classical multivariable calculus. Given a smooth function $f: \mathbb{R} \to \mathbb{R}$, linearization $\mathsf{L}[f]: \mathbb{R} \to \mathbb{R}$ is the best $\mathbb{R}$-linear function which is closest to $f$. This is given by the first degree term in its Maclaurin series expansion (i.e its Taylor series expansion at $0$), that is, $\mathsf{L}[f](x) = f^\prime(0) x$, which is indeed an $\mathbb{R}$-linear function. In terms of the differential combinator, its differential $\mathsf{D}[f]: \mathbb{R} \times \mathbb{R} \to \mathbb{R}$ is defined as $\mathsf{D}[f](x,y) = f^\prime(x) y$, and so $\mathsf{L}[f](x) = \mathsf{D}[f](0,x)$. Therefore, in an arbitrary Cartesian differential category, the linearizing combinator $\mathsf{L}$ is defined by first applying the differential combinator and then evaluating the derivative at zero in its first argument: 
\[ \infer{\text{Evaluate at zero in the first argument} \qquad \mathsf{L}[f]:= \xymatrixcolsep{3pc}\xymatrix{A \ar[r]^-{\langle 0, 1 \rangle} & A \times A \ar[r]^-{\mathsf{D}[f]} & B}}{\infer{ \text{Apply the differential combinator} \qquad \xymatrixcolsep{5pc}\xymatrix{A \times A \ar[r]^-{\mathsf{D}[f]} & B}}{ \xymatrixcolsep{5pc}\xymatrix{A \ar[r]^-{f} & B} }} \]
We can use this to derive an abstract notion of a linearizing combinator, $\mathsf{L}$, for arbitrary Cartesian left additive categories, which satisfies axioms which parallel those of the differential combinator. These include a sort of chain rule for linearizing a composite and the fact that the linearization of a map is always additive. In particular, one can then show that $D_1$, from abelian functor calculus, is an example of such an abstract linearizing combinator. 

To define a differential combinator from linearization, the ability to perform linearization in context is required.  We refer to linearization in context as {\em partial\/} linearization because differentiation in context is usually called partial differentiation.  Consider the classical limit definition of the derivative of a smooth function ${f: \mathbb{R} \to \mathbb{R}}$:
\[ \mathsf{D}[f](x, y) = \lim \limits_{t \to 0} \frac{f(x + ty) - f(x)}{t} \]
Note that if we evaluate at $x=0$, then we obtain an expression of $\mathsf{L}[f]$ in terms of a limit: 
\[ \mathsf{L}[f](y) =  \mathsf{D}[f](0, y) = \lim \limits_{t \to 0} \frac{f(ty) - f(0)}{t} \]
For a fixed $x$, define $g_x: \mathbb{R} \to \mathbb{R}$ to be the smooth function defined as $g_x(y) = f(x + y)$. Then: 
\[ \mathsf{D}[f](x, y) =  \lim \limits_{t \to 0} \frac{f(x + t \cdot y) - f(x)}{t} =  \lim \limits_{t \to 0} \frac{g_x(ty) - g_x(0)}{t} = \mathsf{L}[g_x](y)  \]
Therefore, the derivative of $f$ is the linearization of the function $g_x(y) = f(x + y)$ in the variable $y$. In other words, if we let $g(x,y) = f(x + y)$, then $\mathsf{D}[f]$ is the partial linearization of $g(x,y)$ in its second argument while keep the first argument constant.  We may write this directly as $\mathsf{D}[f](x, y) = \mathsf{L}[z \mapsto f(x+z)](y)$ where we are viewing $z \mapsto f(x+z)$ as a function in the variable context $x$. This is precisely how BJORT define their differential combinator. In fact, every differential combinator in a Cartesian differential category can be defined in this fashion. However, there is a caveat: in an arbitrary Cartesian left additive category, it is not always possible to define partial linearization from total linearization. Indeed, for example, $\mathcal{C}^1$ functions have a total linearization combinator but do not have partial linearization since this would induce a differential combinator, which cannot be the case since the derivative of a $\mathcal{C}^1$ function is not necessarily a  $\mathcal{C}^1$ function (see Example \ref{counter-example} below for more details). Thus, partial linearization, that is linearization in \emph{context}, must be assumed. 

From a categorical perspective, the notion of context is captured by simple slice categories \cite{jacobs1999categorical}, where a map $A \to B$ in the simple slice is a map of type $C \times A \to B$ in the base category. Maps in the simple slice category over an object $C$ are said to be in ``context $C$''. Asking that a Cartesian left additive category has partial linearization is requiring that it comes equipped with a \textbf{system of linearizing combinators} and is the requirement that every simple slice category come equipped with a linearizing combinator $\mathsf{L}^C$.  Thus, for a map $f: C \times A \to B$, ${\mathsf{L}^C[f]: C \times A \to B}$ is its linearization in context $C$, and these linearizing combinators are compatible with one another. For example, given a map of type $C \times A \to B$, we require that partially linearizing $A$ then $C$ is the same as partially linearizing $C$ then $A$. For the abelian functor calculus, BJORT's linearization of a multivariable functor at a single variable by holding all other inputs constant, $D^1_1$, is precisely a linearizing combinator in context.  For a Cartesian differential category, every simple slice category is again a Cartesian differential category where the differential combinator in context is given by partial differentiation. As such, every Cartesian differential category comes equipped with a canonical system of linearizing combinators. Conversely, to define a differential combinator from partial linearization, one must first be able to precompose by a map which captures addition. In a Cartesian left additive category, for every object $A$, there is a map $\oplus_A := \pi_0 + \pi_1: A \times A \to A$ which makes $A$ a commutative monoid. This allows the differential combinator $\mathsf{D}$ to be 
defined on a map by linearizing in context that map precomposed by $\oplus_A$, thus, generalizing the construction above. 
\[ \infer{\text{Linearize in the second argument} \qquad \mathsf{D}[f]:= \xymatrixcolsep{5pc}\xymatrix{A \times A \ar[r]^-{\mathsf{L}^A[\oplus_A f]} & B}}{\infer{ \text{Precompose by addition} \qquad \xymatrixcolsep{5pc}\xymatrix{A \times A \ar[r]^-{\oplus_A} & A \ar[r]^-{f} & B}}{ \xymatrixcolsep{5pc}\xymatrix{A \ar[r]^-{f} & B} }} \]
Furthermore, these constructions are inverses of each other, and so there is a bijective correspondence between differential combinators and systems of linearizing combinators. This shows that a Cartesian differential category is precisely a Cartesian left additive category with a system of linearizing combinators. 

To show how partial linearization arises from total linearization, we investigate linearization in Cartesian closed settings.  For Cartesian \emph{closed} left additive categories, we introduce the notion of an \textbf{exponentiable} linearizing combinator.   We then show how such a total linearizing combinator gives rise to a \textbf{closed} systems of linearizing combinators: that is a system of linearizing combinators, which are compatible with the closed structure.  To obtain a linearizing combinator in context, given a total exponentiable linearizing combinator, one employs the total linearization on the curry of the map and then one uncurries the result:

\[ \infer{\text{Uncurry} \qquad \mathsf{L}^C[f]:= \xymatrixcolsep{5pc}\xymatrix{C \times A \ar[r]^-{\lambda^{-1}\left( \mathsf{L}[\lambda(f)] \right)} & [C,A]} }{\infer{\text{Linearize} \qquad \xymatrixcolsep{5pc}\xymatrix{A \ar[r]^-{\mathsf{L}[\lambda(f)]} & [C,A]}}{\infer{ \text{Curry} \qquad \xymatrixcolsep{5pc}\xymatrix{A \ar[r]^-{\lambda(f)} & [C,A]}}{ \xymatrixcolsep{5pc}\xymatrix{C \times A \ar[r]^-{f} & B} }}} \]

\noindent \textbf{Outline:} Section \ref{CDCsec} is a background section which reviews the basic theory of Cartesian differential categories (Definition \ref{cartdiffdef}) and Cartesian left additive categories (Definition \ref{CLACdef}).  It also provides a list of the main examples of Cartesian differential categories used in this paper. The notion of linear maps (Definition \ref{linmapdef}) and their basic properties (Lemma \ref{linlem} and Lemma \ref{linlemimportant}) are reviewed. Section \ref{LINsec} introduces linearizing combinators (Definition \ref{lindef}) the main concept of study in this paper. In Proposition \ref{DLprop}, we show that every differential combinator induces a linearizing combinator, and afterwards we provide examples of these induced linearizing combinators in our main examples. Section \ref{ContextSec} reviews the notion of partial differentiation (Proposition \ref{Dcontext}) and being linear in context (Definition \ref{lin2def}). Section \ref{system-sec} discusses partial linearization by introducing systems of linearizing combinators (Definition \ref{syslindef}). In Proposition \ref{LDprop}, we show how every system of linearizing combinators induces a differential combinator -- following the BJORT construction. The first main result of this paper is Theorem \ref{DLthm} which says that there is a bijective correspondence between differential combinators and systems of linearizing combinators: thus, a Cartesian differential category is precisely a Cartesian left additive category with a system of linearizing combinators. We also provide an example of a linearizing combinator on a Cartesian left additive category which is not induced from a differential combinator or a system of linearizing combinators (Example \ref{counter-example}). Section \ref{closed-sec} studies how to define partial linearization from total linearization in the closed setting by introducing exponentiable linearizing combinators (Definition \ref{Lexpdef}) and closed systems of linearizing combinators (Definition \ref{closeddef}). In Proposition \ref{LDclosedprop}, we show that every closed system of linearizing combinators induces an exponentiable linearizing combinator, and conversely in Proposition \ref{LDexpprop}, we also show how every exponentiable linearizing combinator induces a closed system of linearizing combinators. Theorem \ref{finalthm} is the second main result of this paper, which states that a Cartesian closed differential category (Definition \ref{CDCcloseddef}) is precisely a Cartesian closed left additive category with a closed system of linearizing combinators, or equivalently an exponentiable linearizing combinator. We conclude with some final remarks in Section \ref{sec:conclusion}. \\

\noindent \textbf{Conventions:} We use diagrammatic order for composition: this means that the composite map $fg: A \to C$ is the map which first does $f: A\to B$ then $g: B \to C$. We denote identity maps simply as $1: A \to A$, thus, to simplify notation, we omit the subscript $\__A$.

\section{Cartesian Differential Categories}\label{CDCsec}

In this section, in order to fix notation, we briefly review Cartesian left additive categories, Cartesian differential categories, and linear maps. We also provide the examples of Cartesian differential categories which we will use throughout this paper. We assume that the reader is familiar with the basic theory of Cartesian differential categories: for a more in-depth introduction to Cartesian differential categories, we refer the reader to the original paper \cite{blute2009cartesian}. 

The underlying structure of a Cartesian differential category is that of a Cartesian left additive category. A category is said to be \emph{left} additive if it is \emph{skew}-enriched \cite{Campbell2018} over the category of commutative monoids. This allows one to have zero maps and sums of maps while allowing for maps which do not preserve the additive structure. Maps which do preserve the additive structure are called \emph{additive} maps. 

\begin{definition}\label{LACdef} A \textbf{left additive category} \cite[Definition 1.1.1]{blute2009cartesian} is a category $\mathbb{X}$ such that each hom-set $\mathbb{X}(A,B)$ is a commutative monoid with addition $+: \mathbb{X}(A,B) \times \mathbb{X}(A,B) \to \mathbb{X}(A,B)$, $(f,g) \mapsto f +g$, and zero $0 \in \mathbb{X}(A,B)$, such that pre-composition preserves the additive structure, that is, $f(g+h)=fg+fh$ and $f0=0$. Furthermore, we say that: 
\begin{enumerate}[{\em (i)}]
\item A map $f: A \to B$ is \textbf{constant} if $0f = f$;
\item A map $f: A \to B$ is \textbf{reduced} if $0f=0$;
\item A map $f: A \to B$ is \textbf{semi-additive} if $(g+h)f= gf + hf$;
\item A map $f: A \to B$ is \textbf{additive} if it is both reduced and semi-additive.
\end{enumerate}
\end{definition}

Next, we turn our attention to left additive categories with finite products. For a category with finite products we use $\times$ for the binary product, $\pi_0: A \times B \to A$ and $\pi_1: A \times B \to B$ for the projection maps, $\langle -, - \rangle$ for the pairing operation, so that $f \times g = \langle \pi_0 f, \pi_1 g \rangle$, and $\top$ for the chosen terminal object. Let $\tau_{A,B}: A \times B \to B \times A$ denote the canonical natural \emph{symmetry} isomorphism which is defined as follows: 
\begin{equation}\label{taudef}\begin{gathered} \tau_{A,B} = \langle \pi_1, \pi_0 \rangle
 \end{gathered}\end{equation}
We also denote the canonical natural \emph{interchange} isomorphism by
\[c_{A,B,C,D}: (A \times B) \times (C \times D) \to (A \times C) \times (B \times D)\]
which is defined as: 
  \begin{equation}\label{cdef}\begin{gathered} c_{A,B,C,D} := \langle \pi_0 \times \pi_0, \pi_1 \times \pi_1 \rangle
  \end{gathered}\end{equation}
To simplify notation, we will often omit the subscripts of $\tau$ and $c$. Note that both $\tau$ and $c$ are self-inverse, that is, $\tau \tau = 1$ and $cc=1$.  

\begin{definition}\label{CLACdef} A \textbf{Cartesian left additive category} \cite[Definition 2.3]{lemay2018tangent} is a left additive category $\mathbb{X}$ which has products for which all the projection maps $\pi_0: A \times B \to A$ and $\pi_1: A \times B \to B$ are additive. 
\end{definition}

The definition of a Cartesian left additive category presented here is not precisely that given in \cite[Definition 1.2.1]{blute2009cartesian}, but was shown to be equivalent in \cite[Lemma 2.4]{lemay2018tangent}. Also note that in a Cartesian left additive category, the unique map to the terminal object $\top$ is the zero map ${0: A \to \top}$.

In a Cartesian left additive category, define the \emph{lifting} map $\ell_{A,B,C,D}: A \times D \to (A \times B) \times (C \times D)$ as the map which inserts zeros in the middle two arguments, that is, define $\ell_{A,B,C,D}$ as follows: 
  \begin{equation}\label{ldef}\begin{gathered} \ell_{A,B,C,D} := \langle 1, 0 \rangle \times \langle 0,1 \rangle
  \end{gathered}\end{equation}
As before, to simplify notation, we will often omit the subscripts of $\ell$ when there is no confusion. It is important to note that in an arbitrary Cartesian left additive category, $\ell$ is \emph{not} a natural transformation. However, $\ell$ is natural whenever  $g$ and $h$ are reduced maps making $(f \times k) \ell = \ell \left( (f \times g) \times (h \times k) \right)$. The lifting map $\ell$ is a crucial ingredient in constructing differential combinators and linearizing combinators in \emph{context}, as will see in later sections. 

Cartesian left additive categories can be equivalently axiomatized by equipping each object with a commutative monoid structure so all the projection maps, $\pi_0$ and  $\pi_1$, are monoid morphisms. In this axiomatization of a Cartesian left additive category, the additive maps are precisely the monoid morphisms with respect to the canonical monoid structure. 
Here is how that monoid structure arises:

\begin{lemma}\label{opluslem0} \cite[Proposition 1.2.2, Lemma 1.2.3]{blute2009cartesian} In a Cartesian left additive category, for every object $A$ define the map $\oplus_A: A \times A \to A$ as $\oplus_A := \pi_0 + \pi_1$. Then: 
\begin{enumerate}[{\em (i)}]
\item \label{opluslem1} For every object $A$, $(A, \oplus_A, 0)$ is a commutative monoid, that is, the following equalities hold:
\begin{align*}
\langle 0,1 \rangle \oplus_A = 1 && \langle 1,0 \rangle \oplus_A = 1 && \tau \oplus_A = \oplus_A && c (\oplus_A \times \oplus_A) \oplus_A = (\oplus_A \times \oplus_A) \oplus_A\end{align*}
\item \label{opluslem2} For every pair of objects $A$ and $B$, the following equalities hold:
\begin{align*}
\oplus_{A \times B} = c (\oplus_A \times \oplus_B) && \ell \oplus_{A \times B} = 1&& \ell(\oplus_A \times \oplus_A) = 1\end{align*}
\item \label{opluslem3} A map $f: A \to B$ is additive if and only if $\oplus_A f = \pi_0 f + \pi_1 f$ and $0f = 0$ (or equivalently if $\oplus_A f = (f \times f) \oplus_B$ and $0f =0$).
\end{enumerate}
\end{lemma}

Cartesian differential categories are Cartesian left additive categories which come equipped with a differential combinator, which in turn is axiomatized by the basic properties of the directional derivative from multivariable differential calculus. In the following definition, note that unlike in the original paper \cite{blute2009cartesian} and other early works on Cartesian differential categories, we use the convention used in the more recent works where the vector argument of $\mathsf{D}[f]$ is its second argument rather than its first argument. There are various equivalent ways of expressing the axioms of a Cartesian differential category. For this paper, we've chosen the one found in \cite[Definition 2.6]{lemay2018tangent} (using the notation for Cartesian left additive categories introduced above).  

\begin{definition}\label{cartdiffdef} A \textbf{Cartesian differential category} \cite[Definition 2.1.1]{blute2009cartesian} is a Cartesian left additive category $\mathbb{X}$ equipped with a \textbf{differential combinator} $\mathsf{D}$, which is a family of operators ${\mathsf{D}: \mathbb{X}(A,B) \to \mathbb{X}(A \times A,B)}$, $f \mapsto \mathsf{D}[f]$, where $\mathsf{D}[f]$ is called the derivative of $f$, such that the following seven axioms hold:  
\begin{enumerate}[{\bf [CD.1]}]
\item $\mathsf{D}[f+g] = \mathsf{D}[f] + \mathsf{D}[g]$ and $\mathsf{D}[0]=0$;
\item $(1 \times \oplus_A) \mathsf{D}[f] = (1 \times \pi_0) \mathsf{D}[f] + (1 \times \pi_1)\mathsf{D}[f]$ and $\langle 1, 0 \rangle \mathsf{D}[f]=0$;
\item $\mathsf{D}[1]=\pi_1$, $\mathsf{D}[\pi_0] = \pi_1\pi_0$ and $\mathsf{D}[\pi_1] = \pi_1\pi_1$;
\item $\mathsf{D}[\langle f, g \rangle] = \langle \mathsf{D}[f] , \mathsf{D}[g] \rangle$; 
\item $\mathsf{D}[fg] = \langle \pi_0 f, \mathsf{D}[f] \rangle \mathsf{D}[g]$ (the chain rule); 
\item $\ell ~\mathsf{D}\!\left[\mathsf{D}[f] \right] = \mathsf{D}[f]$ where $\ell$ is defined as in (\ref{ldef}); 
\item $c~ \mathsf{D}\!\left[\mathsf{D}[f] \right]= \mathsf{D}\left[\mathsf{D}[f] \right]$ where $c$ is defined as in (\ref{cdef}). 
\end{enumerate}
\end{definition}

A discussion on the intuition for the differential combinator axioms can be found in \cite[Remark 2.1.3]{blute2009cartesian}.  Notice, in particular, that {\bf [CD.5]} is the chain rule for the directional derivative.

An important class of maps in a Cartesian differential category is the class of linear maps.

\begin{definition}\label{linmapdef} In a Cartesian differential category with differential combinator $\mathsf{D}$, a map $f$ is said to be \textbf{linear} \cite[Definition 2.2.1]{blute2009cartesian} if $\mathsf{D}[f]= \pi_1 f$. 
\end{definition}

When we need to emphasize the differential sense in which a map is linear we shall say that the map is $\mathsf{D}$-linear.

\begin{lemma}\label{linlem} \cite[Lemma 2.2.2]{blute2009cartesian} In a Cartesian differential category with differential combinator $\mathsf{D}$,
\begin{enumerate}[{\em (i)}]
\item \label{linlem.add} If $f$ is linear then $f$ is additive;
\item  \label{linlem.pre} If $f$ is linear then for every map $g$ which is post-composable with $f$, $\mathsf{D}[fg] = (f \times f) \mathsf{D}[g]$;
\item  \label{linlem.post} If $g$ is linear then for every map $f$ which is pre-composable with $g$, $\mathsf{D}[fg] = \mathsf{D}[f]g$. 
\item \label{linlem.1}  Identity maps are linear;
\item  \label{linlem.0} Zero maps are linear;
\item  \label{linlem.pi} Projection maps $\pi_0$ and $\pi_1$ are linear;
\item  \label{linlem.comp} If $f$ and $g$ are linear and composable, then their composition $fg$ is linear;
\item  \label{linlem.pair} If $f$ and $g$ are linear and pairable, then their pairing $\langle f, g \rangle$ is linear;
\item  \label{linlem.prod} If $f$ and $g$ are linear, then their product $f \times g$ is linear;
\item \label{linlem.sum}  If $f$ and $g$ are linear and summable, then their sum $f+g$ is linear;
\item  \label{linlem.retract} If $f$ is a retract and linear, and if for a map $g$ which is post-composable with $f$ their composite ${fg}$ is linear, then $g$ is linear; 
\item  \label{linlem.iso} If $f$ is linear and an isomorphism, then its inverse $f^{-1}$ is also linear. 
\end{enumerate}
\end{lemma}

It follows that the linear maps form a subcategory with finite bipoducts \cite[Corollary 2.2.3]{blute2009cartesian}. Although additive and linear maps often coincide in the examples, it is important to recall that, in general, while every linear map is additive, not every additive map is necessarily linear. 

 \begin{corollary}\label{corlin} In a Cartesian differential category:
 \begin{enumerate}[{\em (i)}]
\item \label{corlin.tau} The symmetry isomorphism $\tau: A \times B \to B \times A$ is linear; 
\item \label{corlin.c} The interchange isomorphism $c: (A \times B) \times (C \times D) \to (A \times C) \times (B \times D)$ is linear;
\item \label{corlin.ell} The lifting map $\ell: A \times D \to (A \times B) \times (C \times D)$ is linear; 
\item \label{corlin.oplus} The sum map $\oplus_A: A \times A \to A$ is linear. 
\end{enumerate}
 \end{corollary}
 
 A key observation for this paper is that $f$ is linear if and only if $\langle 0,1 \rangle \mathsf{D}[f]$ is linear.  We shall use this fact to construct the \emph{linearizing combinator} of a Cartesian differential category (see Proposition \ref{DLprop}).  This observation 
 was proven in \cite[Corollary 2.2.3]{blute2009cartesian} and we repeat it here for completeness:
 
 \begin{lemma}   \label{linlemimportant} 
 In a Cartesian differential category, 
  \begin{enumerate}[{\em (i)}]
\item\label{linlemimportant1} For any map $f$, $\langle 0,1 \rangle \mathsf{D}[f]$ is linear. 
\item\label{linlemimportant2} $f$ is linear if and only if $f = \langle 0,1 \rangle \mathsf{D}[f]$.
\end{enumerate}
 \end{lemma}
 \begin{proof} For (\ref{linlemimportant1}), we must show that $\mathsf{D}\left[ \langle 0,1 \rangle \mathsf{D}[f] \right] = \pi_1 \langle 0,1 \rangle \mathsf{D}[f]$. First note that by Lemma \ref{linlem}.(\ref{linlem.1}), (\ref{linlem.0}), and (\ref{linlem.pair}) it follows that $\langle 0,1 \rangle$ is linear. Therefore, we compute that: 
   \begin{align*}
\mathsf{D} \left[ \langle 0,1 \rangle \mathsf{D}[f] \right] &=~ (\langle 0,1 \rangle \times \langle 0,1 \rangle) \mathsf{D}[\mathsf{D}[f]] \tag{$\langle 0,1 \rangle$ is linear + Lem.\ref{linlem}.(\ref{linlem.pre})} \\
 &=~ \langle \langle 0,\pi_0 \rangle, \langle 0,\pi_1 \rangle \rangle \mathsf{D}[\mathsf{D}[f]] \\ 
 &=~ \langle \langle 0,0 \rangle, \langle \pi_0,\pi_1 \rangle \rangle \mathsf{D}[\mathsf{D}[f]] \tag*{\textbf{[CD.7]}} \\
 &=~ \langle \langle 0,0 \rangle, \langle \pi_0,0 \rangle + \langle 0,\pi_1 \rangle \rangle \mathsf{D}[\mathsf{D}[f]]  \\
 &=~ \langle \langle 0,0 \rangle, \langle \pi_0,0 \rangle\rangle \mathsf{D}[\mathsf{D}[f]] +  \langle  \langle 0,0 \rangle,  \langle 0,\pi_1 \rangle \rangle  \mathsf{D}[\mathsf{D}[f]] \tag*{\textbf{[CD.2]}} \\
 &=~  \langle \langle 0,\pi_0 \rangle, \langle 0,0 \rangle\rangle \mathsf{D}[\mathsf{D}[f]] + \langle  0,\pi_1 \rangle \mathsf{D}[f] \tag*{\textbf{[CD.7]} + \textbf{[CD.6]}} \\
  &=~  \langle \langle 0,\pi_0 \rangle, 0 \rangle \mathsf{D}[\mathsf{D}[f]] + \langle  0,\pi_1 \rangle \mathsf{D}[f] \\
 &=~ 0 + \langle  0,\pi_1 \rangle \mathsf{D}[f] \tag*{\textbf{[CD.2]}} \\
 &=~  \pi_1 \langle 0,1 \rangle \textsf{D}[f]
 \end{align*}
 So we conclude that $\langle 0,1 \rangle \mathsf{D}[f]$ is linear. Now suppose that $f$ is linear, then we compute: 
  \begin{align*}
\langle 0,1 \rangle \textsf{D}[f] &=~ \langle 0,1 \rangle \pi_1 f \tag{$f$ is linear}\\
&=~ f
 \end{align*}
 So $f = \langle 0,1 \rangle \mathsf{D}[f]$. Conversely, suppose that $f = \langle 0,1 \rangle \mathsf{D}[f]$. By (\ref{linlemimportant1}), $\langle 0,1 \rangle \mathsf{D}[f]$ is linear and so $f$ is also linear. 
 \end{proof}

We conclude this section by providing examples of Cartesian differential categories. The canonical example of a Cartesian differential category is the category of real smooth functions. The main motivating example for this paper is, however, the abelian functor calculus model of \cite{bauer2018directional}. Many other interesting examples of Cartesian differential categories can be found throughout the literature such as smooth functions, polynomials, any category with finite biproducts, cofree Cartesian differential categories \cite{cockett2011faa,lemay2018tangent}, the coKleisli category of a differential category \cite{Blute2019,blute2006differential}, convenient vector spaces \cite{blute2010convenient,kriegl1997convenient,manzyuk2012tangent}, and categorical models of the differential $\lambda$-calculus \cite{bucciarelli2010categorical,Cockett-2019,EHRHARD20031,manzonetto_2012}. 

\begin{example}\label{ex:biproduct} \normalfont Every category with finite biproducts is a Cartesian differential category where the differential combinator is defined as:
\[\mathsf{D}[f] = \pi_1 f\]
In this case, every map is linear by definition. The converse is also true: a Cartesian differential category where every map is linear is precisely a category with finite biproducts. 
\end{example}

\begin{example}\label{ex:smooth} \normalfont  Let $\mathbb{R}$ be the set of real numbers. Define $\mathsf{SMOOTH}$ as the category whose objects are the Euclidean real vector spaces $\mathbb{R}^n$ (including the singleton $\mathbb{R}^0 = \lbrace \top \rbrace$) and whose maps are the real smooth functions ${F: \mathbb{R}^n \to \mathbb{R}^m}$ between them. $\mathsf{SMOOTH}$ is a Cartesian differential category where the differential combinator is defined as the directional derivative of a smooth function. Recall that a smooth function $F: \mathbb{R}^n \to \mathbb{R}^m$ is in fact a tuple $F = \langle f_1, \hdots, f_m \rangle$ of smooth functions $f_i: \mathbb{R}^n \to \mathbb{R}$. Then the Jacobian matrix of $F$ at vector $\vec x \in \mathbb{R}^n$ is the matrix $\nabla(F)(\vec x)$ of size $m \times n$ whose coordinates are the partial derivatives of the $f_i$: 
\[ \nabla(F)(\vec x) := \begin{bmatrix} \frac{\partial f_1}{\partial x_1}(\vec x) & \frac{\partial f_1}{\partial x_2}(\vec x) & \hdots & \frac{\partial f_1}{\partial x_n}(\vec x) \\
 \frac{\partial f_2}{\partial x_1}(\vec x) & \frac{\partial f_2}{\partial x_2}(\vec x) & \hdots & \frac{\partial f_2}{\partial x_n}(\vec x)  \\
 \vdots & \vdots & \vdots & \vdots \\
  \frac{\partial f_m}{\partial x_1}(\vec x) & \frac{\partial f_m}{\partial x_2}(\vec x) & \hdots & \frac{\partial f_m}{\partial x_n}(\vec x) 
\end{bmatrix} \]
So for a smooth function $F: \mathbb{R}^n \to \mathbb{R}^m$, its derivative $\mathsf{D}[F]: \mathbb{R}^n \times \mathbb{R}^n \to \mathbb{R}^m$ is then defined as:
\[\mathsf{D}[F](\vec x, \vec y) := \nabla(F)(\vec x) \cdot \vec y = \left \langle \sum \limits^n_{i=1} \frac{\partial f_1}{\partial x_i}(\vec x) y_i, \hdots, \sum \limits^n_{i=1} \frac{\partial f_n}{\partial x_i}(\vec x) y_i \right \rangle\]
where $\cdot$ is matrix multiplication and $\vec y$ is seen as a $n \times 1$ matrix. 
%For smooth functions of type $f: \mathbb{R}^n \to \mathbb{R}$, its derivative $\mathsf{D}[f]: \mathbb{R}^n \times \mathbb{R}^n \to \mathbb{R}$ is simply the sum of its partial derivatives: 
%\[\mathsf{D}[f](\vec x, \vec y) = \sum \limits^n_{i=1} \frac{\partial f}{\partial x_i}(\vec x) y_i \]
A smooth function $F: \mathbb{R}^n \to \mathbb{R}^m$ is linear in the Cartesian differential sense precisely when it is $\mathbb{R}$-linear in the classical sense, that is, $F(s \vec x + t \vec y) = sF(\vec x) + t F(\vec y)$ for all $s,t \in \mathbb{R}$ and $\vec x, \vec y \in \mathbb{R}^n$. 
\end{example}

\begin{example} \normalfont We very briefly review the abelian functor calculus model of a Cartesian differential category: for more complete details on this example see \cite{bauer2018directional}. Let $\mathbb{A}$ be an abelian category and let $\mathsf{Ch}(\mathbb{A})$ be its category of (non-negative) chain complexes. Define $\mathsf{HoAbCat}_\mathsf{Ch}$ as the category whose objects are abelian categories where a map from $\mathbb{A} \to \mathbb{B}$ is a point-wise chain homotopy equivalence class of functors $\mathbb{A} \to \mathsf{Ch}(\mathbb{B})$, and where composition and identity maps are defined as in \cite[Definition 3.5]{bauer2018directional}. By \cite[Corollary 6.6]{bauer2018directional}, $\mathsf{HoAbCat}_\mathsf{Ch}$ is a Cartesian differential category where the differential combinator, which in this case is written as $\nabla$, is defined for $F: \mathbb{A} \to \mathsf{Ch}(\mathbb{B})$ as follows on objects:
\[ \nabla F (X,V) := D_1 F(X \oplus -)(V)  \]
where $D_1$ is the linearization operator as defined in \cite[Definition 5.1]{bauer2018directional} using cross effects of functors. In this case, a functor $F$ is linear in the Cartesian differential sense if it is linear in the abelian functor calculus sense, that is, if $F$ preserves finite direct sums up to chain homotopy equivalence \cite[Definition 5.5]{bauer2018directional}. 
\end{example}

\begin{example} \normalfont Every Cartesian left additive category has a cofree Cartesian differential category over it which satisfies the obvious couniversal property. Cofree Cartesian differential categories were first constructed in \cite{cockett2011faa} using the Fa\`a di Bruno construction. In this paper, we will use the alternative construction found in \cite{lemay2018tangent}, as the differential combinator is simpler to express. For a Cartesian left additive category $\mathbb{X}$, let $\mathsf{P}: \mathbb{X} \to \mathbb{X}$ be the product functor defined on objects as $\mathsf{P}(A) = A \times A$ and on maps as $\mathsf{P}(f) = f \times f$. Then define $\mathcal{D}(\mathbb{X})$ as the category whose objects are the same as $\mathbb{X}$ and where a map $A \to B$ is a $\mathsf{D}$-sequence which is a sequence of maps $(f_0, f_1, \hdots)$ where $f_n: \mathsf{P}^n(A) \to B$ and satisfying the coherences found in \cite[Definition 4.2]{lemay2018tangent}. Composition and identity maps are defined as in \cite[Definition 3.6]{lemay2018tangent}. $\mathcal{D}(\mathbb{X})$ is a Cartesian differential category \cite[Corollary 4.25]{lemay2018tangent} where the differential combinator is defined by shifting $\mathsf{D}$-sequences to the left:
\[ \mathsf{D}[(f_0, f_1, \hdots)] = (f_1, f_2, \hdots) \]
A $\mathsf{D}$-sequence $(f_0, f_1, \hdots)$ is linear if and only if $f_n = \underbrace{\pi_1 \hdots \pi_1}_{n-times} f_0$ for all $n$ \cite[Lemma 4.26]{lemay2018tangent}. 
\end{example}

\begin{example} \normalfont An important source of examples of Cartesian differential categories are the coKleisli categories of differential categories. For a more on differential categories, see \cite{Blute2019,blute2006differential}. A differential category \cite[Definition 2.4]{blute2006differential} is an additive symmetric monoidal category $\mathbb{X}$ equipped with a comonad $(\oc, \delta_A: \oc A \to \oc \oc A, \varepsilon_A: \oc A \to A)$, two natural transformations $\Delta_A: \oc A \to \oc A \otimes \oc A$ and $e_A: \oc A \to I$ such that $\oc A$ is a cocommutative comonoid, and a natural transformation called a deriving transformation $\mathsf{d}_A: \oc A \otimes A \to \oc A$  satisfying certain coherences which capture the basic properties of differentiation \cite[Definition 7]{Blute2019}. Examples of differential categories can be found in \cite[Section 9]{Blute2019}. When a differential category $\mathbb{X}$ has finite products, define the natural transformation $\chi_{A,B}: \oc(A \times B) \to \oc A \otimes \oc B$ as follows: 
\[\chi_{A,B} := \xymatrixcolsep{5pc}\xymatrix{\oc(A \times B) \ar[r]^-{\Delta_{A \times B}} & \oc(A \times B) \otimes \oc(A \times B) \ar[r]^-{\oc(\pi_0) \otimes \oc(\pi_1)} &\oc A \otimes \oc B
 } \]
By \cite[Proposition 3.2.1]{blute2009cartesian}, for a differential category $\mathbb{X}$ with finite products, its coKleisli category $\mathbb{X}_\oc$ is a Cartesian differential category where the differential combinator is defined using the deriving transformation. For a coKlesili map $f: \oc A \to B$, its derivative $\mathsf{D}[f]: \oc(A \times A) \to B$ is defined as: 
\[ \mathsf{D}[f] := \xymatrixcolsep{4pc}\xymatrix{\oc(A \times A) \ar[r]^-{\chi_{A,A}} & \oc A \otimes \oc A \ar[r]^-{1 \otimes \varepsilon_A} & \oc A \otimes A \ar[r]^-{\mathsf{d}_A} & \oc A \ar[r]^-{f} & B 
 } \]
Applying Lemma \ref{linlemimportant}.(\ref{linlemimportant2}) and being careful with coKleisli composition, one can show that a coKleisli map $f: \oc A \to B$ is linear if and only if $\Delta_A (\oc(0) \otimes \varepsilon_A) \mathsf{d}_A f = f$. In particular, for every map ${g: A \to B}$ in $\mathbb{X}$, $\varepsilon_A g: \oc A \to B$ is a linear map in the coKleisli category $\mathbb{X}_\oc$. 
\end{example}

\begin{example}\label{ex:diffstor} \normalfont A differential storage category \cite[Definition 10]{Blute2019} is a differential category with finite products such that $\chi_{A,B}$ and $e_\top$ are natural isomorphisms, called the Seely isomorphisms, so that $\oc(A \times B) \cong \oc A \otimes \oc B$ and $\oc \top \cong I$. In this case, the Seely isomorphisms induce two extra natural transformations $\nabla_A: \oc A \otimes \oc A \to \oc A$ and $u_A: I \to \oc A$ which make $\oc A$ into a bialgebra. Furthermore, the differential structure can equivalently be axiomatized in terms of a natural transformation ${\eta_A: A \to \oc A}$ called a codereliction \cite[Definition 9]{Blute2019}, that is, there is a bijective correspondence between coderelictions and deriving transformations \cite[Theorem 4]{Blute2019}. Given a codereliction $\eta$, one defines a deriving transformation $\mathsf{d}$ as follows: 
\[ \mathsf{d} := \xymatrixcolsep{5pc}\xymatrix{ \oc A \otimes A \ar[r]^-{1 \otimes \eta_A} & \oc A \otimes \oc A \ar[r]^-{\nabla_A} & \oc A 
 }\]
 and conversely, given deriving transformation $\mathsf{d}$, one defines a codereliction $\eta$ as follows: 
\[ \eta := \xymatrixcolsep{5pc}\xymatrix{ A \ar[r]^-{u_A \otimes 1} & \oc A \otimes A \ar[r]^-{\mathsf{d}_A} & \oc A 
 } \]
 and these constructions are inverses of each other. As such, for a coKlesili map $f: \oc A \to B$, its derivative $\mathsf{D}[f]: \oc(A \times A) \to B$ could also be expressed as follows:
 \[ \mathsf{D}[f] := \xymatrixcolsep{3pc}\xymatrix{\oc(A \times A) \ar[r]^-{\chi_{A,A}} & \oc A \otimes \oc A \ar[r]^-{1 \otimes \varepsilon_A} & \oc A \otimes A \ar[r]^-{1 \otimes \eta_A} & \oc A \otimes \oc A \ar[r]^-{\nabla_A} & \oc A  \ar[r]^-{f} & B 
 } \]
For a differential storage category $\mathbb{X}$, the linear maps in the coKleisli category $\mathbb{X}_\oc$ are precisely those of the form $\varepsilon_A g: \oc A \to B$ for a map $g: A \to B$ in $\mathbb{X}$. 
\end{example}

\begin{example} \normalfont The category of convenient vector spaces and smooth functions between them is an example of a coKleisli category of a differential storage category \cite{blute2010convenient, manzyuk2012tangent}. For a detailed introduction to convenient vector spaces, see \cite{kriegl1997convenient}. Briefly, recall that a locally convex space $E$ is a topological $\mathbb{R}$-vector space which is Hausdorff and such that $0$ has a neighbourhood basis of convex sets, and therefore we have a notion of converging limits. A curve of $E$ is a function $\phi: \mathbb{R} \to E$ and we say that a curve $\phi$ is differentiable if the limit:
\[\psi(x):= \lim \limits_{t \to 0} \frac{\phi(x+t) - \phi(x)}{t}\]
exists for all $x \in E$, and this defines a curve $\psi: \mathbb{R} \to E$ which is called the derivative of $\phi$. A curve $\phi$ is smooth if all its iterated derivatives exists, i.e, if it is infinitely differentiable. A convenient vector space \cite[Theorem 2.14]{kriegl1997convenient} is a locally convex space $E$ such that for every smooth curve $\phi$ there exists a smooth curve $\tilde{\phi}$ such that $\tilde{\phi}^\prime = \phi$. Alternatively, a convenient vector space is a locally convex vector space which $c^\infty$-complete (which is called Mackey complete in \cite[Definition 3.15]{blute2010convenient}) If $E$ and $F$ are both convenient vector spaces, then a smooth function $f: E \to F$ is a function $f$  which preserves smooth curves, that is, if $\phi$ is a smooth curve of $E$ then $\phi f$ is a smooth curve of $F$. Let $\mathsf{CON}$ be the category of convenient vector spaces and smooth functions between them. By \cite[Theorem 6.3]{blute2010convenient}, $\mathsf{CON}$ is isomorphic to the coKleisli category of a comonad on $\mathsf{CON}_{lin}$, the category of convenient vector spaces and smooth \emph{linear} functions (i.e. smooth functions which are also $\mathbb{R}$-linear). Furthermore, $\mathsf{CON}_{lin}$ is a differential storage category \cite[Theorem 6.6]{blute2010convenient} and therefore $\mathsf{CON}$ is a Cartesian differential category (see \cite[Example 2.4.2]{manzyuk2012tangent} for full details). For a smooth function ${f: E \to F}$, its derivative $\mathsf{D}[f]: E \times E \to F$ is defined as follows: 
\[ \mathsf{D}[f](x, y) := \lim \limits_{t \to 0} \frac{f(x + t \cdot y) - f(x)}{t} \]
where $t \in \mathbb{R}$ and $\cdot$ is scalar multiplication. In $\mathsf{CON}$, a smooth function is linear in the Cartesian differential sense precisely when it is a (smooth) linear function in the ordinary sense of linear algebra.  Lastly, note that $\mathsf{SMOOTH}$ is a Cartesian differential subcategory of $\mathsf{CON}$. 
\end{example}

\section{Linearizing Combinators}\label{LINsec}

In this section we introduce the notion of a \emph{linearizing combinator} for a Cartesian left additive category. Linearizing combinators are generalizations of the linearization operation used in the abelian functor calculus \cite[Definition 5.1]{bauer2018directional}. In fact, we will show that every Cartesian differential category comes equipped with a canonical linearizing combinator. The basic idea is that a linearizing combinator produces the linear approximation of maps. 

\begin{definition}\label{lindef} A \textbf{linearizing combinator} $\mathsf{L}$ on a Cartesian left additive category $\mathbb{X}$ is a family of operators, for each $A,B \in \mathbb{X}$
\[ {\mathsf{L}: \mathbb{X}(A,B) \to \mathbb{X}(A,B)}; f \mapsto \mathsf{L}[f] \]
such that the following six axioms hold:
\begin{enumerate}[{\bf [L.1]}]
\item $\mathsf{L}[f+g]=\mathsf{L}[f]+\mathsf{L}[g]$ and $\mathsf{L}[0]=0$
\item $\mathsf{L}[f]$ is additive, or equivalently by Lemma \ref{opluslem0}.(\ref{opluslem3}), $\oplus_A \mathsf{L}[f] = \pi_0 \mathsf{L}[f]+ \pi_1 \mathsf{L}[f]$ and $0 \mathsf{L}[f]=0$; 
\item $\mathsf{L}[1]=1$, $\mathsf{L}[\pi_0]=\pi_0$, and $\mathsf{L}[\pi_1]=\pi_1$
\item $\mathsf{L}\left[ \langle f, g \rangle \right] = \left \langle \mathsf{L}[f], \mathsf{L}[g] \right \rangle $
\item $\mathsf{L}[fg] = \mathsf{L}[f]~\mathsf{L}\left[(1+0f)g\right]$
\item $\mathsf{L}\left[\mathsf{L}[f]\right] = \mathsf{L}[f]$
\end{enumerate}
The expression $\mathsf{L}[f]$ is called the \textbf{linearization} of $f$. 
\end{definition}

The basic intuition of a linearizing combinator $\mathsf{L}$ is that from an arbitrary map $f$, $\mathsf{L}$ produces a linear map $\mathsf{L}[f]$. Examples of linearizing combinators can be found at the end of this section. The motivating example of a linearizing combinator is the linearization operator from abelian functor calculus \cite[Definition 5.1]{bauer2018directional}. The main source of examples of linearizing combinators come from Cartesian differential categories, as we will see in Proposition \ref{DLprop} below, where the linearizing combinator is defined as the differential combinator evaluated at zero in the first argument. Indeed as explained in Lemma \ref{linlemimportant}(\ref{linlemimportant1}), for every map $f$, the composite $\langle 0,1 \rangle \mathsf{D}[f]$ is a linear map. As a simple example, let $f: \mathbb{R} \to \mathbb{R}$ be a smooth function, then $\mathsf{L}[f]: \mathbb{R} \to \mathbb{R}$ is the $\mathbb{R}$-linear map defined as the degree 1 term of the Taylor expansion of $f$, that is, $\mathsf{L}[f](x) = f^\prime(0)x$. 

The axioms of a linearizing combinator are analogues of the first six axioms of a differential combinator. {\bf [L.1]} says that the linearization of a sum of maps is equal to the sum of linearization of maps. {\bf [L.2]} says the linearization of a map is additive. {\bf [L.3]} tells us that identity maps and projection maps are already linearized. {\bf [L.4]} says that the linearization of a pairing of maps is same as the pairing of the linearization of maps. {\bf [L.5]} tells how to linearize a composite of maps, ie., the chain rule for linearization. And lastly, {\bf [L.6]} says that the linearizing combinator is idempotent, that is, since a linearization of a map is already linearized, to apply the linearization combinator twice is the same as doing it once. These axioms can also be found throughout \cite{bauer2018directional}. Indeed, \textbf{[L.1]} is \cite[Lemma 5.6.ii]{bauer2018directional}, \textbf{[L.2]} is \cite[Lemma 5.6.i]{bauer2018directional}, \textbf{[L.3]} is \cite[Lemma 5.16]{bauer2018directional}, \textbf{[L.4]} is \cite[Lemma 5.18]{bauer2018directional}, and \textbf{[L.5]} is a generalization of \cite[Propostion 5.10]{bauer2018directional}. 

The keen eyed reader may have noticed that on the right hand side of {\bf [L.5]}, $\mathsf{L}\left[(1+0f)g\right]$ is a linearization of a composite of maps. In theory one could again apply {\bf [L.5]} to $\mathsf{L}\left[(1+0f)g\right]$. However, the following calculation shows us that doing so does not result in any simplification: 
\begin{align*}
\mathsf{L}\left[(1+0f)g\right] &= \mathsf{L}[1 + 0f] \mathsf{L}\left[\left(1+0(1+0f) \right)g\right] \tag*{\bf [L.5]} \\
&= \mathsf{L}[1 + 0f] \mathsf{L}\left[\left(1+0+0f \right)g\right] \\
&= \mathsf{L}[1 + 0f] \mathsf{L}\left[(1+0f)g\right] \\
&= \left( \mathsf{L}[1] + \mathsf{L}[0f] \right) \mathsf{L}\left[(1+0f)g\right] \tag*{\bf [L.1]} \\ 
&=  \left( 1 + \mathsf{L}[0]\mathsf{L}\left[(1+0)f\right] \right) \mathsf{L}\left[(1+0f)g\right] \tag*{\bf [L.3] + \bf [L.5]} \\
&=  \left( 1 + 0\mathsf{L}\left[ f\right] \right) \mathsf{L}\left[(1+0f)g\right] \tag*{\bf [L.1]} \\
&=  \left( 1 + 0 \right) \mathsf{L}\left[(1+0f)g\right] \tag*{\bf [L.2]} \\
&= \mathsf{L}\left[(1+0f)g\right] 
\end{align*}
So {\bf [L.5]} is indeed simplified as far as possible. That said, {\bf [L.5]} does simplify when the maps are either reduced, semi-additive, or additive. 

\begin{lemma}\label{linprop1} Let $\mathsf{L}$ be a linearizing combinator on a Cartesian left additive category. 
\begin{enumerate}[{\em (i)}]
\item \label{linprop1.i} If $f: A \to B$ is constant then $\mathsf{L}[f]= 0$. 
\item \label{linprop1.ii} For a reduced map $f$ and any map $g$, $\mathsf{L}[fg]= \mathsf{L}[f]~\mathsf{L}[g]$. 
\item \label{linprop1.iii} For a semi-additive map $g$ and any map $f$, $\mathsf{L}[fg]= \mathsf{L}[f]~\mathsf{L}[g]$.
\end{enumerate}
\end{lemma}
\begin{proof} These are mostly straightforward calculations. 
\begin{enumerate}[{\em (i)}]
\item Suppose that $f$ is constant, that is, $0f=f$. Then we have that: 
\begin{align*}
\mathsf{L}[f] &=~ \mathsf{L}[0f] \tag{$f$ constant}\\
&=~ \mathsf{L}[0]~\mathsf{L}\left[(1+00)f\right]  \tag*{\textbf{[L.5]}}\\
&=~ 0 \mathsf{L}\left[f\right] \tag*{\textbf{[L.1]}} \\
&=~ 0 \tag*{\textbf{[L.2]}}
\end{align*}
\item Suppose that $f$ is reduced, that is, $0f=0$. Then we have that: 
\begin{align*}
\mathsf{L}[fg] &=~  \mathsf{L}[f]~\mathsf{L}\left[(1+0f)g\right] \tag*{\textbf{[L.5]}}\\ 
&=~ \mathsf{L}[f]~\mathsf{L}\left[(1+0)g\right] \tag{$f$ reduced} \\
&=~  \mathsf{L}[f]~\mathsf{L}[g]
\end{align*}
\item Suppose that $g$ is semi-additive, that is, $(f+k)g=fg+kg$. Then we have that: 
\begin{align*}
\mathsf{L}[fg] &=~  \mathsf{L}[f]~\mathsf{L}\left[(1+0f)g\right] \tag*{\textbf{[L.5]}}\\ 
&=~ \mathsf{L}[f]~\mathsf{L}\left[g + 0fg\right] \tag{$g$ semi-additive} \\
&=~  \mathsf{L}[f]~\left(\mathsf{L}[g] + \mathsf{L}\left[0fg\right] \right) \tag*{\textbf{[L.1]}} \\
&=~  \mathsf{L}[f]~\left(\mathsf{L}[g] + \mathsf{L}[0] \mathsf{L}\left[(1+0)fg\right]  \right)  \tag*{\textbf{[L.5]}}\\ 
&=~  \mathsf{L}[f]~\left(\mathsf{L}[g] + 0 \mathsf{L}\left[fg\right]  \right)  \tag*{\textbf{[L.1]}}\\ 
&=~  \mathsf{L}[f]~\left(\mathsf{L}[g] + 0 \right)  \tag*{\textbf{[L.2]}}\\ 
&=~ \mathsf{L}[f]~\mathsf{L}[g]
\end{align*}
\end{enumerate}
\end{proof} 

For a linearizing combinator, the analogues of linear maps are the maps for which the linearizing combinator does nothing, that is, $\mathsf{L}[f]=f$. 

\begin{definition} In a Cartesian left additive category with a linearizing combinator $\mathsf{L}$, a map $f$ is said to be \textbf{$\mathsf{L}$-linear} if $\mathsf{L}[f]=f$. 
\end{definition}

As we will see in Proposition \ref{DLprop}, in a Cartesian differential category the $\mathsf{L}$-linear maps are precisely the linear maps. As such, $\mathsf{L}$-linear maps satisfy many of the same basic properties as linear maps. 

\begin{lemma}\label{lstable-lem} In a Cartesian left additive category with a linearizing combinator $\mathsf{L}$, 
\begin{enumerate}[{\em (i)}]
\item \label{lstable-lem.linear} For every map $f$, $\mathsf{L}[f]$ is $\mathsf{L}$-linear; 
\item \label{lstable-lem.add} If $f$ is $\mathsf{L}$-linear then $f$ is additive;
\item \label{lstable-lem.pre} If $f$ is $\mathsf{L}$-linear then for every map $g$ which is post-composable with $f$, $\mathsf{L}[fg] = f \mathsf{L}[g]$;
\item \label{lstable-lem.post} If $g$ is $\mathsf{L}$-linear then for every map $f$ which is pre-composable with $g$, $\mathsf{L}[fg] = \mathsf{L}[f]g$. 
\item \label{lstable-lem.1} Identity maps are $\mathsf{L}$-linear;
\item \label{lstable-lem.0} Zero maps are $\mathsf{L}$-linear;
\item \label{lstable-lem.pi} Projection maps $\pi_0$ and $\pi_1$ are $\mathsf{L}$-linear;
\item \label{lstable-lem.comp} If $f$ and $g$ are $\mathsf{L}$-linear and composable, then their composition $fg$ is $\mathsf{L}$-linear;
\item \label{lstable-lem.pair} If $f$ and $g$ are $\mathsf{L}$-linear and pairable, then their pairing $\langle f, g \rangle$ is $\mathsf{L}$-linear;
\item \label{lstable-lem.prod} If $f$ and $g$ are $\mathsf{L}$-linear, then their product $f \times g$ is $\mathsf{L}$-linear;
\item \label{lstable-lem.sum} If $f$ and $g$ are $\mathsf{L}$-linear and summable, then their sum $f+g: A \to B$ is $\mathsf{L}$-linear;
\item \label{lstable-lem.retract} If $f$ is a retract and $\mathsf{L}$-linear, and if for a map $g$ which is post-composable with $f$ their composite ${fg}$ is $\mathsf{L}$-linear, then $g$ is $\mathsf{L}$-linear. 
\item \label{lstable-lem.iso} If $f$ is $\mathsf{L}$-linear and an isomorphism, then its inverse $f^{-1}$ is also $\mathsf{L}$-linear. 
\end{enumerate}
\end{lemma}
\begin{proof} Most of these follow directly from the axioms of a linearizing combinator. (\ref{lstable-lem.linear}) follows from \textbf{[L.6]}, (\ref{lstable-lem.add}) follows from \textbf{[L.2]}, (\ref{lstable-lem.1}) and (\ref{lstable-lem.pi}) follow from \textbf{[L.3]}, (\ref{lstable-lem.0}) and (\ref{lstable-lem.sum}) follow from \textbf{[L.1]}, (\ref{lstable-lem.pair}) follows from \textbf{[L.4]}. For the rest, we mostly use \textbf{[L.5]} and Lemma \ref{linprop1}. \\ \\
\noindent (\ref{lstable-lem.pre}): Suppose that $f$ is $\mathsf{L}$-linear. By (\ref{lstable-lem.add}), $f$ is additive and therefore reduced. Then: 
\begin{align*} \mathsf{L}[fg] &=~\mathsf{L}[f]~\mathsf{L}[g] \tag{Lemma \ref{linprop1}.(\ref{linprop1.ii})} \\
&=~ f \mathsf{L}[g] \tag{$f$ is $\mathsf{L}$-linear}
\end{align*}
\noindent (\ref{lstable-lem.post}): Suppose that $g$ is $\mathsf{L}$-linear. By (\ref{lstable-lem.add}), $g$ is additive and therefore semi-additive. Then: 
\begin{align*} \mathsf{L}[fg] &=~\mathsf{L}[f]~\mathsf{L}[g] \tag{Lemma \ref{linprop1}.(\ref{linprop1.iii})} \\
&=~ \mathsf{L}[f] g \tag{$g$ is $\mathsf{L}$-linear}
\end{align*}
\noindent (\ref{lstable-lem.comp}): Suppose that $f$ and $g$ are $\mathsf{L}$-linear and composable. Then we have that:
\begin{align*} \mathsf{L}[fg] &=~f \mathsf{L}[g] \tag{$f$ is $\mathsf{L}$-linear + (\ref{lstable-lem.pre})} \\
&=~ f g \tag{$g$ is $\mathsf{L}$-linear}
\end{align*}
So $fg$ is $\mathsf{L}$-linear. \\ \\
\noindent (\ref{lstable-lem.prod}):  Suppose that $f$ and $g$ are $\mathsf{L}$-linear. By (\ref{lstable-lem.pi}) and (\ref{lstable-lem.comp}), $\pi_0f$ and $\pi_1 g$ are also $\mathsf{L}$-linear. Then by (\ref{lstable-lem.pair}), the pairing of $\pi_0f$ and $\pi_1 g$ is also $\mathsf{L}$-linear, that is, $f \times g = \langle \pi_0 f, \pi_1 g \rangle$ is $\mathsf{L}$-linear. \\ \\
\noindent (\ref{lstable-lem.retract}): Suppose that $f$ is a retract (with section $f^\circ$) and $\mathsf{L}$-linear, and that ${fg}$ is $\mathsf{L}$-linear. Then: 
\begin{align*} \mathsf{L}[g] &=~ f^\circ f\mathsf{L}[g] \tag{$f$ is a retract of $f^\circ$} \\
&=~ f^\circ~ \mathsf{L}[f g]  \tag{$f$ is $\mathsf{L}$-linear + (\ref{lstable-lem.pre})} \\
&=~f^\circ  fg \tag{$fg$ is  $\mathsf{L}$-linear} \\
&=~ g  \tag{$f$ is a retract of $f^\circ$}
\end{align*}
So $g$ is $\mathsf{L}$-linear. \\ \\
(\ref{lstable-lem.iso}): Suppose that $f$ is $\mathsf{L}$-linear and an isomorphism. By (\ref{lstable-lem.comp}), the composite $f f^{-1}=1$ is $\mathsf{L}$-linear. Since $f$ is a retract, by (\ref{lstable-lem.retract}) we have that $f^{-1}$ is also $\mathsf{L}$-linear. 
\end{proof}

 %\begin{corollary}\label{corLlin}  In a Cartesian left additive category with a linearizing combinator $\mathsf{L}$, 
 %\begin{enumerate}[{\em (i)}]
%\item \label{corLlin.tau} The symmetry isomorphism $\tau: A \times B \to B \times A$ is is $\mathsf{L}$-linear; 
%\item \label{corLlin.c} The interchange isomorphism $c: (A \times B) \times (C \times D) \to (A \times C) \times (B \times D)$ is is $\mathsf{L}$-linear; 
%\item \label{corLlin.ell} The lifting map $\ell: A \times D \to (A \times B) \times (C \times D)$ is is $\mathsf{L}$-linear; 
%\item \label{corLlin.oplus} The sum map $\oplus_A: A \times A \to A$ is is $\mathsf{L}$-linear. 
%\end{enumerate}
 %\end{corollary}
 
 Once again, it is important to note that while every $\mathsf{L}$-linear map is additive, not every additive map is necessarily $\mathsf{L}$-linear. That said, since every $\mathsf{L}$-linear map is additive, the subcategory of $\mathsf{L}$-linear maps form a category with finite biproducts. 

\begin{lemma} For a Cartesian left additive category $\mathbb{X}$ with a linearizing combinator $\mathsf{L}$, define $\mathbb{X}^{\mathsf{L}}$ as the subcategory of $\mathsf{L}$-linear maps of $\mathbb{X}$, that is, whose objects are the same as $\mathbb{X}$ and whose maps are the $\mathsf{L}$-linear maps between them. Then $\mathbb{X}^{\mathsf{L}}$ is a category with finite biproducts. Furthermore, for every map $f$ in $\mathbb{X}$, $\mathsf{L}[f]$ is a map in $\mathbb{X}^{\mathsf{L}}$. \end{lemma}
\begin{proof} That composition and identity maps in $\mathbb{X}^{\mathsf{L}}$ are well-defined follows from Lemma \ref{lstable-lem}.(\ref{lstable-lem.1}) and (\ref{lstable-lem.comp}). That $\mathbb{X}^{\mathsf{L}}$ has finite products follows from Lemma \ref{lstable-lem}.(\ref{lstable-lem.pi}) and (\ref{lstable-lem.pair}). That $\mathbb{X}^{\mathsf{L}}$ is a Cartesian left additive category follows from Lemma \ref{lstable-lem}.(\ref{lstable-lem.0}) and (\ref{lstable-lem.sum}). Note that a Cartesian left additive category where every map is additive is precisely a category with finite biproducts. By Lemma \ref{lstable-lem}.(\ref{lstable-lem.add}), it follows that every map in $\mathbb{X}^{\mathsf{L}}$ is additive. So $\mathbb{X}^{\mathsf{L}}$ is a category with finite biproducts. Lastly, by Lemma \ref{lstable-lem}.(\ref{lstable-lem.linear}), for every map $f$ in $\mathbb{X}$, $\mathsf{L}[f]$ is a map in $\mathbb{X}^{\mathsf{L}}$.
\end{proof}

Note that in general, the linearizing combinator does not induce a functor from $\mathbb{X}$ to $\mathbb{X}^{\mathsf{L}}$. However by Lemma \ref{linprop1}.(\ref{linprop1.ii}) and (\ref{linprop1.iii}), the linearizing combinator does induce a functor from the subcategories of reduced maps, semi-additive maps, and additive maps to $\mathbb{X}^{\mathsf{L}}$.

We now show that every Cartesian differential category comes equipped with a canonical linearizing combinator. Consider again the example of a smooth function $f: \mathbb{R} \to \mathbb{R}$, where $\mathsf{L}[f](x) = f^\prime(0)x$. Recall that $\mathsf{D}[f](x,y) = f^\prime(x) y$. Therefore, $\mathsf{L}[f](x) = \mathsf{D}[f](0,x)$. This construction generalizes to arbitrary Cartesian differential categories. Furthermore, maps which are linear in the Cartesian differential category sense, that is $\mathsf{D}$-linear, are precisely those for which $\mathsf{L}[f] = f$, that is those which are $\mathsf{L}$-linear . 

\begin{proposition}\label{DLprop} Every Cartesian differential category, with differential combinator $\mathsf{D}$, admits a linearizing combinator $\mathsf{L}_\mathsf{D}$ defined as follows for every map $f: A \to B$: 
\begin{equation}\label{LDdef}\begin{gathered} \mathsf{L}_\mathsf{D}[f] := \langle 0, 1 \rangle \mathsf{D}[f] \end{gathered}\end{equation}
Furthermore,
\begin{enumerate}[{\em (i)}]
\item \label{DLprop.1} For every map $f$, $\mathsf{L}_\mathsf{D}[f]$ is $\mathsf{D}$-linear;
\item \label{DLprop.2} A map $f$ is $\mathsf{D}$-linear if and only if $f$ is $\mathsf{L}_\mathsf{D}$-linear. 
\end{enumerate}
\end{proposition} 

\begin{proof} First note that (\ref{DLprop.1}) and (\ref{DLprop.2}) are precisely a reformulation of Lemma \ref{linlemimportant} in terms of $\mathsf{L}_\mathsf{D}$. Using this, we will now prove that $\mathsf{L}_\mathsf{D}$ satisfies \textbf{[L.1]} to \textbf{[L.6]}. Each of the linearizing combinator axioms will follow mostly from the differential combinator axiom of the same number. \\ \\
\noindent \textbf{[L.1]}: $\mathsf{L}_\mathsf{D}[f+g]=\mathsf{L}_\mathsf{D}[f]+\mathsf{L}_\mathsf{D}[g]$ and $\mathsf{L}_\mathsf{D}[0]=0$
\begin{align*}
\mathsf{L}_\mathsf{D}[f+g] &=~\langle 0, 1 \rangle \mathsf{D}[f+g] \\
&=~\langle 0, 1 \rangle\left(\mathsf{D}[f] + \mathsf{D}[g] \right)\tag*{\textbf{[CD.1]}} \\ 
&=~\langle 0, 1 \rangle \mathsf{D}[f] + \langle 0, 1 \rangle \mathsf{D}[g] \\
&=~  \mathsf{L}_\mathsf{D}[f]+\mathsf{L}_\mathsf{D}[g]
\end{align*}
\begin{align*}
\mathsf{L}_\mathsf{D}[0] &=~ \langle 0, 1 \rangle \mathsf{D}[0] \\
&=~\langle 0, 1 \rangle 0 \tag*{\textbf{[CD.1]}} \\ 
&=~ 0 
\end{align*}
\noindent \textbf{[L.2]}: $\oplus_A \mathsf{L}_\mathsf{D}[f] = \pi_0 \mathsf{L}_\mathsf{D}[f]+ \pi_1 \mathsf{L}_\mathsf{D}[f]$ and $0 \mathsf{L}_\mathsf{D}[f]=0$: \\ \\
\noindent By Proposition \ref{DLprop}.(\ref{DLprop.1}), $\mathsf{L}_\mathsf{D}[f]$ is $\mathsf{D}$-linear and therefore by Lemma \ref{linlem}.(\ref{linlem.add}), $\mathsf{L}_\mathsf{D}[f]$ is also additive, i.e., $\oplus_A \mathsf{L}_\mathsf{D}[f] = \pi_0 \mathsf{L}_\mathsf{D}[f]+ \pi_1 \mathsf{L}_\mathsf{D}[f]$ and $0 \mathsf{L}_\mathsf{D}[f]=0$. \\\\
\noindent \textbf{[L.3]}: $\mathsf{L}_\mathsf{D}[1]=1$ and $\mathsf{L}_\mathsf{D}[\pi_i]=\pi_i$: \\ \\
\noindent By Lemma \ref{linlem}.(\ref{linlem.1}) and (\ref{linlem.pi}), identity maps and projection maps are $\mathsf{D}$-linear. Therefore by Proposition \ref{DLprop}.(\ref{DLprop.2}), identity maps and projection maps are also $\mathsf{L}_\mathsf{D}$-linear, i.e., $\mathsf{L}_\mathsf{D}[1]=1$ and $\mathsf{L}_\mathsf{D}[\pi_i]=\pi_i$. \\\\
\noindent \textbf{[L.4]}: $\mathsf{L}_\mathsf{D}\!\left[ \langle f, g \rangle \right] = \left \langle \mathsf{L}_\mathsf{D}[f], \mathsf{L}_\mathsf{D}[g] \right \rangle$
\begin{align*}
\mathsf{L}_\mathsf{D}\!\left[ \langle f, g \rangle \right] &=~\langle 0, 1 \rangle \mathsf{D}\!\left[ \langle f, g \rangle \right] \\
&=~\langle 0, 1 \rangle \left \langle  \mathsf{D}[f] , \mathsf{D}[g] \right \rangle \tag*{\textbf{[CD.4]}} \\ 
&=~ \left \langle \langle 0, 1 \rangle \mathsf{D}[f] , \langle 0, 1 \rangle \mathsf{D}[g] \right \rangle \\
&=~\left \langle \mathsf{L}_\mathsf{D}[f], \mathsf{L}_\mathsf{D}[g] \right \rangle
\end{align*}
\noindent \textbf{[L.5]}: $\mathsf{L}_\mathsf{D}[fg] = \mathsf{L}_\mathsf{D}[f]~\mathsf{L}_\mathsf{D}\!\left[(1+0f)g\right]$
\begin{align*}
\mathsf{L}_\mathsf{D}[f]~\mathsf{L}_\mathsf{D}\!\left[(1+0f)g\right] &=~ \mathsf{L}_\mathsf{D}[f] \langle 0, 1 \rangle \mathsf{D}\!\left[ (1+0f)g \right] \\
&=~  \mathsf{L}_\mathsf{D}[f] \langle 0, 1 \rangle \left \langle \pi_0 (1+0f), \mathsf{D}[1+0f] \right \rangle  \mathsf{D}\!\left[ g \right] \tag*{\textbf{[CD.5]}}  \\
&=~  \mathsf{L}_\mathsf{D}[f] \langle 0, 1 \rangle \left \langle \pi_0 + \pi_0 0 f, \mathsf{D}[1] + \mathsf{D}[0f] \right \rangle  \mathsf{D}\!\left[ g \right] \tag*{\textbf{[CD.1]}}  \\
&=~ \mathsf{L}_\mathsf{D}[f] \langle 0, 1 \rangle \left \langle \pi_0 + 0 f, \pi_1 + \langle \pi_0 0, \mathsf{D}[0] \rangle \mathsf{D}[f] \right \rangle  \mathsf{D}\!\left[ g \right] \tag*{\textbf{[CD.3]} + \textbf{[CD.5]}}  \\
&=~  \mathsf{L}_\mathsf{D}[f] \langle 0, 1 \rangle \left \langle \pi_0 + 0 f, \pi_1 + \langle 0, 0 \rangle \mathsf{D}[f] \right \rangle  \mathsf{D}\!\left[ g \right] \tag*{\textbf{[CD.1]}}  \\
&=~  \mathsf{L}_\mathsf{D}[f] \langle 0, 1 \rangle \left \langle \pi_0 + 0 f, \pi_1 + 0 \right \rangle  \mathsf{D}\!\left[ g \right] \tag*{\textbf{[CD.2]}}  \\
&=~  \mathsf{L}_\mathsf{D}[f]  \left \langle \langle 0, 1 \rangle (\pi_0 + 0 f), \langle 0, 1 \rangle \pi_1 \right \rangle  \mathsf{D}\!\left[ g \right]   \\
&=~  \mathsf{L}_\mathsf{D}[f]  \left \langle \langle 0, 1 \rangle \pi_0 + \langle 0, 1 \rangle 0 f, \langle 0, 1 \rangle \pi_1 \right \rangle  \mathsf{D}\!\left[ g \right]   \\
&=~  \mathsf{L}_\mathsf{D}[f]  \left \langle  0 +  0 f, 1 \right \rangle  \mathsf{D}\!\left[ g \right]   \\
&=~  \mathsf{L}_\mathsf{D}[f]  \left \langle   0 f, 1 \right \rangle  \mathsf{D}\!\left[ g \right]   \\
&=~    \left \langle  \mathsf{L}_\mathsf{D}[f] 0 f, \mathsf{L}_\mathsf{D}[f] \right \rangle  \mathsf{D}\!\left[ g \right]   \\
&=~    \left \langle  0 f, \langle 0, 1 \rangle \mathsf{D}[f] \right \rangle  \mathsf{D}\!\left[ g \right]   \\
&=~    \left \langle  \langle 0,1 \rangle \pi_0 f, \langle 0, 1 \rangle \mathsf{D}[f] \right \rangle  \mathsf{D}\!\left[ g \right]   \\
&=~    \langle 0,1 \rangle \left \langle   \pi_0 f, \mathsf{D}[f] \right \rangle  \mathsf{D}\!\left[ g \right]   \\
&=~ \langle 0,1 \rangle \mathsf{D}[fg] \tag*{\textbf{[CD.5]}}  \\
&=~\mathsf{L}_\mathsf{D}[fg]
\end{align*}
\noindent \textbf{[L.6]}: $\mathsf{L}_\mathsf{D}\!\left[\mathsf{L}_\mathsf{D}[f]\right] = \mathsf{L}_\mathsf{D}[f]$: \\ \\
\noindent By Proposition \ref{DLprop}.(\ref{DLprop.1}), $\mathsf{L}_\mathsf{D}[f]$ is $\mathsf{D}$-linear. Therefore by Proposition \ref{DLprop}.(\ref{DLprop.2}), it follows that $\mathsf{L}_\mathsf{D}[f]$ is also $\mathsf{L}_\mathsf{D}$-linear which means that $\mathsf{L}_\mathsf{D}\!\left[\mathsf{L}_\mathsf{D}[f]\right] = \mathsf{L}_\mathsf{D}[f]$. \\\\
\noindent So we conclude that $\mathsf{L}_\mathsf{D}$ is a linearizing combinator. 
\end{proof} 

We conclude this section by providing examples of linearizing combinators by applying Proposition \ref{DLprop} to the examples of Cartesian differential categories from Section \ref{CDCsec}.

 \begin{example} \normalfont For a category with finite biproducts, the linearizing combinator is simply the identity combinator: 
 \[ \mathsf{L}_\mathsf{D}[f] = \langle 0,1 \rangle  \mathsf{D}[f] = \langle 0,1 \rangle \pi_1 f = f \]
 This make sense since every map, in this example, is already $\mathsf{D}$-linear by definition 
\end{example}

\begin{example} \normalfont For $\mathsf{SMOOTH}$, the linearizing combinator is defined as evaluating the directional derivative at zero in the first argument. Explicitly, for a smooth function $F: \mathbb{R}^n \to \mathbb{R}^m$, which recall is a tuple of smooth functions $F = \langle f_1, \hdots, f_m \rangle$: 
\[ \mathsf{L}[F](\vec x) = \nabla(F)(\vec 0) \cdot \vec x = \left \langle  \sum \limits^n_{i=1} \frac{\partial f_1}{\partial x_i}(\vec 0) x_i, \hdots,  \sum \limits^n_{i=1} \frac{\partial f_m}{\partial x_i}(\vec 0) x_i \right \rangle \]
%where recall that $\nabla(F)$ is the Jacobian matrix. In particular, for a smooth function $f: \mathbb{R}^n \to \mathbb{R}$: 
 %\[\mathsf{L}[f](\vec x) = \sum \limits^n_{i=1} \frac{\partial f}{\partial x_i}(\vec 0) x_i \]
For example, consider the polynomial function $f(x,y,z) = x^2y + 3x + z + 1$. Then $\mathsf{L}[f]$ picks out all monomials of degree 1 in $f$, so $\mathsf{L}[f](x,y,z) = 3x + z$. As another example, consider $g(x,y) = e^x \cos(y)$. Its derivative is worked out to be $\mathsf{D}[g](x,y, z,w) = e^x\cos(y)z - e^x \sin(y)w$. Then evaluating at $0$ in the first two arguments, we obtain that $\mathsf{L}[g](x,y) = e^0\cos(0)x - e^0 \sin(0)y = x$. 
\end{example}

\begin{example} \normalfont For $\mathsf{HoAbCat}_\mathsf{Ch}$, the linearizing combinator is precisely the linearization operator $\mathsf{D}_1$ as defined in \cite[Definition 5.1]{bauer2018directional}, which in turn is defined using cross effects \cite[Definition 2.1]{bauer2018directional}. 
\end{example}

\begin{example} \normalfont For a Cartesian left additive category $\mathbb{X}$, the linearizing combinator for its cofree Cartesian differential category $\mathcal{D}(\mathbb{X})$ is for a $\mathsf{D}$-sequence $(f_0, f_1, f_2, \hdots)$: 
\[ \mathsf{L}\left[ (f_0, f_1, f_2, \hdots ) \right] = ( \langle 0,1 \rangle f_1, \mathsf{P}(\langle 0,1 \rangle) f_2,  \mathsf{P}^2(\langle 0,1 \rangle) f_3, \hdots )  \]
where recall that $\mathsf{P}$ is the product functor $\mathsf{P}(-) = - \times - $. However by \cite[Lemma 4.26]{lemay2018tangent}, or by the axioms of a $\mathsf{D}$-sequence \cite[Definition 4.2]{lemay2018tangent}, it follows that $\mathsf{L}\left[ (f_0, f_1, f_2, \hdots ) \right]$ can be simplified to: 
\[ \mathsf{L}\left[ (f_0, f_1, f_2, \hdots ) \right] = ( \langle 0,1 \rangle f_1, \pi_1  \langle 0,1 \rangle f_1, \pi_1 \pi_1  \langle 0,1 \rangle f_1, \hdots ) =  i_\bullet \cdot ( \langle 0,1 \rangle f_1)  \]
where the notation $i_\bullet \cdot -$ was introduced in \cite{lemay2018tangent}. 
\end{example}

\begin{example} \normalfont For a differential category $\mathbb{X}$ with finite products, the linearizing combinator for the coKleisli category $\mathbb{X}_\oc$ is for a coKleisli map $f: \oc A \to B$: 
\[ \mathsf{L}[f] := \xymatrixcolsep{4pc}\xymatrix{\oc A \ar[r]^-{\Delta_A} & \oc A \otimes \oc A \ar[r]^-{\oc(0) \otimes \varepsilon_A} & \oc A \otimes A \ar[r]^-{\mathsf{d}_A} & \oc A \ar[r]^-{f} & B 
 } \] 
\end{example}

\begin{example}\label{ex:diffstorlin} \normalfont For a differential storage category $\mathbb{X}$, the linearizing combinator for the coKleisli category $\mathbb{X}_\oc$ can alternatively be expressed using the codereliction map: 
\[ \mathsf{L}[f] := \xymatrixcolsep{4pc}\xymatrix{\oc A \ar[r]^-{\varepsilon_A} & A \ar[r]^-{\eta_A} & \oc A \ar[r]^-{f} & B 
 } \]
In differential linear logic, $\eta$ is interpreted as producing the derivative evaluated at zero since $\eta = (u \otimes 1) \mathsf{d}$. CoKleisli maps $\oc A \to B$ are thought of as smooth maps, while maps in the base category $A \to B$ are thought of as linear. For every smooth map $f: \oc A \to B$, we obtain a linear map $\eta_A f: A \to B$, which we precompose by $\varepsilon$ to reobtain a smooth map $\varepsilon_A \eta_A f: \oc A \to B$. 
\end{example}

\begin{example} \normalfont For $\mathsf{CON}$, the linearizing combinator is defined as follows on a smooth function $f$:
\[ \mathsf{L}[f](x) := \lim \limits_{t \to 0} \frac{f(t \cdot x) - f(0)}{t} \]
\end{example}

\section{Differentiation in Context}\label{ContextSec}

We would like to prove the converse of Proposition \ref{DLprop}, that is, from a linearizing combinator we would like to construct a differential combinator following the same construction as in \cite{bauer2018directional}. However, to do so requires the ability to partially linearize maps, that is, to linearize on certain variables while keeping others constant.  This, equivalently, means we would like to be able to linearize with respect to a fixed ``context''.  From a categorical perspective, a map in a fixed ``context'' $C$ is interpreted as a map in the \textbf{simple slice category} over $C$.  Simple slice categories for a given category organize themselves into a fibration,  called the simple fibration \cite[Chapter 1]{jacobs1999categorical}. 

\begin{definition}\label{simpledef} Let $\mathbb{X}$ be a category with finite products. For each object $C$, the \textbf{simple slice category} \cite[Definition 1.3.1]{jacobs1999categorical} over $C$ is the category $\mathbb{X}[C]$ where:
\begin{enumerate}[{\em (i)}]
    \item The objects are the objects of $\mathbb{X}$, $ob\left(\mathbb{X}[C]\right) := ob\left(\mathbb{X} \right)$;
    \item The hom-sets are defined as $\mathbb{X}[C](A,B) := \mathbb{X}(C \times A,B)$;
    \item The identity maps are the projection maps $\pi_1: C \times A \to A$; 
    \item The composition of maps $f: C \times A \to B$ and $g: C \times B \to D$ is the map ${\langle \pi_0, f \rangle g: C \times A \to D}$. 
\end{enumerate}
For each map $h: C^\prime \to C$ in $\mathbb{X}$, define the \textbf{substitution functor} ${h^\ast: \mathbb{X}[C] \to \mathbb{X}[C^\prime]}$ on objects as $h^\ast(A) := A$ and on maps as $h^\ast(f) := (h \times 1)f$. 
\end{definition}

Note that for the terminal object $\top$ there is an isomorphism of categories $\mathbb{X}[\top] \cong \mathbb{X}$. Every simple slice $\mathbb{X}[C]$ admits finite products where on objects the product is the same as in $\mathbb{X}$, the projection maps are respectively $\pi_1 \pi_0: C \times (A \times B) \to A$ and $\pi_1 \pi_1: C \times (A \times B) \to B$, and the pairing of maps is the same as in $\mathbb{X}$. If $\mathbb{X}$ is a Cartesian left additive category, then so is every simple slice $\mathbb{X}[C]$ \cite[Corollary 1.3.5]{blute2009cartesian} where the sum and zero maps are defined again as in $\mathbb{X}$. As such, we can easily define what it means for a map to be additive in context. 

\begin{definition}\label{addseconddef} In a Cartesian left additive category $\mathbb{X}$, we say that a map $f: C \times A \to B$ is:
\begin{enumerate}[{\em (i)}]
\item \textbf{Constant in its second argument} if it is constant in $\mathbb{X}[C]$, that is, if $\langle \pi_0, 0 \rangle f = f$;
\item \textbf{Reduced in its second argument} if it is reduced in $\mathbb{X}[C]$, that is, if $\langle \pi_0, 0 \rangle f =0$;
\item \textbf{Semi-additive in its second argument} if it is semi-additive in $\mathbb{X}[C]$, that is, if \\
${\langle \pi_0, g+h \rangle f = \langle \pi_0, g \rangle f + \langle \pi_0, h \rangle f}$;
\item \textbf{Additive in its second argument} if it is additive in $\mathbb{X}[C]$, that is, if $\langle \pi_0, 0 \rangle f =0$ and $\langle \pi_0, g+h \rangle f = \langle \pi_0, g \rangle f + \langle \pi_0, h \rangle f$. 
\end{enumerate}
\end{definition}

\begin{lemma}\label{addlem2} In a Cartesian left additive category $\mathbb{X}$, 
\begin{enumerate}[{\em (i)}]
\item \label{addlem2.i} A map $f: C \times A \to B$ is constant in its second argument if and only if $f= \pi_0 g$ for some map ${g: C \to B}$;
\item \label{addlem2.ii} A map $f: C \times A \to B$ is additive in its second argument if and only if $\langle \pi_0, 0 \rangle f=0$ and $(1 \times \oplus_A) f = (1 \times \pi_0) f + (1 \times \pi_1)f$. 
\end{enumerate}
\end{lemma}
\begin{proof} For (\ref{addlem2.i}), it is immediate that $\langle \pi_0, 0 \rangle \pi_0 g = \pi_0 g$, so $f = \pi_0g$ is constant in its second argument.   Conversely, if $f$ is constant in its second argument then set $g = \langle 1,0 \rangle f$ then:
\begin{align*} \pi_0 g &=~ \pi_0 \langle 1,0 \rangle f \\
&=~ \langle \pi_0, 0 \rangle f \\
&=~ f \tag{$f$ is constant in its second argument}
\end{align*}
So, $f = \pi_0 g$. For  (\ref{addlem2.ii}), since being additive in its second argument is the same as being additive in the simple slice,  (\ref{addlem2.ii}) is simply re-expressing Lemma \ref{opluslem0}.(\ref{opluslem3}) using simple slice composition. 
\end{proof}

In classical multivariable differential calculus, the standard way of defining partial differentiation, or in other words differentiation in context, is by evaluating at zero certain terms of the total derivative. This is also how one obtains partial derivatives in a Cartesian differential category. In fact, for a Cartesian differential category, every simple slice is also a Cartesian differential category whose differential combinator is given by partial differentiation, which amounts to evaluating at zero in the context arguments of the total derivative. 

\begin{proposition}\label{Dcontext}\cite[Corollary 4.5.2]{blute2009cartesian} Let $\mathbb{X}$ be a Cartesian differential category with differential combinator $\mathsf{D}$. Then each simple slice $\mathbb{X}[C]$ is a Cartesian differential category with differential combinator $\mathsf{D}^{C}$ defined as follows for a map $f: C \times A \to B$ in $\mathbb{X}$: 
\begin{equation}\label{DCdef}\begin{gathered} \mathsf{D}^{C }[f] = \xymatrixcolsep{5pc}\xymatrix{C \times (A \times A) \ar[r]^-{\left \langle 1 \times \pi_0, 0 \times \pi_1 \right \rangle} & (C \times A) \times (C \times A) \ar[r]^-{\mathsf{D}[f]} & B }\end{gathered}\end{equation}
Furthermore, for every map $h: C^\prime \to C$ in $\mathbb{X}$, the substitution functor $h^\ast$ preserves the differential combinator in context \cite[Proposition 4.1.3]{blute2015cartesian}, that is, $(h \times 1)\mathsf{D}^{C^\prime }[f]= \mathsf{D}^{C } \left[ (h \times 1)f \right]$.
\end{proposition}

As with additivity, we can also define what it means for a map to be linear in context. 

\begin{definition}\label{lin2def} In a Cartesian differential category $\mathbb{X}$, a map $f: C \times A \to B$ is \textbf{linear in its second argument} if it is linear in $\mathbb{X}[C]$, that is, if $\mathsf{D}^{C}[f] = \langle \pi_0, \pi_1 \pi_1 \rangle f=(1 \times \pi_1) f$. 
\end{definition}

Being linear in context can also be expressed using the lifting map $\ell$, and as such by \textbf{[CD.6]}, it follows that derivatives of maps are also linear in their second argument. 

\begin{lemma}\label{lin2lem} In a Cartesian differential category:
\begin{enumerate}[{\em (i)}]
\item \label{lin2lem.i} A map $f: C \times A \to B$ is linear in its second argument if and only if $\ell \mathsf{D}[f] = f$. 
\item \label{lin2lem.ii} For every map $f: A \to B$, $\mathsf{D}[f]: A \times A \to B$ is linear in its second argument. 
\end{enumerate}
\end{lemma}
\begin{proof} For (\ref{lin2lem.i}), we first observe that for an arbitrary map $f: C \times A \to B$, we compute: 
\begin{align*}
\langle \pi_0, \langle 0, \pi_1 \rangle \rangle D^C[f] &=~  \langle \pi_0, \langle 0, \pi_1 \rangle \rangle \langle 1 \times \pi_0, 0 \times \pi_1 \rangle D[f] \\
& = \langle \langle \pi_0, \langle 0, \pi_1 \rangle \rangle (1 \times \pi_0) , \langle \pi_0, \langle 0, \pi_1 \rangle \rangle (0 \times \pi_1) \rangle D[f] \\
& = \langle \langle \pi_0, \langle 0, \pi_1 \rangle \pi_0 \rangle  , \langle 0, \langle 0, \pi_1 \rangle \pi_1 \rangle\rangle D[f] \\
& =  \langle \langle \pi_0, 0 \rangle  , \langle 0, \pi_1 \rangle\rangle D[f] \\
& =  \langle \pi_0 \langle 1, 0 \rangle  , \pi_1 \langle 0, 1 \rangle\rangle D[f] \\
& =  (\langle 1,0 \rangle \times \langle 0,1 \rangle) D[f] \\
& =  \ell D[f]
\end{align*}
So $\langle \pi_0, \langle 0, \pi_1 \rangle \rangle D^C[f]  = \ell D[f]$. Now recall that by Lemma \ref{linlemimportant}.(\ref{linlemimportant2}), in a Cartesian differential category a map $g$ is linear if and only if $\langle 0,1 \rangle D[g] = g$. Since every simple slice category is again a Cartesian differential category, then putting Lemma \ref{linlemimportant}.(\ref{linlemimportant2}) into a context $C$ means that a map $f: C \times A \to B$ is linear in its second argument if and only if  
$\langle \pi_0, \langle 0, \pi_1 \rangle \rangle D^C[f] = f$. However by the above calculations, we may re-express by saying that a map $f: C \times A \to B$ is linear in its second argument if and only if $\ell \mathsf{D}[f] = f$. For (\ref{lin2lem.ii}), for any map $f: A \to B$, by \textbf{[CD.6]} we have that $\ell ~\mathsf{D}\!\left[\mathsf{D}[f] \right] = \mathsf{D}[f]$. Then by (\ref{lin2lem.i}), it follows that $\mathsf{D}[f]$ is linear in its second argument. 
\end{proof}

We conclude this section by taking a look at the differential combinators in context and maps which are linear in context in the examples of Cartesian differential categories from Section \ref{CDCsec}. 

\begin{example} \normalfont In a category with finite biproducts, the differential combinator in context $C$ is defined on a map $f: C \times A \to B$ as follows:  
\[\mathsf{D}^C[f] = (0 \times \pi_1) f\] 
A map $f: C \times A \to B$ is linear in its second argument if and only if $f = (0 \times 1) f$, or in other words, if $f$ does not depend on its first argument. 
\end{example}

\begin{example} \normalfont   In $\mathsf{SMOOTH}$, for a smooth function $F: \mathbb{R}^k \times \mathbb{R}^n \to \mathbb{R}^m$, $F = \langle f_1, \hdots, f_m \rangle$, its partial derivative is the smooth function $\mathsf{D}^{\mathbb{R}^k}[F]: \mathbb{R}^k \times (\mathbb{R}^n \times \mathbb{R}^n) \to \mathbb{R}^m$ defined as follows: 
\[ \mathsf{D}^{\mathbb{R}^k}[F](\vec z, \vec x, \vec y) = \nabla(F)(\vec z, \vec x) \cdot (\vec 0, \vec y) =\left \langle \sum \limits^n_{i=1} \frac{\partial f_1}{\partial x_i}(\vec z,\vec x) y_i, \hdots, \sum \limits^n_{i=1} \frac{\partial f_m}{\partial x_i}(\vec z,\vec x) y_i \right \rangle \]
A smooth function $F: \mathbb{R}^k \times \mathbb{R}^n \to \mathbb{R}^m$ is linear in its second argument if it is $\mathbb{R}$-linear in its second argument, that is, $F(\vec z, s \vec x + t \vec y) = sF(\vec z, \vec x) + t F(\vec z, \vec y)$ for all $s,t \in \mathbb{R}$, $\vec z \in \mathbb{R}^k$, and $\vec x, \vec y \in \mathbb{R}^n$. 
\end{example}

\begin{example} \normalfont In $\mathsf{HoAbCat}_\mathsf{Ch}$, the differential combinator in context $C$ is defined on a functor $F: \mathbb{C} \times \mathbb{A} \to \mathsf{Ch}(\mathbb{B})$ as follows: 
\[ \nabla^C F (Z,X,V) := D^2_1 F(Z,X \oplus -)(V)  \]
where $\mathsf{D}^i_1$ is the partial linearization operator as defined in \cite[Convention 5.11]{bauer2018directional}. A functor ${F: \mathbb{C} \times \mathbb{A} \to \mathsf{Ch}(\mathbb{B})}$ is linear in its second argument if $F$ preserves finite direct sums up to chain homotopy equivalence in its second argument. 
\end{example}

\begin{example} \normalfont For a Cartesian left additive category $\mathbb{X}$, the differential combinator in context $C$ for its cofree Cartesian differential category $\mathcal{D}(\mathbb{X})$ is worked out to be as follows for a $\mathsf{D}$-sequence $(f_0, f_1, f_2, \hdots): C \times A \to B$ (so $f_n: \mathsf{P}^n(C \times A) \to B$): 
\[ \mathsf{D}^C\left[ (f_0, f_1, f_2, \hdots ) \right] = \left( (\left \langle 1 \times \pi_0, 0 \times \pi_1 \right \rangle) f_1, \mathsf{P}\left( \left \langle 1 \times \pi_0, 0 \times \pi_1 \right \rangle \right) f_2,  \mathsf{P}^2\left( \left \langle 1 \times \pi_0, 0 \times \pi_1 \right \rangle \right) f_3, \hdots \right)  \]
A $\mathsf{D}$-sequence $(f_0, f_1, f_2, \hdots): C \times A \to B$ is linear in its second argument if and only if for all $n \in \mathbb{N}$: $f_n =  \underbrace{\ell \hdots \ell}_{n-times} f_0$ . 
\end{example}

\begin{example} \normalfont For a differential category $\mathbb{X}$ with finite products, the differential combinator in context $C$ for the coKleisli category $\mathbb{X}_\oc$ is worked out to be as follows for a coKleisli map ${f: \oc (C \times A) \to B}$: 
\[\begin{array}[c]{c} \mathsf{D}^C[f] \end{array} := \begin{array}[c]{c} \xymatrixrowsep{1pc}\xymatrixcolsep{3.5pc}\xymatrix{\oc\left(C \times (A \times A)\right) \ar[r]^-{\langle 1 \times \pi_0, \pi_1 \pi_1 \rangle} &\oc\left((C \times A) \times A\right)  \ar[r]^-{\chi_{C \times A, A}} & \oc(C \times A) \otimes \oc A \ar[r]^-{1 \otimes \varepsilon_A} &  \\
\oc (C \times A) \otimes A \ar[r]^-{1 \otimes \langle 0,1 \rangle} & \oc(C \times A) \otimes (C \times A) \ar[r]^-{\mathsf{d}_{C \times A}} & \oc(C \times A) \ar[r]^-{f} & B 
 }  \end{array}\] 
Applying Lemma \ref{lin2lem}.(\ref{lin2lem.i}), and being careful with coKleisli composition, one can show that a coKleisli map $f: \oc (C \times A) \to B$ is linear in its second argument if and only if 
\[\chi_{C,A} (1 \otimes \varepsilon_A)\left( \oc(\langle 1, 0 \rangle) \otimes \langle 0,1 \rangle \right) \mathsf{d}_{C \times A} f = f\] 
In particular, for every map $g: \oc C \otimes A \to B$ in $\mathbb{X}$, $\chi_{C,A}  (1 \otimes \varepsilon_A) g: \oc (C \times A) \to B$ is linear in its second argument in the coKleisli category $\mathbb{X}_\oc$. 
\end{example}

\begin{example} \normalfont For a differential storage category $\mathbb{X}$, the differential combinator in context $C$ for the coKleisli category $\mathbb{X}_\oc$ can alternatively be expressed using the codereliction map and the Seely isomorphisms as follows for a coKleisli map ${f: \oc (C \times A) \to B}$: 
\[  \begin{array}[c]{c}\mathsf{D}^C[f]  \end{array} \!\!\!:=\!\!\! \begin{array}[c]{c} \xymatrixrowsep{1pc}\xymatrixcolsep{2.5pc}\xymatrix{\oc\left(C \times (A \times A)\right) \ar[r]^-{\chi_{C,A \times A}} & \oc C \otimes \oc (A \times A)  \ar[r]^-{1 \otimes \chi_{A,A}} & \oc C \otimes \left( \oc A \otimes \oc A \right)   \ar[r]^-{1 \otimes (1 \otimes \varepsilon_A) } &\\
 \oc C \otimes \left( \oc A \otimes A \right)  \ar[r]^-{1 \otimes (1 \otimes \eta_A) } &  \oc C \otimes \left( \oc A \otimes \oc A \right)  \ar[r]^-{1 \otimes \nabla_A} & \oc C \otimes \oc A \ar[r]^-{\chi^{-1}_{C,A}} & \oc(C \times A)  \ar[r]^-{f} & B 
 }  \end{array} \]
In this case, the maps which are linear in their second argument in the coKleisli category $\mathbb{X}_\oc$ are precisely those of the form  $\chi_{C,A}  (1 \otimes \varepsilon_A) g: \oc (C \times A) \to B$ for a map $g: \oc C \otimes A \to B$ in $\mathbb{X}$. 
\end{example}

\begin{example} \normalfont In $\mathsf{CON}$, the differential combinator in context $C$ is defined as follows on a smooth function $f: C \times E \to F$:
\[ \mathsf{D}^C[f](z,x,y) := \lim \limits_{t \to 0} \frac{f(z, x + t \cdot y) - f(z,x)}{t} \]
A smooth function is linear in its second argument in the Cartesian differential sense precisely when it is $\mathbb{R}$-linear in its second argument. 
\end{example}

\section{System of Linearizing Combinators} \label{system-sec}

Partial linearization is a key operation in \cite{bauer2018directional} as it used in the construction of the differential combinator. However, while it is always possible to define partial differentiation from total differentiation, in general it is not necessarily possible to define partial linearization from total linearization (see Example \ref{counter-example} at the end of this section). As such, we need to separately define the notion of linearizing combinators in contexts, which we call a \emph{system} of linearizing combinators, which amounts to requiring that each simple slice admits a linearizing combinator. This means that the first six axioms in Definition \ref{syslindef} below simply place the axioms of Definition \ref{lindef} in an arbitrary context --  this is why they are given the same label.  However, we need an additional two axioms to ensure the correct relation between the linearizing combinators in the slices. 

We will show systems of linearizing combinators are in bijective correspondence with differential combinators. Therefore, a system of linearizing combinators provides an alternative axiomatization for a Cartesian differential category. 

\begin{definition}\label{syslindef} A \textbf{system of linearizing combinators} on a Cartesian left additive category $\mathbb{X}$ is a family of linearizing combinators $\mathsf{L}^{C}$ indexed by every object $C \in \mathbb{X}$, where $\mathsf{L}^{C}$ is a linearizing combinator for the simple slice category $\mathbb{X}[C]$, that is, the following axioms hold: 
\begin{enumerate}[{\bf [L.1]}]
\item $\mathsf{L}^{C}[f+g]=\mathsf{L}^{C}[f]+\mathsf{L}^{C}[g]$ and $\mathsf{L}^{C}[0]=0$
\item $\mathsf{L}^{C}[f]$ is additive in its second argument, or equivalently by Lemma \ref{addlem2}.(\ref{addlem2.ii}):
\begin{align*}(1 \times \oplus_A) \mathsf{L}^{C}[f] = (1 \times \pi_0) \mathsf{L}^{C}[f]+ (1 \times \pi_1) \mathsf{L}^{C}[f] && \langle \pi_0, 0 \rangle \mathsf{L}^{C}[f]=0
\end{align*}
\item $\mathsf{L}^{C}[\pi_1]=\pi_1$, $\mathsf{L}^{C}[\pi_1\pi_0]=\pi_1\pi_0$, and $\mathsf{L}^{C}[\pi_1\pi_1]=\pi_1\pi_1$
\item $\mathsf{L}^{C}\left[ \langle f, g \rangle \right] = \left \langle \mathsf{L}^{C}[f], \mathsf{L}^{C}[g] \right \rangle $
\item $\mathsf{L}^{C}[\langle \pi_0, f \rangle g] = \langle \pi_0, \mathsf{L}^{C}[f] \rangle~ \mathsf{L}^{C}\left[\left \langle \pi_0, \pi_1 + \langle \pi_0, 0 \rangle f \right \rangle g\right]$
\item $\mathsf{L}^{C}\left[\mathsf{L}^{C}[f]\right] = \mathsf{L}^{C}[f]$
\end{enumerate}
and such that the following two extra axioms hold: 
\begin{enumerate}[{\bf [L.1]}]
\setcounter{enumi}{6}
\item Let $\alpha: C \times (A \times B) \to (C \times A) \times B$ and $\beta: C \times (A \times B) \to (C \times B) \times A$ be the canonical natural isomorphisms respectively defined as follows: 
\begin{align*}
\alpha = \langle 1 \times \pi_0, \pi_1 \pi_1 \rangle && \beta = \langle 1 \times \pi_1, \pi_1 \pi_0 \rangle
\end{align*}
and for a map $f: C \times (A \times B) \to D$, define the maps $\mathsf{L}^C_0[f]:  C \times (A \times B) \to D$ and $\mathsf{L}^C_1[f]:  C \times (A \times B) \to D$ respectively as follows: 
\begin{align*}
\mathsf{L}^C_0[f] : = \beta \mathsf{L}^{C \times B}[\beta^{-1} f] && \mathsf{L}^C_1[f] : = \alpha \mathsf{L}^{C \times A}[\alpha^{-1} f] 
\end{align*}
Then for any map $f: C \times (A \times B) \to D$, $ \mathsf{L}^C_1[\mathsf{L}^C_0[f]] =  \mathsf{L}^C_0[\mathsf{L}^C_1[f]]$.  
\item For any map $h: C^\prime \to C$ in $\mathbb{X}$, the substitution functor $h^\ast$ (as defined in Definition \ref{simpledef}) preserves the linearizing combinator in context, that is, $(h \times 1)\mathsf{L}^{C}[f]  = \mathsf{L}^{C^\prime}[(h \times 1)f]$ 
\end{enumerate}
\end{definition}

 {\bf [L.8]} simply says that partial linearization is unaffected by changes in the context argument.  On the other hand, {\bf [L.7]} is admittedly slightly complex at first glance, however it amounts to the linearizing combinator analogue of {\bf [CD.7]} and states the symmetry of partial linearization. Indeed, $\mathsf{L}^C_0[f]$ is the linearization of $A$ while keeping $C$ and $B$ in context, while $\mathsf{L}^C_1[f]$ is the linearization of $B$ while keeping $C$ and $A$ in context. In particular, {\bf [L.7]} is also a generalization of \cite[Lemma 5.15]{bauer2018directional}, which expresses sequential linearization as discussed in \cite[Convention 5.11]{bauer2018directional} (though in \cite{bauer2018directional}, there is no extra context variables, that is, $C = \top$ -- the terminal object). Therefore, {\bf [L.7]} expresses that linearizing $A$ first then linearizing $B$ (while keeping the other variables in context) is the same as linearizing $B$ first then $A$. As an example, consider the polynomial function $f(x,y) = xy + 2xy^3+ 3x + 4y$. The total linearization of $f$, that is, linearizing $f$ jointly in $x$ and $y$ is the polynomial $\mathsf{L}[f](x,y) = 3x + 4y$. Linearizing $f$ in terms of $x$ while keeping $y$ in context picks out the terms where $x$ is of degree 1, and therefore results in the polynomial $\mathsf{L}_0[f] = xy + 2xy^3 + 3x$, which is now linear in $x$. On the other hand, linearizing $f$ in terms of $y$ while keeping $x$ in context results in the polynomial $\mathsf{L}_1[f] = xy + 4y$, which this time is linear in $y$. Linearizing $xy + 2xy^3 + 3x$ in terms of $y$ or linearizing $xy + 4y$ in terms of $x$ both results in $\mathsf{L}_1[\mathsf{L}_0[f]] =  \mathsf{L}_0[\mathsf{L}_1[f]] = xy$, which is an example of {\bf [L.7]}. In Proposition \ref{L7Prop}, we will provide an equivalent alternative version of \textbf{[L.7]} which requires less setup.  

Our first observation is that, since there is an isomorphism between the base category and the simple slice category over the terminal object, it follows that a system of linearizing combinators also induces a linearizing combinator on the base category. 

\begin{proposition}\label{lemL1} Let $\mathbb{X}$ be a Cartesian left additive category with a system of linearizing combinators $\mathsf{L}^{C}$. Then $\mathbb{X}$ has a linearizing combinator $\mathsf{L}$ defined as follows for a map $f: A \to B$: 
\begin{equation}\label{L1def}\begin{gathered} \mathsf{L}[f] = \xymatrixcolsep{5pc}\xymatrix{A \ar[r]^-{\langle 0,1 \rangle} & \top \times A \ar[r]^-{\mathsf{L}^\top[\pi_1 f]} & B }\end{gathered}\end{equation}
where $\top$ is the terminal object. Furthermore: 
\begin{enumerate}[{\em (i)}]
\item \label{lemL1.i} For every map $f: A \to B$ and every object $C$, $\mathsf{L}^C[\pi_1f] = \pi_1 \mathsf{L}[f]$;
\item \label{lemL1.ii} If $f$ is $\mathsf{L}$-linear then for every object $C$, $\pi_1 f$ is $\mathsf{L}^C$-linear;
\item \label{lemL1.iii} For every map $\mathsf{L}$-linear map $f$, $(h \times f)\mathsf{L}^{C^\prime}[g] =\mathsf{L}^C[(h \times f)g]$;
\item \label{lemL1.iv} For a map $f: A \times B \to C$, define $\mathsf{L}_0[f]: A \times B \to C$ and $\mathsf{L}_1[f]: A \times B \to C$ respectively as: 
\begin{align*}
\mathsf{L}_0[f] := \tau \mathsf{L}^B[\tau f] && \mathsf{L}_1[f] := \mathsf{L}^A[f]
\end{align*}
where $\tau$ is the canonical symmetry isomorphism as defined in (\ref{taudef}). Then for every map ${f: A \times B \to C}$, $\mathsf{L}_0[\mathsf{L}_1[f]] =  \mathsf{L}_1[\mathsf{L}_0[f]]$.  
\end{enumerate}
\end{proposition}
\begin{proof} First note that for the terminal object, $\pi_1: \top \times A \to A$ and $\langle 0,1 \rangle: A \to \top\times A$ are inverses of each other. We now show that $\mathsf{L}$ is a linearizing combinator by showing it satisfies {\bf [L.1]} to {\bf [L.6]} of Definition \ref{lindef}: \\ \\
\noindent {\bf [L.1]}: $\mathsf{L}[f+g]=\mathsf{L}[f]+\mathsf{L}[g]$ and $\mathsf{L}[0]=0$ 
\begin{align*}
\mathsf{L}[f+g]&=~ \langle 0,1 \rangle \mathsf{L}^\top[\pi_1 (f+g)] \\
&=~ \langle 0,1 \rangle \mathsf{L}^\top[\pi_1 f + \pi_1 g] \\
&=~   \langle 0,1 \rangle \left(\mathsf{L}^\top[\pi_1 f] +  \mathsf{L}^\top[\pi_1 g] \right) \tag*{\bf [L.1]} \\
&=~ \langle 0,1 \rangle \mathsf{L}^\top[\pi_1 f ] + \langle 0,1 \rangle \mathsf{L}^\top[\pi_1 g] \\
&=~ \mathsf{L}[f]+\mathsf{L}[g] \\ \\
\mathsf{L}[0] &=~  \langle 0,1 \rangle \mathsf{L}^\top[\pi_1 0 ] \\
&=~  \langle 0,1 \rangle \mathsf{L}^\top[0] \\
&=~     \langle 0,1 \rangle 0 \tag*{\bf [L.1]} \\
&=~ 0 
\end{align*}
\noindent {\bf [L.2]}: $\oplus_A \mathsf{L}[f] = \pi_0 \mathsf{L}[f]+ \pi_1 \mathsf{L}[f]$ and $0 \mathsf{L}[f]=0$ 
\begin{align*}
\oplus_A \mathsf{L}[f] &=~ \oplus_A \langle 0,1 \rangle \mathsf{L}^\top[\pi_1 f ] \\
&=~  \langle 0,\oplus_A \rangle\mathsf{L}^\top[\pi_1 f ] \\
&=~ \langle 0, 1 \rangle (1 \times \oplus_A) \mathsf{L}^\top[\pi_1 f ] \\
&=~ \langle 0, 1 \rangle (1 \times \pi_0) \mathsf{L}^\top[\pi_1 f ] + \langle 0, 1 \rangle (1 \times \pi_1) \mathsf{L}^\top[\pi_1 f] \tag*{\bf [L.2]} \\
&=~ \langle 0, \pi_0 \rangle \mathsf{L}^\top[\pi_1 f ] +  \langle 0, \pi_1 \rangle \mathsf{L}^\top[\pi_1 f ] \\
&=~ \pi_0 \langle 0,1 \rangle \mathsf{L}^\top[\pi_1 f] + \pi_1\langle 0,1 \rangle \mathsf{L}^\top[\pi_1 f]   \\
&=~ \pi_0 \mathsf{L}[f]+ \pi_1 \mathsf{L}[f] \\ \\
0 \mathsf{L}[f] &=~ 0 \langle 0,1 \rangle \mathsf{L}^\top[\pi_1 f ] \\
&=~  \langle 0,0\rangle\mathsf{L}^\top[\pi_1 f ] \\
&=~  \langle 0 \pi_0,0\rangle\mathsf{L}^\top[\pi_1 f ] \tag{$\pi_0$ is additive}\\
&=~ 0 \langle \pi_0,0 \rangle \mathsf{L}^\top[\pi_1 f ] \\
&=~ 0 0 \tag*{\bf [L.2]} \\
&=~ 0
\end{align*}
\noindent  {\bf [L.3]}:  $\mathsf{L}[1]=1$ and $\mathsf{L}[\pi_i]=\pi_i$
\begin{align*}
\mathsf{L}[1] &=~ \langle 0,1 \rangle \mathsf{L}^\top[\pi_1] \\
&=~ \langle 0,1 \rangle \pi_1  \tag*{\bf [L.3]} \\
&=~ 1 \\ \\
\mathsf{L}[\pi_i] &=~ \langle 0,1 \rangle \mathsf{L}^\top[\pi_1\pi_i] \\
&=~ \langle 0,1 \rangle \pi_1\pi_i  \tag*{\bf [L.3]} \\
&=~ \pi_i
\end{align*} 
\noindent  {\bf [L.4]}: $\mathsf{L}\left[ \langle f, g \rangle \right] = \left \langle \mathsf{L}[f], \mathsf{L}[g] \right \rangle$ 
\begin{align*}
\mathsf{L}\left[ \langle f, g \rangle \right] &=~ \langle 0,1 \rangle \mathsf{L}^\top[\pi_1\langle f, g \rangle] \\
&=~ \langle 0,1 \rangle \mathsf{L}^\top[\langle \pi_1f, \pi_1g \rangle] \\
&=~  \langle 0,1 \rangle\left \langle \mathsf{L}^\top[\pi_1f], \mathsf{L}^\top[\pi_1g] \right \rangle \tag*{\bf [L.4]} \\
&=~ \left \langle \langle 0,1 \rangle\mathsf{L}^\top[\pi_1f], \langle 0,1 \rangle\mathsf{L}^\top[\pi_1g] \right \rangle \\
&=~ \left \langle \mathsf{L}[f], \mathsf{L}[g] \right \rangle
\end{align*}
\noindent  {\bf [L.5]}: $\mathsf{L}[fg] = \mathsf{L}[f]~\mathsf{L}\left[(1+0f)g\right]$
\begin{align*}
\mathsf{L}[fg] &=~  \langle 0,1 \rangle \mathsf{L}^\top[\pi_1fg] \\
&=~  \langle 0,1 \rangle \mathsf{L}^\top[\langle \pi_0, \pi_1f \rangle \pi_1 g] \\
&=~ \langle 0,1 \rangle \langle \pi_0, \mathsf{L}^{\top}[\pi_1f] \rangle~ \mathsf{L}^{\top}\left[\left \langle \pi_0, \pi_1 + \langle \pi_0, 0 \rangle \pi_1 f \right \rangle \pi_1 g\right] \tag*{\bf [L.5]} \\ 
&=~ \langle 0,1 \rangle \langle \pi_0, \mathsf{L}^{\top}[\pi_1f] \rangle~ \mathsf{L}^{\top}\left[\left \langle \pi_0, \pi_1 + 0 f \right \rangle \pi_1 g\right] \\
&=~ \langle 0,1 \rangle \langle \pi_0, \mathsf{L}^{\top}[\pi_1f] \rangle~ \mathsf{L}^{\top}\left[ \left(\pi_1 + 0 f \right) g\right] \\
&=~ \left \langle \langle 0,1 \rangle \pi_0, \langle 0,1 \rangle\mathsf{L}^{\top}[\pi_1f] \right \rangle~ \mathsf{L}^{\top}\left[ \left(\pi_1 + 0 f \right) g\right] \\
&=~ \left \langle 0, \langle 0,1 \rangle\mathsf{L}^{\top}[\pi_1f] \right \rangle~ \mathsf{L}^{\top}\left[ \left(\pi_1 + 0 f \right) g\right] \\
&=~ \left \langle 0, \mathsf{L}[f] \right \rangle~ \mathsf{L}^{\top}\left[ \left(\pi_1 + 0 f \right) g\right] \\
&=~ \mathsf{L}[f] \langle 0,1 \rangle ~ \mathsf{L}^{\top}\left[ \left(\pi_1 + 0 f \right) g\right] \\
&=~ \mathsf{L}[f] \langle 0,1 \rangle ~ \mathsf{L}^{\top}\left[ \pi_1 \left(1 + 0 f \right) g\right] \\
&=~ \mathsf{L}[f]~\mathsf{L}\left[(1+0f)g\right]
\end{align*}
\noindent  {\bf [L.6]}: $\mathsf{L}\left[\mathsf{L}[f]\right] = \mathsf{L}[f]$ 
\begin{align*}
\mathsf{L}\left[\mathsf{L}[f]\right] &=~ \langle 0,1 \rangle\mathsf{L}^\top\left[\pi_1 \langle 0,1 \rangle\mathsf{L}^\top[\pi_1f] \right] \\
&=~ \langle 0,1 \rangle\mathsf{L}^\top\left[  \mathsf{L}^\top[\pi_1f] \right] \tag{$\pi_1$ and $\langle 0,1 \rangle$ are inverses} \\ 
&=~ \langle 0,1 \rangle \mathsf{L}^\top[ \pi_1f] \tag*{\bf [L.6]} \\ 
&=~ \mathsf{L}[f]
\end{align*}
So we conclude that $\mathsf{L}$ is a linearizing combinator. For (\ref{lemL1.i}), for every object $C$, we compute: 
\begin{align*}
\mathsf{L}^C[\pi_1f]&=~ \mathsf{L}^C[(0 \times 1) \pi_1f] \\
&=~ (0 \times 1) \mathsf{L}^\top[ \pi_1f] \tag*{\bf [L.8]} \\ 
&=~ \langle 0, \pi_1 \rangle  \mathsf{L}^\top[ \pi_1f] \\
&=~ \pi_1  \langle 0,1 \rangle \mathsf{L}^\top[ \pi_1f] \\
&=~ \pi_1 \mathsf{L}[f] 
\end{align*}
So $\mathsf{L}^C[\pi_1f] = \pi_1 \mathsf{L}[f]$. For (\ref{lemL1.ii}), suppose that $f$ is $\mathsf{L}$-linear, that is, $\mathsf{L}[f] = f$. Then it follows that $\mathsf{L}^C[\pi_1f] = \pi_1 f$ and so $\pi_1 f$ is $\mathsf{L}^C$-linear. For (\ref{lemL1.iii}), suppose again that $f$ is $\mathsf{L}$-linear, and so $\pi_1 f$ is $\mathsf{L}^{C^\prime}$-linear. By Lemma \ref{lstable-lem}.(\ref{lstable-lem.add}), $\pi_1 f$ is also additive (and so reduced) in the simple slice category. Then using Lemma \ref{linprop1} with respect to simple slice composition, we have that:
\begin{align*}
\mathsf{L}^C[(h \times f)g] &=~ \mathsf{L}^C[(h \times 1)(1 \times f)g] \\
&=~ (h \times 1) \mathsf{L}^{C^\prime}[(1 \times f)g] \tag*{\bf [L.8]} \\ 
&=~ (h \times 1) \mathsf{L}^{C^\prime}[\langle \pi_0, \pi_1f \rangle g] \\
&=~ (h \times 1) \langle \pi_0, \mathsf{L}^{C^\prime}[\pi_1 f] \rangle \mathsf{L}^{C^\prime}[g] \tag{Lemma \ref{linprop1}.(\ref{linprop1.ii})}\\
&=~ (h \times 1) \langle \pi_0,\pi_1 f \rangle\mathsf{L}^{C^\prime}[g] \tag{$f$ is $\mathsf{L}$-linear, so $\pi_1 f$ is $\mathsf{L}^{C^\prime}$-linear}\\
&=~ (h \times 1)(1 \times f) \mathsf{L}^{C^\prime}[g] \\
&=~ (h \times f) \mathsf{L}^{C^\prime}[g]
\end{align*}
So we have that $ (h \times f) \mathsf{L}^{C^\prime}[g]=\mathsf{L}^C[(h \times f)g]$, when $f$ is $\mathsf{L}$-linear. Lastly (\ref{lemL1.iv}) is a special case of \textbf{[L.7]} when $C = \top$. First observe that $\beta = (1 \times \tau) \alpha$, where $\alpha$ and $\beta$ are defined as in Definition \ref{syslindef}.\textbf{[L.7]}. So we compute: 
\begin{align*}
\mathsf{L}_0[\mathsf{L}_1[f]] &=~ \tau \mathsf{L}^B[\tau \mathsf{L}^A[ f ] ] \\
&=~ \tau \mathsf{L}^B[\tau \mathsf{L}^A[ \langle 0,1 \times 1 \rangle \pi_1 f ] ] \\
&=~ \tau \mathsf{L}^B[\tau \mathsf{L}^A[ (\langle 0,1 \rangle \times 1) \alpha^{-1} \pi_1 f ] ] \\
&=~ \tau \mathsf{L}^B[\tau  (\langle 0,1 \rangle \times 1) \mathsf{L}^{\top \times A}[ \alpha^{-1} \pi_1 f ] ] \tag*{\textbf{[L.8]}} \\
&=~ \tau \mathsf{L}^B[(\langle 0,1 \rangle \times 1)\beta^{-1} \alpha \mathsf{L}^{\top \times A}[ \alpha^{-1} \pi_1 f ] ] \\
&=~ \tau (\langle 0,1 \rangle \times 1) \mathsf{L}^{\top \times B}[\beta^{-1} \alpha \mathsf{L}^{\top \times A}[ \alpha^{-1} \pi_1 f ] ]  \tag*{\textbf{[L.8]}} \\
&=~ (\langle 0,1 \rangle \times 1) \alpha^{-1} \beta \mathsf{L}^{\top \times B}[\beta^{-1} \alpha \mathsf{L}^{\top \times A}[ \alpha^{-1} \pi_1 f ] ]  \\
&=~ (\langle 0,1 \rangle \times 1)  \alpha^{-1} \mathsf{L}^\top_0 [ \mathsf{L}^\top_1[ \pi_1 f]] \\ 
&=~  (\langle 0,1 \rangle \times 1)  \alpha^{-1} \mathsf{L}^\top_1 [ \mathsf{L}^\top_0[ \pi_1 f]] \tag*{\textbf{[L.7]}} \\
&=~  (\langle 0,1 \rangle \times 1)  \alpha^{-1} \alpha \mathsf{L}^{\top \times A} [\alpha^{-1} \beta \mathsf{L}^{\top \times B}[\beta^{-1} \pi_1 f]] \\
&=~  (\langle 0,1 \rangle \times 1) \mathsf{L}^{\top \times A} [\alpha^{-1} \beta \mathsf{L}^{\top \times B}[\beta^{-1} \pi_1 f]] \\
&=~  \mathsf{L}^{A} [(\langle 0,1 \rangle \times 1) \alpha^{-1} \beta \mathsf{L}^{\top \times B}[\beta^{-1} \pi_1 f]]   \tag*{\textbf{[L.8]}} \\
&=~  \mathsf{L}^{A} [\tau (\langle 0,1 \rangle \times 1) \mathsf{L}^{\top \times B}[\beta^{-1} \pi_1 f]] \\
&=~  \mathsf{L}^{A} [\tau  \mathsf{L}^{B}[(\langle 0,1 \rangle \times 1) \beta^{-1} \pi_1 f]]  \tag*{\textbf{[L.8]}} \\
&=~  \mathsf{L}^{A} [\tau  \mathsf{L}^{B}[\tau \langle 0, 1 \times 1 \rangle \pi_1 f]]  \\
&=~  \mathsf{L}^{A} [\tau  \mathsf{L}^{B}[\tau \pi_1 f]]  \\
&=~ \mathsf{L}_1[\mathsf{L}_0[f]] 
\end{align*}
So we conclude that $\mathsf{L}_0[\mathsf{L}_1[f]] =  \mathsf{L}_1[\mathsf{L}_0[f]]$. 
\end{proof}

The following lemma will be useful in the proofs of Proposition \ref{L7Prop} and Proposition \ref{LDprop}: 

\begin{lemma}\label{linclem1} In a Cartesian left additive category with a system of linearizing combinators $\mathsf{L}^{C}$,
\begin{enumerate}[{\em (i)}]
\item \label{linclem1.i} For every map $h: C \to C^\prime$, $\mathsf{L}^C[\pi_0h] = 0$;
\item \label{linclem1.ii} For every map $f: (C \times A) \times (B \times D) \to E$, $\ell~ \mathsf{L}^{C \times A}[f] = \mathsf{L}^C[\ell f]$;
\item  \label{linclem1.iii} For every map $f: C \times A \to B$, $\oplus_{C\times A}\mathsf{L}^C[f] = c~ \mathsf{L}^{C \times C}\left[ (\oplus_C \times \oplus_A)f \right]$
\end{enumerate}
\end{lemma}
\begin{proof} For (\ref{linclem1.i}), first note that by Lemma \ref{addlem2}.(\ref{addlem2.i}), $\pi_0 h$ is constant in the simple slice category. Then since $\mathsf{L}^C$ is a linearizing combinator, by Lemma \ref{linprop1}.(\ref{linprop1.i}), it follows that $\mathsf{L}^C[\pi_0h] = 0$. For (\ref{linclem1.ii}), recall that $\ell = \langle 1,0 \rangle \times \langle 0,1 \rangle$. By Lemma \ref{lstable-lem}.(\ref{lstable-lem.1}), (\ref{lstable-lem.0}) and (\ref{lstable-lem.pair}), $\langle 0,1 \rangle$ is $\mathsf{L}$-linear, and so (\ref{linclem1.ii}) is simply an application of Proposition \ref{lemL1}.(\ref{lemL1.iii}). For (\ref{linclem1.iii}), recall that $\oplus_A = \pi_0 + \pi_1$ and so by Lemma  \ref{lstable-lem}, $\oplus_A$ is $\mathsf{L}$-linear. Note that by Lemma \ref{opluslem0}.(\ref{opluslem2}) that $\oplus_{C\times A} = c (\oplus_C \times \oplus_A)$, and so (\ref{linclem1.iii}) is simply an application of Proposition \ref{lemL1}.(\ref{lemL1.iii}). 
\end{proof} 

As previously discussed, it may be tempting to assume that from a linearizing combinator $\mathsf{L}$ on the base category, one should be able to define the linearizing combinator in context $\mathsf{L}^C$ by doing the same evaluate at zero trick as for differential combinators. This however does not work. Instead, in order to prove the converse of Proposition \ref{lemL1}, we will require the extra assumption that our Cartesian left additive category be Cartesian closed, which we discuss in Section \ref{closed-sec}. 

Our next observation is that \textbf{[L.7]} can equivalently be stated in a more compact way as {\bf [L.{7}.a]} below, using the canonical interchange isomorphism. This equivalent version will be more useful in the proofs of Proposition \ref{DLprop2} and Proposition \ref{LDprop}, while on the other hand, \textbf{[L.7]} is somewhat more intuitive and will be more useful in Section \ref{closed-sec}. The proof that {\bf [L.{7}]} and {\bf [L.{7}.a]} are equivalent include probably the ``nastiest'' calculations in this paper. 

\begin{proposition}\label{L7Prop} In the presence of the other axioms \textbf{[L.1]}-\textbf{[L.6]} and \textbf{[L.8]}, \textbf{[L.7]} is equivalent to the following: 
\begin{enumerate}[{\bf [L.{7}.a]}]
\item For a map $f: (C \times A) \times (B \times D) \to E$, $c~ \mathsf{L}^{C\times B}\left[c~ \mathsf{L}^{C\times A}[f] \right]   = \mathsf{L}^{C\times A}\left[c~ \mathsf{L}^{C\times B}[c~ f] \right]$ 
\end{enumerate}
where recall that $c$ is the canonical natural interchange isomorphism as defined in (\ref{cdef}).
\end{proposition} 
\begin{proof} Suppose that \textbf{[L.7]} holds. Then for any $f: (C \times A) \times (B \times D) \to E$, we compute that: 
\begin{align*}
&c~ \mathsf{L}^{C\times B}\left[c~ \mathsf{L}^{C\times A}[f] \right] =~ \\
&=~ c\left( (1 \times 1) \times (1 \times 1) \right) \mathsf{L}^{C\times B}\left[c~ \mathsf{L}^{C\times A}[f] \right] \\
&=~ c\left( \left( (1 \times 1) \times (1 \times 0) \right) + \left( (1 \times 1) \times (0 \times 1) \right) \right) \mathsf{L}^{C\times B}\left[c~ \mathsf{L}^{C\times A}[f] \right] \\
&=~c \left( \left( (1 \times 1) \times (1 \times 0) \right)  \mathsf{L}^{C\times B}\left[c~ \mathsf{L}^{C\times A}[f] \right] +  \left( (1 \times 1) \times (0 \times 1) \right)  \mathsf{L}^{C\times B}\left[c~ \mathsf{L}^{C\times A}[f] \right] \right) \tag*{\textbf{[L.2]}} \\
&=~c\left( (1 \times 1) \times (1 \times 0) \right)  \mathsf{L}^{C\times B}\left[c~ \mathsf{L}^{C\times A}[f] \right] +  c \left( (1 \times 1) \times (0 \times 1) \right)  \mathsf{L}^{C\times B}\left[c~ \mathsf{L}^{C\times A}[f] \right] \\
&=~ c\left( (1 \times 1) \times (1 \times 0) \right)  \mathsf{L}^{C\times B}\left[c\left( (1 \times 1) \times (1 \times 1) \right) \mathsf{L}^{C\times A}[f] \right] \\
&+~ c\left( (1 \times 1) \times (0 \times 1) \right)  \mathsf{L}^{C\times B}\left[c\left( (1 \times 1) \times (1 \times 1) \right) \mathsf{L}^{C\times A}[f] \right] \\
&=~ c\left( (1 \times 1) \times (1 \times 0) \right)  \mathsf{L}^{C\times B}\left[c\left( \left( (1 \times 1) \times (1 \times 0) \right) + \left( (1 \times 1) \times (0 \times 1) \right) \right) \mathsf{L}^{C\times A}[f] \right] \\
&+~ c\left( (1 \times 1) \times (0 \times 1) \right)  \mathsf{L}^{C\times B}\left[c\left( \left( (1 \times 1) \times (1 \times 0) \right) + \left( (1 \times 1) \times (0 \times 1) \right) \right) \mathsf{L}^{C\times A}[f] \right] \\
&=~ c\left( (1 \times 1) \times (1 \times 0) \right)  \mathsf{L}^{C\times B}\left[c \left( \left( (1 \times 1) \times (1 \times 0) \right) \mathsf{L}^{C\times A}[f] +  \left( (1 \times 1) \times (0 \times 1) \right) \mathsf{L}^{C\times A}[f]  \right) \right] \\
&+~ c \left( (1 \times 1) \times (0 \times 1) \right)  \mathsf{L}^{C\times B}\!\!\left[c \left( \left( (1 \times 1) \times (1 \times 0) \right) \mathsf{L}^{C\times A}[f] +  \left( (1 \times 1) \times (0 \times 1) \right) \mathsf{L}^{C\times A}[f]  \right) \right] \\
&=~ c\left( (1 \times 1) \times (1 \times 0) \right)  \mathsf{L}^{C\times B}\left[c \left( (1 \times 1) \times (1 \times 0) \right) \mathsf{L}^{C\times A}[f] +  c\left( (1 \times 1) \times (0 \times 1) \right) \mathsf{L}^{C\times A}[f] \right] \\
&+~ c \left( (1 \times 1) \times (0 \times 1) \right)  \mathsf{L}^{C\times B}\left[c  \left( (1 \times 1) \times (1 \times 0) \right) \mathsf{L}^{C\times A}[f] + c \left( (1 \times 1) \times (0 \times 1) \right) \mathsf{L}^{C\times A}[f]  \right] \\
&=~ c\left( (1 \times 1) \times (1 \times 0) \right) \!\!\left(\!\!\mathsf{L}^{C\times B}\!\!\left[c \left( (1 \times 1) \!\times\! (1 \times 0) \right) \mathsf{L}^{C\times A}[f] \right]  + \mathsf{L}^{C\times B}\!\!\left[c\left( (1 \times 1) \!\times\! (0 \times 1) \right) \mathsf{L}^{C\times A}[f] \right] \!\!\right)  \tag*{\textbf{[L.2]}} \\
&+~  c\left( (1 \times 1) \times (0 \times 1) \right)\!\!\left(\!\!\mathsf{L}^{C\times B}\!\!\left[c \left( (1 \times 1) \!\times\! (1 \times 0) \right) \mathsf{L}^{C\times A}[f] \right]  + \mathsf{L}^{C\times B}\!\!\left[c\left( (1 \times 1) \!\times\! (0 \times 1) \right) \mathsf{L}^{C\times A}[f] \right] \!\!\right)  \tag*{\textbf{[L.2]}} \\
&=~  c\left( (1 \times 1) \times (1 \times 0) \right) \mathsf{L}^{C\times B}\left[c \left( (1 \times 1) \times (1 \times 0) \right) \mathsf{L}^{C\times A}[f] \right] \\
&+~ c\left( (1 \times 1) \times (1 \times 0) \right) \mathsf{L}^{C\times B}\left[c\left( (1 \times 1) \times (0 \times 1) \right) \mathsf{L}^{C\times A}[f] \right] \\
&+~ c\left( (1 \times 1) \times (0 \times 1) \right) \mathsf{L}^{C\times B}\left[c \left( (1 \times 1) \times (1 \times 0) \right) \mathsf{L}^{C\times A}[f] \right] \\
&+~  c\left( (1 \times 1) \times (0 \times 1) \right) \mathsf{L}^{C\times B}\left[c\left( (1 \times 1) \times (0 \times 1) \right) \mathsf{L}^{C\times A}[f] \right] \\
&=~  c\left( (1 \times 1) \times (1 \times 0) \right) \mathsf{L}^{C\times B}\left[c~  \mathsf{L}^{C\times A}[\left( (1 \times 1) \times (1 \times 0) \right) f] \right]  \tag{Lem.\ref{lstable-lem}.(\ref{lstable-lem.1})+(\ref{lstable-lem.0})+(\ref{lstable-lem.prod}) + Prop.\ref{lemL1}.(\ref{lemL1.iii})} \\
&+~ c~\mathsf{L}^{C\times B}\left[\left( (1 \times 1) \times (1 \times 0) \right) c\left( (1 \times 1) \times (0 \times 1) \right) \mathsf{L}^{C\times A}[f] \right] \tag{Lem.\ref{lstable-lem}.(\ref{lstable-lem.1})+(\ref{lstable-lem.0})+(\ref{lstable-lem.prod}) + Prop.\ref{lemL1}.(\ref{lemL1.iii})} \\
&+~ c\left( (1 \times 1) \times (0 \times 1) \right) \mathsf{L}^{C\times B}\left[\left( (1 \times 1) \times (1 \times 0) \right) c~  \mathsf{L}^{C\times A}[f] \right] \tag{Nat. of $c$} \\
&+~  c~\mathsf{L}^{C\times B}\left[\left( (1 \times 1) \times (0 \times 1) \right) c\left( (1 \times 1) \times (0 \times 1) \right) \mathsf{L}^{C\times A}[f] \right]  \tag{Lem.\ref{lstable-lem}.(\ref{lstable-lem.1})+(\ref{lstable-lem.0})+(\ref{lstable-lem.prod}) + Prop.\ref{lemL1}.(\ref{lemL1.iii})} \\
&=~  c~\mathsf{L}^{C\times B}\left[\left( (1 \times 1) \times (1 \times 0) \right) c~  \mathsf{L}^{C\times A}[\left( (1 \times 1) \times (1 \times 0) \right) f] \right]  \tag{Lem.\ref{lstable-lem}.(\ref{lstable-lem.1})+(\ref{lstable-lem.0})+(\ref{lstable-lem.prod}) + Prop.\ref{lemL1}.(\ref{lemL1.iii})} \\
&+~ c~\mathsf{L}^{C\times B}\left[c \left( (1 \times 1) \times (1 \times 0) \right) \left( (1 \times 1) \times (0 \times 1) \right) \mathsf{L}^{C\times A}[f] \right] \tag{Nat. of $c$} \\
&+~ c\left( (1 \times 1) \times (0 \times 1) \right)\left( (1 \times 1) \times (1 \times 0) \right)  \mathsf{L}^{C\times B}\left[c~  \mathsf{L}^{C\times A}[f] \right] \tag{Lem.\ref{lstable-lem}.(\ref{lstable-lem.1})+(\ref{lstable-lem.0})+(\ref{lstable-lem.prod}) + Prop.\ref{lemL1}.(\ref{lemL1.iii})} \\
&+~  c~\mathsf{L}^{C\times B}\left[ c \left( (1 \times 0) \times (1 \times 1) \right) \left( (1 \times 1) \times (0 \times 1) \right) \mathsf{L}^{C\times A}[f] \right]  \tag{Nat. of $c$} \\
&=~  c~\mathsf{L}^{C\times B}\left[c \left( (1 \times 1) \times (1 \times 0) \right)   \mathsf{L}^{C\times A}[\left( (1 \times 1) \times (1 \times 0) \right) f] \right]  \tag{Nat. of $c$} \\
&+~ c~\mathsf{L}^{C\times B}\left[c \left( (1 \times 1) \times (0 \times 0) \right) \mathsf{L}^{C\times A}[f] \right] \\
&+~ c\left( (1 \times 1) \times (0 \times 0) \right) \mathsf{L}^{C\times B}\left[c~  \mathsf{L}^{C\times A}[f] \right] \\
&+~  c~\mathsf{L}^{C\times B}\left[ c \left( (1 \times 0) \times (0 \times 1) \right) \mathsf{L}^{C\times A}[f] \right]  \\
&=~  c~\mathsf{L}^{C\times B}\left[c~ \mathsf{L}^{C\times A}[ \left( (1 \times 1) \times (1 \times 0) \right)  \left( (1 \times 1) \times (1 \times 0) \right) f] \right]  \tag{Lem.\ref{lstable-lem}.(\ref{lstable-lem.1})+(\ref{lstable-lem.0})+(\ref{lstable-lem.prod}) + Prop.\ref{lemL1}.(\ref{lemL1.iii})} \\
&+~ c~\mathsf{L}^{C\times B}\left[c 0 \right] \tag*{\textbf{[L.2]}} \\ 
&+~ c 0  \tag*{\textbf{[L.2]}} \\ 
&+~  c~\mathsf{L}^{C\times B}\left[ (\pi_0 \times \pi_1) \ell \mathsf{L}^{C\times A}[f] \right]  \\ 
&=~  c~\mathsf{L}^{C\times B}\left[c~ \mathsf{L}^{C\times A}[  \left( (1 \times 1) \times (1 \times 0) \right) f] \right] + c~\mathsf{L}^{C\times B}\left[0 \right] + 0 +  c~ (\pi_0 \times \pi_1)\mathsf{L}^{C}\left[ \mathsf{L}^{C}[\ell f] \right]  \tag{Lem.\ref{lstable-lem}.(\ref{lstable-lem.pi}) + Prop.\ref{lemL1}.(\ref{lemL1.iii}) + Lem.\ref{linclem1}.(\ref{linclem1.iii})} \\
&=~  c~\mathsf{L}^{C\times B}\left[c~ \mathsf{L}^{C\times A}[ \left( (1 \times 1) \times \pi_0 \right) \left( (1 \times 1) \times \langle 1, 0 \rangle \right)  f] \right] + c~0 +  c~ (\pi_0 \times \pi_1) \mathsf{L}^{C}\left[ \mathsf{L}^{C}[\ell f] \right] \tag*{\textbf{[L.1]} } \\
&=~  c~\mathsf{L}^{C\times B}\left[c~ \mathsf{L}^{C\times A}[ \left( (1 \times 1) \times \pi_0 \right) \left( (1 \times 1) \times \langle 1, 0 \rangle \right)  f] \right] + 0 +  (\pi_0 \times \pi_1) \mathsf{L}^{C}\left[ \mathsf{L}^{C}[\ell f] \right] \tag{$c (\pi_0 \times \pi_1) = \pi_0 \times \pi_1$}  \\
&=~  c~\mathsf{L}^{C\times B}\left[c \left( (1 \times 1) \times \pi_0 \right)  \mathsf{L}^{C\times A}[ \left( (1 \times 1) \times \langle 1, 0 \rangle \right)  f] \right]  +  (\pi_0 \times \pi_1)\mathsf{L}^{C}\left[ \mathsf{L}^{C}[\ell f] \right]  \tag{Lem.\ref{lstable-lem}.(\ref{lstable-lem.1})+(\ref{lstable-lem.pi})+(\ref{lstable-lem.prod}) + Prop.\ref{lemL1}.(\ref{lemL1.iii})} \\
&=~  c~\mathsf{L}^{C\times B}\left[c \left( (1 \times 1) \times \pi_0 \right)  \mathsf{L}^{C\times A}[ \left( (1 \times 1) \times \langle 1, 0 \rangle \right)  f] \right] + (\pi_0 \times \pi_1) \mathsf{L}^{C}[\ell f]  \tag*{\textbf{[L.6]}} \\
&=~ c~\mathsf{L}^{C\times B}\left[\left( (1 \times 1) \times \pi_0 \right) \beta^{-1} \alpha \mathsf{L}^{C\times A}[ \left( (1 \times 1) \times \langle 1, 0 \rangle \right)  f] \right] +  (\pi_0 \times \pi_1)\mathsf{L}^{C}[\ell f] \\
&=~ c\left( (1 \times 1) \times \pi_0 \right) \mathsf{L}^{C\times B}\left[\beta^{-1} \alpha \mathsf{L}^{C\times A}[ \left( (1 \times 1) \times \langle 1, 0 \rangle \right)  f] \right] + (\pi_0 \times \pi_1) \mathsf{L}^{C}[\ell f] \tag{Lem.\ref{lstable-lem}.(\ref{lstable-lem.pi}) + Prop.\ref{lemL1}.(\ref{lemL1.iii})} \\
&=~ \left( (1 \times 1) \times \pi_0 \right) \alpha^{-1} \beta \mathsf{L}^{C\times B}\left[\beta^{-1} \alpha \mathsf{L}^{C\times A}[ \left( (1 \times 1) \times \langle 1, 0 \rangle \right)  f] \right] + (\pi_0 \times \pi_1) \mathsf{L}^{C}[\ell f]  \\
&=~ \left( (1 \times 1) \times \pi_0 \right) \alpha^{-1} \beta \mathsf{L}^{C\times B}\left[\beta^{-1} \alpha \mathsf{L}^{C\times A}[\alpha^{-1} \alpha \left( (1 \times 1) \times \langle 1, 0 \rangle \right)  f] \right] + (\pi_0 \times \pi_1) \mathsf{L}^{C}[\ell f]  \\
&=~ \left( (1 \times 1) \times \pi_0 \right) \alpha^{-1} \mathsf{L}^{C}_0\left[ \mathsf{L}^{C}_1[\alpha \left( (1 \times 1) \times \langle 1, 0 \rangle \right)  f] \right] + (\pi_0 \times \pi_1) \mathsf{L}^{C}[\ell f]  
\end{align*}
So we have the following equality: 
\begin{equation}\label{L74}\begin{gathered}  
c~ \mathsf{L}^{C\times B}\left[c~ \mathsf{L}^{C\times A}[f] \right] = \left( (1 \times 1) \times \pi_0 \right) \alpha^{-1} \mathsf{L}^{C}_0\left[ \mathsf{L}^{C}_1[\alpha \left( (1 \times 1) \times \langle 1, 0 \rangle \right)  f] \right] +  (\pi_0 \times \pi_1)\mathsf{L}^{C}[\ell f] 
 \end{gathered}\end{equation}
 On the other hand, using the above equality and that $\ell c = \ell$, we compute that: 
 \begin{align*}
&\mathsf{L}^{C\times A}\left[c~ \mathsf{L}^{C\times B}[c~ f] \right] =~ cc~ \mathsf{L}^{C\times A}\left[c~ \mathsf{L}^{C\times B}[c~ f] \right] \tag{$c$ is self-inverse} \\
&=~ c \left( \left( (1 \times 1) \times \pi_0 \right) \alpha^{-1} \mathsf{L}^{C}_0\left[ \mathsf{L}^{C}_1[\alpha \left( (1 \times 1) \times \langle 1, 0 \rangle \right)  c~ f] \right] +   (\pi_0 \times \pi_1)\mathsf{L}^{C}[\ell c~ f] \right) \tag{(\ref{L74}) for $cf$}\\
&=~ c \left( \left( (1 \times 1) \times \pi_0 \right) \alpha^{-1} \mathsf{L}^{C}_0\left[ \mathsf{L}^{C}_1[\alpha \left( (1 \times 1) \times \langle 1, 0 \rangle \right)  c~ f] \right] +   (\pi_0 \times \pi_1)\mathsf{L}^{C}[\ell f] \right) \tag{$\ell c = \ell$} \\ 
&=~ c \left( (1 \times 1) \times \pi_0 \right) \alpha^{-1} \mathsf{L}^{C}_0\left[ \mathsf{L}^{C}_1[\alpha \left( (1 \times 1) \times \langle 1, 0 \rangle \right)  c~ f] \right] + c (\pi_0 \times \pi_1) \mathsf{L}^{C}[\ell f]  \\
&=~ c \left( (1 \times 1) \times \pi_0 \right) \alpha^{-1} \mathsf{L}^{C}_0\left[ \mathsf{L}^{C}_1[\alpha \left( (1 \times 1) \times \langle 1, 0 \rangle \right)  c~ f] \right] + (\pi_0 \times \pi_1) \mathsf{L}^{C}[\ell f] \tag{$c (\pi_0 \times \pi_1) = \pi_0 \times \pi_1$}  \\
&=~ \left( (1 \times 1) \times \pi_0 \right) \beta^{-1} \mathsf{L}^{C}_0\left[ \mathsf{L}^{C}_1[\beta \left( (1 \times 1) \times \langle 1, 0 \rangle \right) f] \right] +  (\pi_0 \times \pi_1)\mathsf{L}^{C}[\ell f]  \\
&=~ \left( (1 \times 1) \times \pi_0 \right) \beta^{-1} \mathsf{L}^{C}_1\left[ \mathsf{L}^{C}_0[\beta \left( (1 \times 1) \times \langle 1, 0 \rangle \right) f] \right] +  (\pi_0 \times \pi_1)\mathsf{L}^{C}[\ell f]  \tag*{\textbf{[L.7]}} \\ 
&=~ \left( (1 \times 1) \times \pi_0 \right) \beta^{-1} \alpha \mathsf{L}^{C \times B}\left[\alpha^{-1} \beta \mathsf{L}^{C \times A}[\beta^{-1}\beta \left( (1 \times 1) \times \langle 1, 0 \rangle \right) f] \right] +  (\pi_0 \times \pi_1)\mathsf{L}^{C}[\ell f] \\ 
&=~ \left( (1 \times 1) \times \pi_0 \right) \alpha^{-1} \beta  \mathsf{L}^{C \times B}\left[\beta^{-1} \alpha  \mathsf{L}^{C \times A}[ \left( (1 \times 1) \times \langle 1, 0 \rangle \right) f] \right] +  (\pi_0 \times \pi_1)\mathsf{L}^{C}[\ell f] \\ 
&=~ \left( (1 \times 1) \times \pi_0 \right) \alpha^{-1} \beta  \mathsf{L}^{C \times B}\left[\beta^{-1} \alpha  \mathsf{L}^{C \times A}[\alpha^{-1} \alpha \left( (1 \times 1) \times \langle 1, 0 \rangle \right) f] \right] +  (\pi_0 \times \pi_1)\mathsf{L}^{C}[\ell f] \\ 
&=~ \left( (1 \times 1) \times \pi_0 \right) \alpha^{-1} \mathsf{L}^{C}_0\left[ \mathsf{L}^{C}_1[\alpha \left( (1 \times 1) \times \langle 1, 0 \rangle \right)  f] \right] +  (\pi_0 \times \pi_1)\mathsf{L}^{C}[\ell f] \\
&=~ c~ \mathsf{L}^{C\times B}\left[c~ \mathsf{L}^{C\times A}[f] \right] \tag{\ref{L74}}
\end{align*}
So we conclude that \textbf{[L.7.a]} holds. 

Conversely, suppose that \textbf{[L.7.a]} holds. For a map $f: C \times (A \times B) \to D$, define the map $f^\circ: (C \times A) \times (B \times \top) \to D$ as the following composite: 
  \[ f^\circ := \xymatrixcolsep{5pc}\xymatrix{(C \times A) \times (B \times \top) \ar[r]^-{(1 \times 1) \times \pi_0} & (C \times A) \times B \ar[r]^-{\alpha^{-1}} & C \times (A \times B) \ar[r]^-{f} & D  
  } \]
First recall that for the terminal object $\top$, $\pi_0: X \times \top \to X$ is an isomorphism with inverse ${\langle 1,0\rangle: X \to X \times \top}$. Therefore, we have the following equality: 
  \begin{equation}\label{L71}\begin{gathered} 
  f = \alpha  \left( (1 \times 1) \times \langle 1, 0 \rangle \right) f^\circ 
 \end{gathered}\end{equation}
We also have the following equalities (which we leave to the reader to check for themselves): 
 \begin{equation}\label{L72}\begin{gathered} 
 \beta^{-1} \alpha \left( (1 \times 1) \times \langle 1, 0 \rangle \right) =  \left( (1 \times 1) \times \langle 1, 0 \rangle \right) c = \alpha^{-1} \beta  \left( (1 \times 1) \times \langle 1, 0 \rangle \right)
 \end{gathered}\end{equation}
 Lastly, by Proposition \ref{lemL1}.(\ref{lemL1.iii}) (which did not require \textbf{[L.7]} to prove), one can show that for any map $g: (C \times X) \times (Y \times \top) \to D$, the following equality holds: 
 \begin{equation}\label{L73}\begin{gathered} 
 \mathsf{L}^{C \times X}\left[ \left( (1 \times 1) \times \langle 1, 0 \rangle \right) g \right] = \left( (1 \times 1) \times \langle 1, 0 \rangle \right) \mathsf{L}^{C \times X}\left[ g \right] 
 \end{gathered}\end{equation}
 Therefore, we compute that: 
 \begin{align*}
 \mathsf{L}^C_0[\mathsf{L}^C_1[f]] &=~ \beta \mathsf{L}^{C \times B}\left[\beta^{-1} \alpha \mathsf{L}^{C \times A}[\alpha^{-1} f]  \right] \\
 &=~ \beta \mathsf{L}^{C \times B}\left[\beta^{-1} \alpha \mathsf{L}^{C \times A}\left[\left( (1 \times 1) \times \langle 1, 0 \rangle \right) f^\circ  \right]  \right] \tag{\ref{L71}} \\
  &=~ \beta \mathsf{L}^{C \times B}\left[\beta^{-1} \alpha \left( (1 \times 1) \times \langle 1, 0 \rangle \right) \mathsf{L}^{C \times A}[ f^\circ  ]  \right] \tag{\ref{L73}} \\
    &=~ \beta \mathsf{L}^{C \times B}\left[ \left( (1 \times 1) \times \langle 1, 0 \rangle \right)c~ \mathsf{L}^{C \times A}[ f^\circ  ]  \right] \tag{\ref{L72}} \\
       &=~ \beta  \left( (1 \times 1) \times \langle 1, 0 \rangle \right) \mathsf{L}^{C \times B}\left[c~ \mathsf{L}^{C \times A}[ f^\circ  ]  \right] \tag{\ref{L73}} \\
              &=~ \alpha \left( (1 \times 1) \times \langle 1, 0 \rangle \right) c~ \mathsf{L}^{C \times B}\left[c~ \mathsf{L}^{C \times A}[ f^\circ  ]  \right] \tag{\ref{L72}} \\ 
&=~ \alpha \left( (1 \times 1) \times \langle 1, 0 \rangle \right)  \mathsf{L}^{C \times A}\left[c~ \mathsf{L}^{C \times B}[c f^\circ  ]  \right] \tag*{\textbf{[L.7.a]}} \\ 
&=~ \alpha  \mathsf{L}^{C \times A}\left[ \left( (1 \times 1) \times \langle 1, 0 \rangle \right)  c~ \mathsf{L}^{C \times B}[c f^\circ  ]  \right]\tag{\ref{L73}} \\
&=~ \alpha  \mathsf{L}^{C \times A}\left[\alpha^{-1} \beta \left( (1 \times 1) \times \langle 1, 0 \rangle \right)   \mathsf{L}^{C \times B}[c f^\circ  ]  \right]\tag{\ref{L72}} \\
&=~ \alpha  \mathsf{L}^{C \times A}\left[\alpha^{-1} \beta    \mathsf{L}^{C \times B}\left[ \left( (1 \times 1) \times \langle 1, 0 \rangle \right) c f^\circ  \right]   \right]\tag{\ref{L73}}\\
&=~ \alpha  \mathsf{L}^{C \times A}\left[\alpha^{-1} \beta    \mathsf{L}^{C \times B}\left[ \beta^{-1} \alpha \left( (1 \times 1) \times \langle 1, 0 \rangle \right) f^\circ  \right]   \right]\tag{\ref{L72}}\\
&=~ \alpha  \mathsf{L}^{C \times A}\left[\alpha^{-1} \beta    \mathsf{L}^{C \times B}\left[ \beta^{-1} f \right]   \right]\tag{\ref{L71}} \\
&=~  \mathsf{L}^C_1[\mathsf{L}^C_0[f]] 
\end{align*}
So we conclude that \textbf{[L.7]} holds. 
\end{proof}  

We now turn our attention to the relationship between differential combinators and systems of linearizing combinators. We first show that every differential combinator induces a system of linearizing combinators. Indeed, since every simple slice category of a Cartesian differential category is again Cartesian differential category, and every Cartesian differential category comes equipped with a canonical linearizing combinator, it follows that every Cartesian differential category admits a system of linearizing combinators. 

\begin{proposition}\label{DLprop2} Every Cartesian differential category, with differential combinator $\mathsf{D}$, admits a system of linearizing combinators where the linearizing combinators $\mathsf{L}_{\mathsf{D}^{C }}$ for the simple slice categories are defined as in Proposition \ref{DLprop}. As to not overload the subscripts, we denote this linearizing combinator instead by $\mathsf{L}^{C}_{\mathsf{D}} := \mathsf{L}_{\mathsf{D}^{C }}$. Equivalently, $\mathsf{L}^{C}_{\mathsf{D}}$ is defined as follows on a map $f: C \times A \to B$: 
\begin{equation}\label{contextlin}\begin{gathered} \mathsf{L}^{C}_{\mathsf{D}}[f] = \xymatrixcolsep{5pc}\xymatrix{C \times A \ar[r]^-{\ell} & (C \times A) \times (C \times A) \ar[r]^-{\mathsf{D}[f]} & B } \end{gathered}\end{equation}
where $\ell$ is the lifting map as defined as in (\ref{ldef}). Furthermore, 
\begin{enumerate}[{\em (i)}]
\item \label{DLprop2.i} For every map $f: C \times A \to B$, $\mathsf{L}^{C}_\mathsf{D}[f]$ is linear in its second argument;
\item \label{DLprop2.ii} A map $f: C \times A \to B$ is linear in its second argument if and only if $f$ is $\mathsf{L}^{C}_\mathsf{D}$-linear. 
\item \label{DLprop2.iii} $\mathsf{L} = \mathsf{L}_\mathsf{D}$, where $\mathsf{L}$ is the induced linearizing combinator from Proposition \ref{lemL1} and $\mathsf{L}_\mathsf{D}$ is the induced linearizing combinator from Proposition \ref{DLprop}. 
\end{enumerate}
\end{proposition} 
\begin{proof} By Proposition \ref{Dcontext}, every simple slice category of a Cartesian differential category is again a Cartesian differential category with differential combinator $\mathsf{D}^{C }[f]$. Then by applying Proposition \ref{DLprop} to the simple slice categories, we obtain a linearizing combinator $\mathsf{L}^{C}_{\mathsf{D}}$ for each simple slice category. So $\mathsf{L}^{C}_{\mathsf{D}}$ satisfies {\bf [L.1]} through {\bf [L.6]} as in Definition \ref{syslindef}. Furthermore, since a map of type $C \times A \to B$ is linear in its second argument if it linear in the simple slice category, it follows from Proposition \ref{DLprop}.(\ref{DLprop.1}) that for every map $f: C \times A \to B$, $\mathsf{L}^{C}_\mathsf{D}[f]$ is linear in its second argument. Similarly, by Proposition \ref{DLprop}.(\ref{DLprop.2}), a map $f: C \times A \to B$ is linear in its second argument (i.e. $\mathsf{D}^{C }[f]=(1 \times \pi_1) f$) if and only if $f$ is $\mathsf{L}^{C}_\mathsf{D}$-linear (i.e. $\mathsf{L}^{C}_\mathsf{D}[f] = f$). 

 For a map $f: C \times A \to B$, by Proposition \ref{Dcontext} and by definition of composition in the simple slice category, $\mathsf{L}^{C}_{\mathsf{D}}[f]: C \times A \to B$ is easily worked out to be:
\begin{equation}\label{contextlin2}\begin{gathered} \mathsf{L}^{C}_{\mathsf{D}}[f] = \left \langle \pi_0, \langle 0, \pi_1 \rangle \right \rangle \mathsf{D}^{C }[f] \end{gathered}\end{equation}
Which can equivalently be rewritten as: 
\begin{equation}\label{contextlin1}\begin{gathered} \mathsf{L}^{C}_{\mathsf{D}}[f] = (1 \times \langle 0,1 \rangle) \mathsf{D}^{C }[f] \end{gathered}\end{equation}
Expanding out the definition of $\mathsf{D}^{C }[f]$, we obtain: 
\begin{align*} \mathsf{L}^{C}_{\mathsf{D}}[f] &=~ (1 \times \langle 0,1 \rangle) \mathsf{D}^{C }[f] \tag{\ref{contextlin1}} \\
&=~  (1 \times \langle 0,1 \rangle) \langle 1 \times \pi_0, 0 \times \pi_1 \rangle \mathsf{D}[f]  \tag{\ref{DCdef}} \\
&=~ \left \langle (1 \times \langle 0,1 \rangle)(1 \times \pi_0), (1 \times \langle 0,1 \rangle)(0 \times \pi_1) \right \rangle \mathsf{D}[f] \\
&=~ \left \langle 1 \times 0, 0 \times 1 \right \rangle \mathsf{D}[f] \\
&=~ \left( \langle 1, 0 \rangle \times \langle 0,1 \rangle \right) \mathsf{D}[f] \\
&=~ \ell ~ \mathsf{D}[f] 
\end{align*}
So we have that $\mathsf{L}^{C}_{\mathsf{D}}[f] = \ell \mathsf{D}[f]$. We now show that {\bf [L.8]} and {\bf [L.7.a]} also hold (which recall by Proposition \ref{L7Prop} is equivalent to showing that \textbf{[L.7]} holds) : \\\\
\noindent \textbf{[L.8]}: $(h \times 1)\mathsf{L}_{\mathsf{D}}^{C^\prime}[f]  = \mathsf{L}_{\mathsf{D}}^C[(h \times 1)f]$
\begin{align*}
(h\times 1) \mathsf{L}_{\mathsf{D}}^{C^\prime}[f] &=~ (h \times 1) (1 \times \langle 0,1 \rangle) \mathsf{D}^{C^\prime }[f]\tag{\ref{contextlin1}} \\
&=~  (1 \times \langle 0,1 \rangle) (h \times 1)\mathsf{D}^{C^\prime }[f] \\
&=~(1 \times \langle 0,1 \rangle) \mathsf{D}^{C}[(h \times 1)f] \tag{Proposition \ref{Dcontext}} \\
&=~  \mathsf{L}_{\mathsf{D}}^C[(h \times 1) f] \tag{\ref{contextlin1}}
\end{align*}
\noindent \textbf{[L.7.a]}: $c~ \mathsf{L}^{C\times B}_{\mathsf{D}}\left[c~ \mathsf{L}^{C\times A}_{\mathsf{D}}[f] \right]   = \mathsf{L}^{C\times A}_{\mathsf{D}}\left[c~ \mathsf{L}^{C\times B}_{\mathsf{D}}[c~ f] \right]$ \\\\
We leave it to the reader to check for themselves that the following equality holds (which can be checked by a straightforward but tedious calculation): 
\begin{equation}\label{clidentity}\begin{gathered} c\ell (c \times c) (\ell \times \ell) c = \ell (c \times c) (\ell \times \ell) \left( (c \times c) \times (c \times c) \right) \end{gathered}\end{equation}
Then we have that: 
\begin{align*} c~ \mathsf{L}^{C\times B}_{\mathsf{D}}\left[c~ \mathsf{L}^{C\times A}_{\mathsf{D}}[f] \right] &=~c \ell ~\mathsf{D}\left [ c \ell \mathsf{D}[f] \right] \tag{\ref{contextlin}} \\ 
&=~ c\ell (c \times c) (\ell \times \ell) ~\mathsf{D}\left [  \mathsf{D}[f] \right] \tag{Cor. \ref{corlin}.(\ref{corlin.c})+(\ref{corlin.ell}) + Lem. \ref{linlem}.(\ref{linlem.pre})} \\ 
&=~ c\ell (c \times c) (\ell \times \ell) c ~ \mathsf{D}\left [  \mathsf{D}[f] \right] \tag*{\textbf{[CD.7]}} \\ 
&=~\ell (c \times c) (\ell \times \ell) \left( (c \times c) \times (c \times c) \right) \mathsf{D}\left [  \mathsf{D}[f] \right] \tag{\ref{clidentity}} \\
&=~ \ell \mathsf{D}\left [ c \ell (c \times c) \mathsf{D}[f] \right] \tag{Cor. \ref{corlin}.(\ref{corlin.c})+(\ref{corlin.ell}) + Lem. \ref{linlem}.(\ref{linlem.pre})} \\ 
&=~ \ell \mathsf{D}\left [ c \ell~ \mathsf{D}[ c f] \right] \tag{Cor. \ref{corlin}(\ref{corlin.c}) + Lem. \ref{linlem}.(\ref{linlem.pre})} \\ 
&=~\mathsf{L}^{C\times A}_{\mathsf{D}}\left[c~ \mathsf{L}^{C\times B}_{\mathsf{D}}[c~ f] \right]
\end{align*}
We conclude that every Cartesian differential category has a system of linearizing combinators. We now show that the linearizing combinators from Proposition \ref{lemL1} and Proposition \ref{DLprop} are the same:  
\begin{align*}
\mathsf{L}[f] &=~ \langle 0,1 \rangle \pi_1 \mathsf{L}[f] \\
&=~  \langle 0,1 \rangle \mathsf{L}^C_\mathsf{D}[\pi_1 f] \tag{Proposition \ref{lemL1}.(\ref{lemL1.i})} \\
&=~ \langle 0,1 \rangle \ell \mathsf{D}[\pi_1 f]  \tag{\ref{contextlin}} \\ 
&=~ \langle 0,1 \rangle \ell(\pi_1 \times \pi_1) \mathsf{D}[f] \tag{Lem \ref{linlem}.(\ref{linlem.pre})+(\ref{linlem.pi})} \\
&=~  \langle 0,1 \rangle (\langle 1,0\rangle \times \langle 0,1 \rangle)(\pi_1 \times \pi_1) \mathsf{D}[f] \\
&=~ \langle 0,1 \rangle (0 \times 1) \mathsf{D}[f] \\
&=~ \langle 0,1 \rangle \mathsf{D}[f] \\
&=~ \mathsf{L}_\mathsf{D}[f] 
\end{align*}
So we conclude that $\mathsf{L} = \mathsf{L}_\mathsf{D}$. 
\end{proof} 

We now apply Proposition \ref{DLprop2} to the examples of Cartesian differential categories from Section \ref{CDCsec} to obtain examples of systems of linearizing combinators in context, specifically using the construction given in (\ref{contextlin}). 

 \begin{example} \normalfont In a category with finite biproducts, for a map $f: C \times A \to B$, the linearizing combinator in context $C$ is defined as evaluating $f$ at zero in its first argument: 
 \[ \mathsf{L}^C[f] = (0 \times 1) f \]
\end{example}

\begin{example} \normalfont In $\mathsf{SMOOTH}$, for a smooth function $F: \mathbb{R}^k \times \mathbb{R}^n \to \mathbb{R}^m$, $F = \langle f_1, \hdots, f_m \rangle$, its partial linearization is the smooth function $\mathsf{L}^{\mathbb{R}^k}[F]: \mathbb{R}^k \times \mathbb{R}^n \to \mathbb{R}^m$ defined as follows: 
\[ \mathsf{L}^{\mathbb{R}^k}[F](\vec z, \vec x) = \nabla(F)(\vec z, \vec 0) \cdot (\vec 0, \vec x) =\left \langle \sum \limits^n_{i=1} \frac{\partial f_1}{\partial x_i}(\vec z,\vec 0) x_i, \hdots, \sum \limits^n_{i=1} \frac{\partial f_m}{\partial x_i}(\vec z,\vec 0) x_i \right \rangle \]
%In particular, for a smooth function $f: \mathbb{R}^k \times \mathbb{R}^n \to \mathbb{R}$, its partial linearization is the sum of sum of its partial derivatives in the $\mathbb{R}^n$ variables evaluated at zero in its first $\mathbb{R}^n$:
% \[\mathsf{L}^{\mathbb{R}^k}[f](\vec x) = \sum \limits^n_{i=1} \frac{\partial f}{\partial x_i}(\vec z,\vec 0) x_i \] 
 For example consider the polynomial function $f: \mathbb{R} \times \mathbb{R} \to \mathbb{R}$ defined as $f(z,x) = z^3x + z^2x^3+ x + 1$. The partial linearization of $f$ is defined by picking out the terms which are linear in $x$, that is, $\mathsf{L}^{\mathbb{R}}[f](z,x) = z^3 x + x$.  
\end{example}

\begin{example} \normalfont For $\mathsf{HoAbCat}_\mathsf{Ch}$, the partial linearizing combinator is precisely the partial linearization operator $\mathsf{D}^i_1$ as defined in \cite[Convention 5.11]{bauer2018directional}. Explicitly, for a functor $F: C \times A \to \mathsf{Ch}(B)$, its partial linearization is $\mathsf{L}^C[F] = \mathsf{D}^1_1[F]$. 
\end{example}

\begin{example} \normalfont For a Cartesian left additive category $\mathbb{X}$, the linearizing combinator in context $C$ for its cofree Cartesian differential category $\mathcal{D}(\mathbb{X})$ is worked out to be as follows for a $\mathsf{D}$-sequence $(f_0, f_1, f_2, \hdots): C \times A \to B$ (so $f_n: \mathsf{P}^n(C \times A) \to B$): 
\[ \mathsf{L}^C\left[ (f_0, f_1, f_2, \hdots ) \right] = ( \ell f_1, \mathsf{P}(\ell) f_2,  \mathsf{P}^2(\ell) f_3, \hdots )  \]
where recall that $\mathsf{P}$ is the product functor $\mathsf{P}(-) = - \times -$. 
\end{example}

\begin{example} \normalfont For a differential category $\mathbb{X}$ with finite products, the linearizing combinator in context $C$ for the coKleisli category $\mathbb{X}_\oc$ is worked out to be as follows for a coKleisli map ${f: \oc (C \times A) \to B}$: 
\[ \begin{array}[c]{c} \mathsf{L}^C[f] \end{array} := \begin{array}[c]{c} \xymatrixrowsep{1pc}\xymatrixcolsep{5pc}\xymatrix{\oc(C \times A) \ar[r]^-{\chi_{C,A}} & \oc C \otimes \oc A \ar[r]^-{1 \otimes \varepsilon_A} & \oc C \otimes A \ar[r]^-{\oc(\langle 1, 0 \rangle) \otimes \langle 0,1 \rangle} &  \\
\oc(C \times A) \otimes (C \times A) \ar[r]^-{\mathsf{d}_{C \times A}} & \oc(C \times A) \ar[r]^-{f} & B 
 }  \end{array}\] 
 where recall that $\chi_{C,A} = \Delta_{C \times A}(\oc(\pi_0) \otimes \oc(\pi_1))$. 
\end{example}

\begin{example} \normalfont For a differential storage category $\mathbb{X}$, the linearizing combinator in context $C$ for the coKleisli category $\mathbb{X}_\oc$ can alternatively be expressed using the codereliction map and the Seely isomorphisms as follows for a coKleisli map ${f: \oc (C \times A) \to B}$: 
\[ \mathsf{L}^C[f] := \!\xymatrixcolsep{3pc}\xymatrix{\oc(C \times A) \ar[r]^-{\chi_{C,A}} & \oc C \otimes \oc A \ar[r]^-{1 \otimes \varepsilon_A} & \oc C \otimes A \ar[r]^-{1 \otimes \eta_A} & \oc C \otimes \oc A \ar[r]^-{\chi^{-1}_{C,A}} & \oc(C \times A)  \ar[r]^-{f} & B 
 } \]
\end{example}

\begin{example} \normalfont For $\mathsf{CON}$, the partial linearizing combinator is defined as follows on a smooth function $f: C \times E \to F$:
\[ \mathsf{L}^C[f](z,x) := \lim \limits_{t \to 0} \frac{f(z,t \cdot x) - f(z,0)}{t} \]
\end{example}

We now prove the converse of Proposition \ref{DLprop2} and show that from a system of linearizing combinators one can build a differential combinator. The construction is a generalization of the differential combinator found in \cite{bauer2018directional}. The construction can also be described in terms of smooth functions. Consider the polynomial function $f(x) = x^3 + x$, then:
\[f(x+y) = (x+y)^3 + x + y = x^3 + 3 x^2 y + 3 x y^2 + y^3 + x +y\]
The linearization of $f(x+y)$ in terms of $y$ is $3x^2y + y$ which is precisely the directional derivative $\mathsf{D}[f](x,y)$. Therefore, the derivative of $f$ can be defined by linearizing in context $f$ precompose by the addition map. 

\begin{proposition}\label{LDprop} Every Cartesian left additive category with a system of linearizing combinators $\mathsf{L}^{C}$ is a Cartesian differential category with differential combinator $\mathsf{D}_\mathsf{L}$ defined as follows on a map $f: A \to B$: 
\begin{equation}\label{DLdef}\begin{gathered} \mathsf{D}_\mathsf{L}[f] := \mathsf{L}^{A}\left[\oplus_A f \right]
\end{gathered}\end{equation} 
where $\oplus_A$ is defined as in Lemma \ref{opluslem0}. Furthermore, 
\begin{enumerate}[{\em (i)}]
\item \label{LDprop.i} For every map $f:A \to B$, $\mathsf{D}_\mathsf{L}[f]$ is $\mathsf{L}^{A}$-linear;
\item \label{LDprop.ii} A map $f: C \times A \to B$ is linear in its second argument if and only if $f$ is $\mathsf{L}^{C}$-linear. 
\item \label{LDprop.iii} $\mathsf{L} = \mathsf{L}_{\mathsf{D}_\mathsf{L}}$, where $\mathsf{L}_{\mathsf{D}_\mathsf{L}}$ is the induced linearizing combinator from Proposition \ref{DLprop} and $\mathsf{L}$ is the induced  linearizing combinator from Proposition \ref{lemL1}. 
\end{enumerate}
\end{proposition} 
\begin{proof} We must show that $\mathsf{D}_\mathsf{L}$ satisfies \textbf{[CD.1]} to \textbf{[CD.7]}. \\ \\
\noindent \textbf{[CD.1]} $\mathsf{D}_\mathsf{L}[f+g] = \mathsf{D}_\mathsf{L}[f] + \mathsf{D}_\mathsf{L}[g]$ and $\mathsf{D}_\mathsf{L}[0]=0$ 
\begin{align*}
\mathsf{D}_\mathsf{L}[f+g] &=~\mathsf{L}^{A}\left[\oplus_A (f+g) \right] \\
&=~\mathsf{L}^{A}\left[\oplus_Af +  \oplus_Ag\right] \\
&=~\mathsf{L}^{A}\left[\oplus_Af \right] +  \mathsf{L}^{A}\left[\oplus_Ag\right] \tag*{\textbf{[L.1]}} \\ 
&=~\mathsf{D}_\mathsf{L}[f] + \mathsf{D}_\mathsf{L}[g] 
\end{align*}
\begin{align*}
\mathsf{D}_\mathsf{L}[0] &=~\mathsf{L}^{A}\left[\oplus_A 0 \right] \\
&=~\mathsf{L}^{A}\left[0 \right] \\
&=~ 0 \tag*{\textbf{[L.1]}}
\end{align*}

\noindent \textbf{[CD.2]} $(1 \times \oplus_A) \mathsf{D}_\mathsf{L}[f] = (1 \times \pi_0) \mathsf{D}_\mathsf{L}[f] + (1 \times \pi_1) \mathsf{D}_\mathsf{L}[f]$ and $\langle 1, 0 \rangle \mathsf{D}_\mathsf{L}[f]=0$ 
\begin{align*}
(1 \times \oplus_A) \mathsf{D}_\mathsf{L}[f]  &=~ (1 \times \oplus_A) \mathsf{L}^{A}\left[\oplus_A f \right] \\
&=~ (1 \times \pi_0) \mathsf{L}^{A}\left[\oplus_A f \right] + (1 \times \pi_1) \mathsf{L}^{A}\left[\oplus_A f \right] \tag*{\textbf{[L.2]}} \\ 
&=~(1 \times \pi_0) \mathsf{D}_\mathsf{L}[f] +(1 \times \pi_1) \mathsf{D}_\mathsf{L}[f] \\ \\
\langle 1,0 \rangle \mathsf{D}_\mathsf{L}[f] &=~ \langle 1, 0 \rangle \mathsf{L}^{A}\left[\oplus_A f \right] \\
&=~0 \tag*{\textbf{[L.2]}} 
\end{align*}

\noindent \textbf{[CD.3]} $\mathsf{D}_\mathsf{L}[1]=\pi_1$, $\mathsf{D}_\mathsf{L}[\pi_0] = \pi_1\pi_0$, and $\mathsf{D}_\mathsf{L}[\pi_1] = \pi_1\pi_1$ 
\begin{align*}
\mathsf{D}_\mathsf{L}[1] &=~\mathsf{L}^{A}\left[\oplus_A\right] \\
&=~\mathsf{L}^{A}\left[\pi_0 + \pi_1\right] \\
&=~\mathsf{L}^{A}\left[\pi_0 \right] + \mathsf{L}^{A}\left[\pi_1 \right] \tag*{\textbf{[L.1]}} \\
&=~ 0 + \mathsf{L}^{A}\left[\pi_1 \right] \tag{Lemma \ref{linclem1}.(\ref{linclem1.i})} \\
&=~\pi_1 \tag*{\textbf{[L.3]}} 
\end{align*}
\begin{align*}
\mathsf{D}_\mathsf{L}[\pi_i] &=~\mathsf{L}^{A \times B}\left[\oplus_{A \times B}\pi_i\right] \\
&=~ \mathsf{L}^{A \times B}\left[(\pi_0 + \pi_1)\pi_i\right] \\
&=~\mathsf{L}^{A \times B}\left[\pi_0\pi_i + \pi_1\pi_i \right] \tag{$\pi_i$ is additive}\\
&=~\mathsf{L}^{A \times B}\left[\pi_0\pi_i \right] + \mathsf{L}^{A \times B}\left[\pi_1\pi_i \right] \tag*{\textbf{[L.1]}} \\
&=~ 0 + \mathsf{L}^{A \times B}\left[\pi_1\pi_i \right] \tag{Lemma \ref{linclem1}.(\ref{linclem1.i})} \\
&=~ \pi_1\pi_i \tag*{\textbf{[L.3]}} 
\end{align*}

\noindent \textbf{[CD.4]} $\mathsf{D}_\mathsf{L}\left[\langle f, g \rangle \right] = \left \langle  \mathsf{D}_\mathsf{L}[f] , \mathsf{D}_\mathsf{L}[g] \right \rangle$ 
\begin{align*}
\mathsf{D}_\mathsf{L}\left[\langle f, g \rangle \right] &=~\mathsf{L}^{A}\left[\oplus_A\langle f, g \rangle \right] \\
&=~\mathsf{L}^{A}\left[\langle \oplus_Af, \oplus_Ag \rangle \right] \\
&=~\left \langle \mathsf{L}^{A}\left[\oplus_A f \right] , \mathsf{L}^{A}\left[\oplus_A g \right]  \right \rangle  \tag*{\textbf{[L.4]}} \\
&=~\left \langle  \mathsf{D}_\mathsf{L}[f] , \mathsf{D}_\mathsf{L}[g] \right \rangle
\end{align*}

\noindent\textbf{[CD.5]} $\mathsf{D}_\mathsf{L}[fg] = \langle \pi_0 f, \mathsf{D}_\mathsf{L}[f] \rangle \mathsf{D}_\mathsf{L}[g]$
\begin{align*}
\mathsf{D}_\mathsf{L}[fg] &=~\mathsf{L}^{A}\left[\oplus_Afg \right] \\
&=~\mathsf{L}^{A}\left[\langle \pi_0, \oplus_Af \rangle \pi_1g \right] \\
&=~ \langle \pi_0, \mathsf{L}^{A}\left[\oplus_Af \right]  \rangle~ \mathsf{L}^{A}\left[\left \langle \pi_0, \pi_1 + \langle \pi_0, 0 \rangle  \oplus_Af \right \rangle \pi_1g\right] \tag*{\textbf{[L.5]}} \\
&=~ \langle \pi_0, \mathsf{L}^{A}\left[\oplus_Af \right]  \rangle~ \mathsf{L}^{A}\left[\left(\pi_1 + \langle \pi_0, 0 \rangle  \oplus_Af \right) g\right] \\
&=~ \langle \pi_0, \mathsf{L}^{A}\left[\oplus_Af \right]  \rangle~ \mathsf{L}^{A}\left[\left(\pi_1 + \pi_0\langle 1, 0 \rangle  \oplus_Af \right) g\right] \\
&=~ \langle \pi_0, \mathsf{L}^{A}\left[\oplus_Af \right]  \rangle~ \mathsf{L}^{A}\left[\left(\pi_1 + \pi_0f \right) g\right] \tag{Lemma \ref{opluslem0}.(\ref{opluslem1})} \\
&=~ \langle \pi_0, \mathsf{L}^{A}\left[\oplus_Af \right]  \rangle~ \mathsf{L}^{A}\left[\left((f\times 1)\pi_1 + (f \times 1)\pi_0 \right) g\right] \\
&=~ \langle \pi_0, \mathsf{L}^{A}\left[\oplus_Af \right]  \rangle~ \mathsf{L}^{A}\left[(f\times 1)\left(\pi_1 +\pi_0 \right) g\right] \\
&=~ \langle \pi_0, \mathsf{L}^{A}\left[\oplus_Af \right]  \rangle~ \mathsf{L}^{A}\left[(f\times 1)\left(\pi_0 +\pi_1 \right) g\right] \\
&=~ \langle \pi_0, \mathsf{L}^{A}\left[\oplus_Af \right]  \rangle~ \mathsf{L}^{A}\left[(f \times 1)\oplus_B g\right] \\
&=~ \langle \pi_0, \mathsf{L}^{A}\left[\oplus_Af \right]  \rangle (f \times 1) \mathsf{L}^{B}\left[\oplus_B g\right]  \tag*{\textbf{[L.8]}} \\
&=~ \langle \pi_0 f, \mathsf{L}^{A}\left[\oplus_Af \right]  \rangle ~ \mathsf{L}^{B}\left[\oplus_B g\right] \\
&=~ \langle \pi_0 f, \mathsf{D}_\mathsf{L}[f] \rangle \mathsf{D}_\mathsf{L}[g] 
\end{align*}

\noindent \textbf{[CD.6]} $\ell \mathsf{D}_\mathsf{L}\left[\mathsf{D}_\mathsf{L}[f] \right] = \mathsf{D}_\mathsf{L}[f]$

\begin{align*} \ell~ \mathsf{D}_\mathsf{L}\left[\mathsf{D}_\mathsf{L}[f] \right]&=~ \ell ~ \mathsf{L}^{A \times A}\left[\oplus_{A \times A} \mathsf{D}_\mathsf{L}[f] \right]  \\
&=~ \mathsf{L}^{A}\left[\ell \oplus_{A \times A} \mathsf{D}_\mathsf{L}[f] \right]  \tag{Lemma \ref{linclem1}.(\ref{linclem1.ii})}\\
&=~ \mathsf{L}^{A}\left[  \mathsf{D}_\mathsf{L}[f] \right]  \tag{Lemma \ref{opluslem0}.(\ref{opluslem2})}  \\
&=~ \mathsf{L}^{A}\left[ \mathsf{L}^{A}\left[\oplus_Af \right] \right] \\
&=~\mathsf{L}^{A}\left[\oplus_Af \right] \tag*{\textbf{[L.6]}} \\
&=~  \mathsf{D}_\mathsf{L}[f]
\end{align*}

\noindent \textbf{[CD.7]} $c~ \mathsf{D}_\mathsf{L}\left[\mathsf{D}_\mathsf{L}[f] \right]=  \mathsf{D}_\mathsf{L}\left[\mathsf{D}_\mathsf{L}[f] \right]$
\begin{align*} c~ \mathsf{D}_\mathsf{L}\left[\mathsf{D}_\mathsf{L}[f] \right] &=~  c~ \mathsf{L}^{A \times A}\left[\oplus_{A \times A} \mathsf{L}^{A}[\oplus_Af] \right]  \\
&=~ c~ \mathsf{L}^{A \times A}\left[c (\oplus_{A} \times \oplus_A) \mathsf{L}^{A}[\oplus_Af] \right]  \tag{Lemma \ref{opluslem0}.(\ref{opluslem2})}  \\
&=~ c~ \mathsf{L}^{A \times A}\left[c~ \mathsf{L}^{A \times A}[(\oplus_{A} \times \oplus_A)\oplus_Af] \right]   \tag{Lemma \ref{linclem1}.(\ref{linclem1.iii})}  \\
&=~  \mathsf{L}^{A \times A}\left[c~ \mathsf{L}^{A \times A}[c(\oplus_{A} \times \oplus_A)\oplus_Af] \right]    \tag*{\textbf{[L.7.a]}} \\
&=~  \mathsf{L}^{A \times A}\left[c~ \mathsf{L}^{A \times A}[(\oplus_{A} \times \oplus_A)\oplus_Af] \right]   \tag{Lemma \ref{opluslem0}.(\ref{opluslem1})}  \\
&=~  \mathsf{L}^{A \times A}\left[c(\oplus_{A} \times \oplus_A) \mathsf{L}^{A}[\oplus_Af] \right] \tag{Lemma \ref{linclem1}.(\ref{linclem1.iii})}  \\
&=~ \mathsf{L}^{A \times A}\left[\oplus_{A \times A} \mathsf{L}^{A}[\oplus_Af] \right] \tag{Lemma \ref{opluslem0}.(\ref{opluslem2})}  \\
&=~ \mathsf{D}_\mathsf{L}\left[\mathsf{D}_\mathsf{L}[f] \right]
\end{align*}
So we conclude that $\mathsf{D}_\mathsf{L}$ is a differential combinator. Next, it follows immediately from \textbf{[L.6]} that:
\begin{align*}
\mathsf{L}^A[\mathsf{D}_\mathsf{L}[f]] &=~ \mathsf{L}^A[\mathsf{L}^A[\oplus_A f]] \\
&=~ \mathsf{L}^A[\oplus_A f] \tag*{\textbf{[L.6]}} \\
&=~ \mathsf{D}_\mathsf{L}[f]
\end{align*}
Therefore, $\mathsf{D}_\mathsf{L}[f]$ is $\mathsf{L}^A$-linear. Now suppose that a map $f: C \times A \to B$ was $\mathsf{L}^C$-linear, that is, $\mathsf{L}^C[f] = f$. Then we compute that:
\begin{align*}
\ell \mathsf{D}_\mathsf{L}[f] &=~ \ell  \mathsf{L}^{C \times A}[\oplus_{C \times A} f] \\
&=~ \mathsf{L}^C[\ell \oplus_{C \times A} f] \tag{Lemma \ref{linclem1}.(\ref{linclem1.ii})}\\
&=~ \mathsf{L}^C[ f]\tag{Lemma \ref{opluslem0}.(\ref{opluslem2})}  \\
&=~ f \tag{$f$ is $\mathsf{L}^C$-linear}
\end{align*}
Then by Lemma \ref{lin2lem}.(\ref{lin2lem.i}), $f$ is linear in its second argument. Conversely, suppose that ${f: C \times A \to B}$ is linear in its second argument, that is, $ \ell \mathsf{D}_\mathsf{L}[f] = f$. Then we have that: 
\begin{align*}
\mathsf{L}^C[ f] &=~\mathsf{L}^C[\ell \oplus_{C \times A} f]\tag{Lemma \ref{opluslem0}.(\ref{opluslem2})}  \\
&=~ \ell  \mathsf{L}^{C \times A}[\oplus_{C \times A} f] \tag{Lemma \ref{linclem1}.(\ref{linclem1.ii})}\\
&=~ \ell \mathsf{D}_\mathsf{L}[f] \\
&=~ f \tag{$f$ is linear in its second argument}
\end{align*}
Therefore, $f$ is $\mathsf{L}^C$-linear. Lastly, we show that, in this case, the constructions of the linearizing combinators from Proposition \ref{lemL1} and Proposition \ref{DLprop} are the same:  
\begin{align*}
\mathsf{L}_{\mathsf{D}_\mathsf{L}}[f] &=~ \langle 0,1 \rangle \mathsf{D}_\mathsf{L}[f] \\
&=~ \langle 0,1 \rangle \mathsf{L}^A[\oplus_A f] \\
&=~ \langle 0,1 \rangle (0 \times 1) \mathsf{L}^A[\oplus_A f] \\
&=~ \langle 0,1 \rangle \mathsf{L}^\top[(0 \times 1) \oplus_A f] \tag*{\textbf{[L.8]}} \\
&=~ \langle 0,1 \rangle \mathsf{L}^\top[(0 \times 1)(\pi_0 + \pi_1) f]  \\
&=~ \langle 0,1 \rangle \mathsf{L}^\top[\left((0 \times 1)\pi_0 + (0 \times 1)\pi_1 \right) f]  \\
&=~ \langle 0,1 \rangle \mathsf{L}^\top[\left(0 + \pi_1 \right) f]  \\
&=~ \langle 0,1 \rangle \mathsf{L}^\top[\pi_1 f]  \\
&=~ \mathsf{L}[f] 
\end{align*}
So we conclude that $\mathsf{L} = \mathsf{L}_{\mathsf{D}_\mathsf{L}}$. 
\end{proof}

We may now state the main result of this paper. 

\begin{theorem}\label{DLthm} For a Cartesian left additive category, there is a bijective correspondence between:
\begin{enumerate}[{\em (i)}]
\item Differential combinators;
\item Systems of linearizing combinators. 
\end{enumerate}
Therefore, a Cartesian differential category is precisely a Cartesian left additive category equipped with a system of linearizing combinators. 
\end{theorem} 
\begin{proof}
It suffices to show that the constructions of Proposition \ref{LDprop} and Proposition \ref{DLprop} are inverses of each other. Starting with a differential combinator $\mathsf{D}$, we first show that $\mathsf{D}_{\mathsf{L}_\mathsf{D}} = \mathsf{D}$: 
\begin{align*}
\mathsf{D}_{\mathsf{L}_\mathsf{D}}[f] &=~ \mathsf{L}^{A}_\mathsf{D}[\oplus_Af] \tag{\ref{DLdef}} \\
&=~ \ell  ~\mathsf{D}\left[\oplus_Af\right] \tag{\ref{contextlin}}\\ 
&=~\ell  \left( \oplus_A \times \oplus_A \right)  \mathsf{D}[f]  \tag{Cor. \ref{corlin}(\ref{corlin.oplus}) + Lem. \ref{linlem}.(\ref{linlem.pre})}\\
&=~  \mathsf{D}[f]  \tag{Lemma \ref{opluslem0}.(\ref{opluslem2})} 
\end{align*}
Next, starting with a system of linearizing combinators $\mathsf{L}^{C}$, we show that $\mathsf{L}^{C}_{\mathsf{D}_\mathsf{L}} = \mathsf{L}^{C}$: 
\begin{align*}
\mathsf{L}^{C}_{\mathsf{D}_\mathsf{L}}[f] &=~ \ell~ \mathsf{D}_\mathsf{L}[f] \tag{\ref{contextlin}}\\
&=~ \ell~  \mathsf{L}^{C \times A}[\oplus_{C \times A}f] \tag{\ref{DLdef}} \\
&=~\mathsf{L}^{C}[\ell \oplus_{C \times A}f] \tag{Lemma \ref{linclem1}.(\ref{linclem1.ii})} \\
&=~ \mathsf{L}^{C}\left[ f \right]  \tag{Lemma \ref{opluslem0}.(\ref{opluslem2})} 
\end{align*}
Thus, differential combinators and systems of linearizing combinators are in bijective correspondence. Therefore, we conclude that a Cartesian differential category is precisely a Cartesian left additive category equipped with a system of linearizing combinators. 
\end{proof}

It is worth pointing out that the bijective correspondence between differential combinators and systems of linearizing combinators is analogous to the bijective correspondence between deriving transformations and coderelictions for differential categories \cite[Theorem 4]{Blute2019} (or as explained in Example \ref{ex:diffstor}). Indeed, recall that from a codereliction $\eta$, one defines a deriving transformation as $\mathsf{d} = (1 \otimes \eta) \nabla$. In the coKleisli category, the multiplication $\nabla$ plays the role of pre-composing by addition $\oplus$ (since $\oc$ is an additive bialgebra modality \cite[Definition 5]{Blute2019}), while $1 \otimes \eta$ plays the role of linearizing the second argument, that is, the linearizing combinator in context $\mathsf{L}^C$. The converse construction is explained in Example \ref{ex:diffstorlin}. The keen eye reader may note that the ``partial'' codereliction $1 \otimes \eta$ can easily be defined from the ``total'' codereliction. The reason for this is the presence of the Seely isomorphisms $\oc(C \times A) \cong \oc C \otimes \oc A$ which allows us to split off the context part and then bring it back afterwards. Unfortunately, as previously mentioned, this does not work in arbitrary Cartesian differential categories. To do so, we require the base category to be Cartesian closed, which we discuss in the next section. 

We conclude this section by providing an example of a Cartesian left additive category which has a total linearization but does not have partial linearization.  This means that it is not possible, in general, to derive partial linearization from the presence of a total linearizing combinator. 

\begin{example}\label{counter-example} \normalfont Recall that a function $F: \mathbb{R}^n \to \mathbb{R}^m$, which is a tuple $F = \langle f_1, \hdots, f_m \rangle$ of functions $f_i: \mathbb{R}^n \to \mathbb{R}$, is a $\mathcal{C}^1$ function if for each $f_i$, all partial derivatives $\frac{\partial f_i}{\partial x_j}$ exists and are continuous. Then define $\mathcal{C}^1\text{-}\mathsf{DIFF}$ be the category whose objects are the Euclidean real vector spaces $\mathbb{R}^n$ and whose maps are $\mathcal{C}^1$ functions ${F: \mathbb{R}^n \to \mathbb{R}^m}$ between them. $\mathcal{C}^1\text{-}\mathsf{DIFF}$ is a Cartesian left additive category in the obvious way, and note that $\mathsf{SMOOTH}$ is a sub-Cartesian left additive category of $\mathcal{C}^1\text{-}\mathsf{DIFF}$. Notice that $\mathcal{C}^1\text{-}\mathsf{DIFF}$ has a (total) linearizing combinator $\mathsf{L}$ defined in the same way as the linearizing combinator in $\mathsf{SMOOTH}$, that is, for a $\mathcal{C}^1$ function $F = \langle f_1, \hdots, f_m \rangle$: 
\[ \mathsf{L}[F](\vec x) =  \left \langle  \sum \limits^n_{i=1} \frac{\partial f}{\partial x_i}(\vec 0) x_i, \hdots,  \sum \limits^n_{i=1} \frac{\partial f}{\partial x_i}(\vec 0) x_i \right \rangle \]
However, this category, while having a total linearizing combinator, does not have partial linearization.  If $\mathcal{C}^1\text{-}\mathsf{DIFF}$ had partial linearization then $\mathcal{C}^1\text{-}\mathsf{DIFF}$ would also have a differential combinator, but this can't be since the derivative of $\mathcal{C}^1$ functions are not necessarily $\mathcal{C}^1$ functions. 

Explicitly, consider the function ${f: \mathbb{R} \to \mathbb{R}}$, $f(x) = \abs{x}^{\frac{3}{2}}$, which is a $\mathcal{C}^1$ function since its derivative $f^\prime(x) =  \frac{3x}{2 \sqrt{\abs{x}}}$ 
%\[f^\prime(x) = \begin{cases} \frac{3x}{2 \sqrt{\abs{x}}} & \text{for } x\neq 0\\
%0 & \text{for } x=0 
%\end{cases}\] 
exists and is continuous. If partial linearization was possible, then we would be able to define $\mathsf{D}[f]: \mathbb{R} \times \mathbb{R} \to \mathbb{R}$ as follows:
\[ \mathsf{D}[f](x,y) = \mathsf{L}[z \mapsto f(x+z)](y) =  \frac{3xy}{2 \sqrt{\abs{x}}} \]
However, this linearization is not a $\mathcal{C}^1$ function (since its derivative is undefined at $0$) and so not a map in $\mathcal{C}^1\text{-}\mathsf{DIFF}$. So we conclude that $\mathcal{C}^1\text{-}\mathsf{DIFF}$ has a total linearizing combinator, however, it is not induced by a differential combinator and, therefore, the category does not have partial linearization.  
 \end{example}

\section{Linearizing Combinators in the Closed Setting}\label{closed-sec}

We would like to prove the converse of Proposition \ref{lemL1}, that is, we would like to define partial linearization from total linearization. As previously discussed, in general this is not necessarily possible. However in the setting of a Cartesian closed category, it is possible to construct a system of linearizing combinators from a linearizing combinator on the base category. The key to this construction is the ability to curry and uncurry maps, which allows us to move the context of a map from its domain to its codomain. Indeed, given a map $f: C \times A \to B$, to linearize $A$ while keeping $C$ in context, one takes the total linearization of its curry $\lambda(f): A \to [C,B]$ and then uncurry to obtain $\mathsf{L}^C[f]: C \times A \to B$. For this to work, one must also require that the linearizing combinator be compatible with the closed structure, which we call an \emph{exponential} linearizing combinator. Furthermore, we will also show that Cartesian \emph{closed} differential categories are precisely Cartesian \emph{closed} left additive categories equipped with an exponential linearizing combinator. 

We begin this section by setting up notation for Cartesian closed categories and reviewing some basic, but very important, properties (see \cite[Part I]{lambek1988introduction} for a more detailed introduction on Cartesian closed categories). For a Cartesian closed category $\mathbb{X}$, we denote the internal-hom by $[C,A]$, the evaluation map by $\epsilon_{C,A}: C \times [C,A] \to A$ (from now on we will omit the subscripts and simply write $\epsilon$ when there is no confusion), and the curry of a map $f: C \times A \to B$ as the map ${\lambda(f): A \to [C,B]}$, that is, $\lambda(f)$ is the unique map such that:
\[ (1 \times \lambda(f)) \epsilon = f \]
Conversely, define the un-curry of a map of type ${g: A \to [C,B]}$ as the map $\lambda^{-1}(g) : C \times A \to B$ which is defined as:
\[\lambda^{-1}(g) := (1 \times g) \epsilon\]
Therefore, $\lambda\left( \lambda^{-1}(g) \right) = g$ and $\lambda^{-1}\left( \lambda(f) \right) = f$. 

Next we review the notion of Cartesian closed differential categories. As the name suggests, Cartesian closed differential categories are Cartesian differential categories whose underlying category is also Cartesian closed and such that the differential combinator is compatible with the curry operator. Furthermore, Cartesian closed differential categories provide suitable models to interpret \emph{differential} $\lambda$-calculus \cite{EHRHARD20031}. Cartesian closed differential categories are also sometimes called differential $\lambda$ categories. For a more in-depth introduction to Cartesian closed differential categories, we refer the reader to \cite{bucciarelli2010categorical,Cockett-2019,manzonetto_2012}. 

We must first discuss the notion of Cartesian closed left additive categories: 

\begin{definition} A \textbf{Cartesian closed left additive category} \cite[Section 1.4]{blute2009cartesian} is a Cartesian left additive category which is also a Cartesian closed category such that the currying operator preserves the additive structure, that is, $\lambda(f+g) = \lambda(f) + \lambda(g)$ and $\lambda(0) = 0$ (note that this implies that $\lambda^{-1}(f+g) = \lambda^{-1}(f) + \lambda^{-1}(g)$ and $\lambda^{-1}(0) = 0$). 
\end{definition}

As shown in \cite[Lemma 4.10]{Cockett-2019}, there are two equivalent ways of expressing compatibility between the closed structure and the differential combinator: one in terms of the curry operator and one in terms of the evaluation map. 

\begin{definition}\label{CDCcloseddef} A \textbf{Cartesian closed differential category} \cite[Section 4.6]{Cockett-2019} (also known as a \textbf{differential $\lambda$ category} \cite{bucciarelli2010categorical,manzonetto_2012}) is a Cartesian differential category which is also a Cartesian closed left additive category such that one of the following additional axioms hold: 
\begin{description}
\item[{\bf [CD.$\lambda$]}] For every map $f: C \times A \to B$, $\mathsf{D}[\lambda(f)] = \lambda\left( \mathsf{D}^C[f] \right)$, where $\mathsf{D}^C$ is defined as in (\ref{DCdef}). 
\end{description}
or equivalently,
\begin{description}
\item[{\bf [CD.ev]}] Evaluation maps $\epsilon: C \times [C,A] \to A$ are linear in their second argument (Definition \ref{lin2def}), that is, $\mathsf{D}^C[\epsilon] = (1 \times \pi_1) \epsilon$, or equivalently by Lemma \ref{lin2lem}.(\ref{lin2lem.i}), $\ell \mathsf{D}[\epsilon] = \epsilon$. 
\end{description}
\end{definition}

Here are now some examples of Cartesian closed differential categories. 

\begin{example} \normalfont Every model of the differential $\lambda$-calculus \cite{EHRHARD20031} induces a Cartesian closed differential category \cite[Theorem 4.3]{Cockett-2019}, and conversely every Cartesian closed differential category gives rise to a model of the differential $\lambda$-calculus \cite[Theorem 4.12]{bucciarelli2010categorical}. 
\end{example}

\begin{example} \normalfont Let $\mathbb{X}$ be a differential storage category such that $\mathbb{X}$ is also a symmetric monoidal closed category, where we denote the internal-hom in $\mathbb{X}$ as $A \multimap B$. Then the coKleisli category $\mathbb{X}_\oc$ is a Cartesian closed differential category \cite[Theorem 4.4.2]{blute2015cartesian}. The internal-homs in the coKleisli category $\mathbb{X}_\oc$ are defined as $[A,B] = \oc A \multimap B$. Examples of such coKleisli categories are discussed in \cite[Section 5]{bucciarelli2010categorical}, which include the relational model and the finiteness space model. 
\end{example}

\begin{example} \normalfont $\mathsf{CON}_{lin}$ is a differential storage category such that $\mathsf{CON}_{lin}$ is symmetric monoidal closed \cite[Theorem 4.2]{blute2010convenient}. Therefore, since $\mathsf{CON}$ is isomorphic to the coKleisli category of the comonad $\oc$ on $\mathsf{CON}_{lin}$, it follows that $\mathsf{CON}$ is also a Cartesian closed differential category (see \cite[Theorem 3.12]{kriegl1997convenient} for its Cartesian closed structure). In particular, for convenient vector spaces $E$ and $F$, if we let $\mathcal{L}(E,F)$ denote the set of (smooth) linear function between $E$ and $F$ and $\mathcal{C}^\infty(E,F)$ the set of all smooth functions between $E$ and $F$, then $\mathcal{C}^\infty(E,F) \cong L(\oc E, F)$ \cite[Theorem 6.3]{blute2010convenient}.
\end{example}

We now turn our attention to the main objective of this section: on how to define partial linearization from total linearization in the setting of a Cartesian closed left additive category. To do so, we introduce the notions of linearizing combinators and systems of linearizing combinators which are compatible with the closed structure. We begin with \emph{closed} systems of linearizing combinators, which are the Cartesian closed differential category version of systems of linearizing combinators. 

\begin{definition}\label{closeddef} A \textbf{closed system of linearizing combinators} on a Cartesian closed left additive category $\mathbb{X}$ is a system of linearizing combinators $\mathsf{L}^C$ on $\mathbb{X}$ such that the following extra axiom holds: 
\begin{description}
\item[{\bf [L.$\lambda$]}] For every map $f: C \times A \to B$, $\mathsf{L}[\lambda(f)] = \lambda\left( \mathsf{L}^C[f] \right)$, where $\mathsf{L}$ is defined as in (\ref{L1def}). 
%\item[{\bf [L.ev]}] Evaluation maps $\epsilon: C \times [C,A] \to A$ are $\mathsf{L}^C$-linear, that is, $\mathsf{L}^C[\epsilon] = \epsilon$. 
\end{description}
\end{definition}

As we will see in Theorem \ref{finalthm}, to give a Cartesian closed differential category is precisely to give a closed system of linearizing combinators. As such, {\bf [L.$\lambda$]} is the linearizing combinator analogue of {\bf [CD.$\lambda$]}. Therefore, the extra axiom of a closed system of linearizing combinators can equivalently be defined in terms of the evaluation map, {\bf [L.ev]}, which is the linearizing combinator analgoue of {\bf [CD.ev]}. 

\begin{lemma} {\bf [L.$\lambda$]} is equivalent to the following: 
\begin{description}
\item[{\bf [L.ev]}] Evaluation maps $\epsilon: C \times [C,A] \to A$ are $\mathsf{L}^C$-linear, that is, $\mathsf{L}^C[\epsilon] = \epsilon$. 
\end{description}
\end{lemma}
\begin{proof} Suppose that {\bf [L.$\lambda$]} holds. Since $\epsilon = \lambda^{-1}(1)$, we have that: 
\begin{align*}
\mathsf{L}^C[\epsilon] &=~ \lambda^{-1}\left( \lambda\left( \mathsf{L}^C[\epsilon] \right) \right) \\
&=~ \lambda^{-1}\left( \mathsf{L}[\lambda\left( \epsilon \right)] \right)\tag*{\textbf{[L.$\lambda$]}} \\
&=~ \lambda^{-1}\left( \mathsf{L}[1] \right) \\
&=~ \lambda^{-1}(1) \tag*{\textbf{[L.3]}} \\
&=~ \epsilon
\end{align*}
So $\mathsf{L}^C[\epsilon] = \epsilon$, and so $\epsilon$ is $\mathsf{L}^C$-linear. Conversely, suppose that {\bf [L.ev]} holds. Then we compute:
\begin{align*}
\lambda\left( \mathsf{L}^C[f] \right) &=~ \lambda\left( \mathsf{L}^C[\lambda^{-1}\left( \lambda(f) \right)] \right) \\
&=~ \lambda\left( \mathsf{L}^C[(1 \times \lambda(f) ) \epsilon] \right) \\
&=~  \lambda\left( \mathsf{L}^C[ \langle \pi_0, \pi_1 \lambda(f) \rangle \epsilon] \right) \\
&=~ \lambda\left( \langle \pi_0, \mathsf{L}^{C}[\pi_1 \lambda(f)] \rangle~ \epsilon \right) \tag{\textbf{[L.ev]} +Lem.\ref{lstable-lem}.(\ref{lstable-lem.post})} \\
&=~  \lambda\left( \langle \pi_0, \pi_1 \mathsf{L}[\lambda(f)] \rangle~ \epsilon \right)  \tag{Prop.\ref{lemL1}.(\ref{lemL1.i})} \\
&=~  \lambda\left( (1 \times \mathsf{L}[\lambda(f)])  \epsilon \right) \\
&=~ \lambda\left(\lambda^{-1}\left(\mathsf{L}[\lambda(f)] \right) \right) \\
&=~ \mathsf{L}[\lambda(f)]
\end{align*}
So $\mathsf{L}[\lambda(f)] = \lambda\left( \mathsf{L}^C[f] \right)$. 
\end{proof} 

We now define \emph{exponentiable} linearizing combinators, which from a system of linear maps perspective is the analogue of an exponentiable system of maps \cite[Definition 2.2.1]{blute2015cartesian}. To do so, we must first review the canonical monads of the form $[C, -]$ in a Cartesian closed category. For a pair of maps $f: C \to D$ and $g: A \to B$, define the map $[f,g]: [D, A] \to [C,B]$ as:
\[ [f,g] := \lambda \left( (f \times 1) \epsilon g \right) \]
Intuitively, $[f,g]$ is the map which pre-composes by $f$ and post-composes by $g$. In particular, note that $[-,-]$ is contravariant in its first argument and covariant in its second argument, that is:
\[[fh,kg] = [h,k][f,g]\] 
For each object $C$, define the functor $E^C: \mathbb{X} \to \mathbb{X}$ on objects as $E^C(A) = [C,A]$ and on maps $E^C(f) = [1,f]$. $E^C$ is a monad \cite[Part I, Section 7]{lambek1988introduction} where the monad unit $\eta^C_A: A \to [C,A]$ and the monad multiplication $\mu^C_A: [C, [C,A]] \to [C,A] $ are defined respectively as follows: 
\begin{equation}\label{etamudef}\begin{gathered} \eta^C_A := \lambda(\pi_1) \quad\quad \quad \quad \mu^C_A := \lambda \left( \langle \pi_0, \epsilon \rangle \epsilon \right) 
 \end{gathered}\end{equation}
 Once again, as to not overload notation, we will omit the subscripts and superscripts and simply write $\eta$ and $\mu$ when there is no confusion. 

\begin{definition}\label{Lexpdef} An \textbf{exponentiable linearizing combinator} $\mathsf{L}$ on a Cartesian closed left additive category $\mathbb{X}$ is a linearizing combinator $\mathsf{L}$ on $\mathbb{X}$ such that the following extra three axioms hold: 
\begin{enumerate}[{\bf [EL.1]}]
\item $\mathsf{L}[\eta]=\eta$ and $\mathsf{L}[\mu]=\mu$
\item $\mathsf{L}\left [ [f,g] \right] = \left[ f, \mathsf{L}[g] \right]$
\item For a map $f: A \times B \to C$, define $\mathsf{L}_0[f]: A \times B \to C$ and $\mathsf{L}_1[f]: A \times B \to C$ respectively as follows: 
\begin{align*}
\mathsf{L}_0[f] := \tau \lambda^{-1} \left( \mathsf{L}\left[\lambda\left(\tau f \right)\right] \right) && \mathsf{L}_1[f] :=  \lambda^{-1} \left( \mathsf{L}\left[\lambda\left(f \right)\right] \right)
\end{align*}
where $\tau$ was the canonical symmetry isomorphism defined in (\ref{taudef}). Then for every map $f: A \times B \to C$, $\mathsf{L}_0[\mathsf{L}_1[f]] =  \mathsf{L}_1[\mathsf{L}_0[f]]$.  
\end{enumerate}
\end{definition}

As we will see in Proposition \ref{LDexpprop}, from an exponentiable linearizing combinator we will be able to construct a closed system of linearizing combinators by uncurrying the linearization of the curry. In other words, we will be able to define total linearization from partial linearization. We first show that, as expected, a Cartesian closed left additive category with a closed systems of linearizing combinators is in fact a Cartesian closed differential category, and its induced linearizing combinator is an exponentiable linearizing combinator. Alternatively, we could have instead shown that the induced linearizing combinator and system of linearizing combinators of a Cartesian closed differential category are respectively exponentiable and closed. Therefore, a Cartesian closed differential category is precisely a Cartesian closed left additive category with a closed system of linearizing combinators. 

\begin{proposition}\label{LDclosedprop} For a Cartesian closed left additive category with a closed system of linearizing combinators $\mathsf{L}^C$:
\begin{enumerate}[{\em (i)}]
\item \label{LDclosedprop.i} The induced linearizing combinator $\mathsf{L}$ from Proposition \ref{lemL1} is an exponentiable linearizing combinator and $\mathsf{L}^C[f] = \lambda^{-1}\left( \mathsf{L}[\lambda(f)] \right)$. 
\item \label{LDclosedprop.ii} The induced differential combinator $\mathsf{D}_\mathsf{L}$ from Proposition \ref{LDprop} satisfies {\bf [CD.$\lambda$]} (or equivalently {\bf [CD.ev]}) and $\mathsf{D}_\mathsf{L}[f] = \lambda^{-1}\left( \mathsf{L}[\lambda(\oplus_A f)] \right)$. 
\end{enumerate}
Therefore, a Cartesian closed left additive category with a closed system of linearizing combinators is a Cartesian closed differential category. 
\end{proposition} 
\begin{proof} First note that $\mathsf{L}^C[f] = \lambda^{-1}\left( \mathsf{L}[\lambda(f)] \right)$ follows immediately from \textbf{[L.$\lambda$]}, and therefore we also have that $\mathsf{D}_\mathsf{L}[f] = \lambda^{-1}\left( \mathsf{L}[\lambda(\oplus_A f)] \right)$. Next we show that $\mathsf{L}$ satisfies \textbf{[EL.1]}, \textbf{[EL.2]}, and \textbf{[EL.3]}. \\ \\
\noindent \textbf{[EL.1]}: $\mathsf{L}[\eta]=\eta$ and $\mathsf{L}[\mu]=\mu$ 
\begin{align*}
\mathsf{L}[\eta] &=~ \mathsf{L}[\lambda(\pi_1)] \\
&=~ \lambda(\mathsf{L}^C[\pi_1]) \tag*{\textbf{[L.$\lambda$]}} \\
&=~ \lambda(\pi_1) \tag*{\textbf{[L.3]}} \\
&=~ \eta \\ \\
\mathsf{L}[\mu] &=~ \mathsf{L}[\lambda \left( \langle \pi_0, \epsilon \rangle \epsilon \right) ] \\
&=~ \lambda(\mathsf{L}^C[\left( \langle \pi_0, \epsilon \rangle \epsilon \right)]) \tag*{\textbf{[L.$\lambda$]}} \\
&=~ \lambda \left( \langle \pi_0, \epsilon \rangle \epsilon \right) \tag{\textbf{[L.ev]}  + Lem.\ref{lstable-lem}.(\ref{lstable-lem.comp})} \\
&=~ \mu
\end{align*}
\noindent \textbf{[EL.2]}: $\mathsf{L}\left [ [f,g] \right] = \left[ f, \mathsf{L}[g] \right]$:
\begin{align*}
\mathsf{L}\left [ [f,g] \right] &=~ \mathsf{L}[\lambda \left( (f \times 1) \epsilon g \right) ] \\
&=~ \lambda \left( \mathsf{L}^C[(f \times 1) \epsilon g] \right) \\ 
&=~  \lambda \left( (f \times 1)  \mathsf{L}^{C^\prime}[\epsilon g] \right)  \tag*{\textbf{[L.8]}} \\
&=~ \lambda \left( (f \times 1)  \mathsf{L}^{C^\prime}[\langle \pi_0, \epsilon \rangle \pi_1 g] \right) \\ 
&=~ \lambda \left( (f \times 1) \langle \pi_0, \epsilon \rangle \mathsf{L}^{C^\prime}[ \pi_1 g] \right) \tag{\textbf{[L.ev]} + Lem.\ref{lstable-lem}.(\ref{lstable-lem.pre})} \\
&=~ \lambda \left( (f \times 1) \langle \pi_0, \epsilon \rangle \pi_1 \mathsf{L}[ g] \right) \tag{Prop.\ref{lemL1}.(\ref{lemL1.i})} \\
&=~ \lambda \left( (f \times 1) \epsilon \mathsf{L}[ g] \right) \\
&=~ \left[ f, \mathsf{L}[g] \right]
\end{align*}
\noindent \textbf{[EL.3]}: $\mathsf{L}_0[\mathsf{L}_1[f]] =  \mathsf{L}_1[\mathsf{L}_0[f]]$: \\ \\
\noindent Note that by $\textbf{[L.$\lambda$]}$, $\lambda^{-1}\left(\mathsf{L}[f] \right) = \mathsf{L}^C[\lambda^{-1}(f)]$. As such, it immediately follows that the $\mathsf{L}_0$ and $\mathsf{L}_1$ as defined in \textbf{[EL.3]} are precisely the same as $\mathsf{L}_0$ and $\mathsf{L}_1$ defined in Proposition \ref{lemL1}.(\ref{lemL1.iv}). Therefore \textbf{[EL.3]} is precisely Proposition \ref{lemL1}.(\ref{lemL1.iv}). \\ 

So we conclude that $\mathsf{L}$ is an exponentiable linearizing combinator. Next we must check that $\mathsf{D}_\mathsf{L}$ satisfies {\bf[CD.$\lambda$]} or equivalently {\bf [CD.ev]}. By {\bf [L.ev]} , $\epsilon$ is $\mathsf{L}^C$-linear and so by Proposition \ref{LDprop}.(\ref{LDprop.ii}), $\epsilon$ is linear in its second argument. Therefore, {\bf [CD.ev]} holds and we conclude that a Cartesian closed left additive category with a closed systems of linearizing combinators is a Cartesian closed differential category. 
\end{proof} 

\begin{corollary}\label{cor:6.10} For a Cartesian closed differential category with differential combinator $\mathsf{D}$:
\begin{enumerate}[{\em (i)}]
\item  \label{cor:6.10.i}The induced system of linearizing combinators $\mathsf{L}^C_\mathsf{D}$ from Proposition \ref{DLprop2} is a closed system of linearizing combinators. 
\item The induced linearizing combinator $\mathsf{L}_\mathsf{D}$ from Proposition \ref{DLprop} is an exponential linearizing combinator. 
\end{enumerate}
\end{corollary} 
\begin{proof} We must show that $\mathsf{L}^C_\mathsf{D}$ satisfies {\bf[L.$\lambda$]} or equivalently {\bf [L.ev]}. By {\bf [CD.ev]} , $\epsilon$ is $\mathsf{D}$-linear in its second argument, and so by Proposition \ref{DLprop2}.(\ref{DLprop2.ii}), $\epsilon$ is $\mathsf{L}^C_\mathsf{D}$-linear. Therefore, {\bf [L.ev]} holds and we conclude that $\mathsf{L}^C_\mathsf{D}$ is a closed system of linearizing combinators. By Proposition \ref{LDclosedprop}.(\ref{LDclosedprop.i}), the induced linearizing combinator from Proposition \ref{lemL1} is an exponentiable linearizing combinator. However by Proposition \ref{LDprop}.(\ref{LDprop.iii}), the induced linearizing combinator from Proposition \ref{DLprop} is precisely the same as the one from Proposition \ref{lemL1}. Therefore, $\mathsf{L}_\mathsf{D}$ is an exponentiable linearizing combinator. 

\hfill
\end{proof} 

We now prove the converse of Proposition \ref{LDclosedprop}, that in the closed setting we may define partial linearization from total linearization, that is, we will show that an exponentiable linearizing combinator induces a closed system of linearizing combinators. As a consequence, it follows that a Cartesian closed differential category is precisely a Cartesian left additive category with an exponentiable linearizing combinator.

\begin{proposition}\label{LDexpprop} For Cartesian closed left additive category $\mathbb{X}$ with an exponential linearizing combinator $\mathsf{L}$:
\begin{enumerate}[{\em (i)}]
\item \label{LDexpprop.i} $\mathbb{X}$ comes equipped with a closed system of linearizing combinators $\mathsf{L}^C$ defined as follows for a map $f: C \times A \to B$: 
\[ \mathsf{L}^C[f] = \lambda^{-1}\left( \mathsf{L}[\lambda(f)] \right) \]
 and the resulting induced linearizing combinator from Proposition \ref{lemL1} is precisely $\mathsf{L}$. 
 \item \label{LDexpprop.ii} $\mathbb{X}$ is a Cartesian closed differential category with differential combinator $\mathsf{D}_\mathsf{L}$ defined as follows for a map $f: A \to B$: 
 \[ \mathsf{D}_\mathsf{L}[f] = \lambda^{-1}\left( \mathsf{L}[\lambda(\oplus_A f)] \right) \]
 and furthermore, this differential combinator is precisely the induced differential combinator from Proposition \ref{LDprop}. 
\end{enumerate}
Therefore, a Cartesian closed left additive category with an exponential linearizing combinator is a Cartesian closed differential category. 
\end{proposition} 
\begin{proof} First, here are some useful identities which hold in any Cartesian closed category \cite[Part I, Section 3]{lambek1988introduction}: 
\begin{equation}\label{CCC1}\begin{gathered} \lambda\left( (f \times g) h k \right) = g \lambda(h) [f,k] \qquad \qquad (f \times g) \lambda^{-1}(h^\prime) k = \lambda^{-1}\left( g h^\prime [f,k] \right) \end{gathered}\end{equation}
Now we show that $\mathsf{L}^C$ satisfies \textbf{[L.1]}-\textbf{[L.8]} and \textbf{[L.ev]}: \\ \\
\noindent \textbf{[L.1]}: $\mathsf{L}^{C}[f+g]=\mathsf{L}^{C}[f]+\mathsf{L}^{C}[g]$ and $\mathsf{L}^{C}[0]=0$
\begin{align*}
\mathsf{L}^{C}[f+g] &=~ \lambda^{-1}\left( \mathsf{L}[\lambda(f+g)] \right) \\
&=~ \lambda^{-1}\left( \mathsf{L}[\lambda(f) + \lambda(g)] \right) \\
&=~ \lambda^{-1}\left( \mathsf{L}[\lambda(f)] + \mathsf{L}[\lambda(g)] \right) \tag*{\textbf{[L.1]}} \\
&=~ \lambda^{-1}\left( \mathsf{L}[\lambda(f)]\right) + \lambda^{-1}\left( \mathsf{L}[\lambda(g)] \right) \\
&=~ \mathsf{L}^{C}[f]+\mathsf{L}^{C}[g] \\ \\
\mathsf{L}^{C}[0] &=~ \lambda^{-1}\left( \mathsf{L}[\lambda(0)] \right) \\
&=~ \lambda^{-1}\left( \mathsf{L}[0] \right) \\
&=~ \lambda^{-1}\left( 0 \right) \tag*{\textbf{[L.1]}} \\
&=~ 0
\end{align*}
\noindent \textbf{[L.2]}: $(1 \times \oplus_A) \mathsf{L}^{C}[f] = (1 \times \pi_0) \mathsf{L}^{C}[f]+ (1 \times \pi_1) \mathsf{L}^{C}[f]$ and $\langle 1, 0 \rangle \mathsf{L}^{C}[f]=0$:
\begin{align*}
(1 \times \oplus_A) \mathsf{L}^{C}[f]  &=~ (1 \times \oplus_A)  \lambda^{-1}\left( \mathsf{L}[\lambda(f)]\right) \\
&=~  \lambda^{-1}\left( \oplus_A \mathsf{L}[\lambda(f)]\right) \tag{\ref{CCC1}}\\
&=~  \lambda^{-1}\left( \pi_0 \mathsf{L}[\lambda(f)] + \pi_1 \mathsf{L}[\lambda(f)]  \right) \tag*{\textbf{[L.2]}} \\
&=~  \lambda^{-1}\left( \pi_0 \mathsf{L}[\lambda(f)] \right)  + \lambda^{-1}\left( \pi_1 \mathsf{L}[\lambda(f)]  \right) \\
&=~ (1 \times \pi_0) \lambda^{-1}\left( \mathsf{L}[\lambda(f)]\right) + (1 \times \pi_1)\lambda^{-1}\left( \mathsf{L}[\lambda(f)]\right)  \tag{\ref{CCC1}}\\
&=~ (1 \times \pi_0) \mathsf{L}^{C}[f]+ (1 \times \pi_1) \mathsf{L}^{C}[f] \\ \\
\langle 1, 0 \rangle \mathsf{L}^{C}[f] &=~ \langle 1, 1\rangle (1 \times 0)  \mathsf{L}^{C}[f] \\
&=~ \langle 1, 1\rangle (1 \times 0) \lambda^{-1}\left( \mathsf{L}[\lambda(f)]\right) \\ 
&=~ \langle 1, 1\rangle \lambda^{-1}\left( 0 \mathsf{L}[\lambda(f)]\right)  \tag{\ref{CCC1}}\\
&=~ \langle 1, 1\rangle \lambda^{-1}\left( 0 \right) \tag*{\textbf{[L.2]}} \\
&=~ \langle 1, 1\rangle 0 \\
&=~ 0
\end{align*}
\noindent \textbf{[L.3]}: $\mathsf{L}^{C}[\pi_1]=\pi_1$ and $\mathsf{L}^{C}[\pi_1\pi_i]=\pi_1\pi_i$: 
\begin{align*}
\mathsf{L}^{C}[\pi_1] &=~  \lambda^{-1}\left( \mathsf{L}[\lambda(\pi_1)]\right) \\
&=~ \lambda^{-1}\left( \mathsf{L}[\eta]\right) \\
&=~ \lambda^{-1}\left( \eta \right)  \tag*{\textbf{[EL.1]}} \\
&=~ \lambda^{-1}\left( \lambda(\pi_1) \right) \\
&=~ \pi_1 \\ \\
\mathsf{L}^{C}[\pi_1\pi_i]&=~ \lambda^{-1}\left( \mathsf{L}[\lambda(\pi_1\pi_i)]\right) \\
&=~ \lambda^{-1}\left( \mathsf{L}\left[\lambda(\pi_1) [1, \pi_i] \right]\right)  \tag{\ref{CCC1}}\\
&=~ \lambda^{-1}\left( \mathsf{L}\left[\eta [1, \pi_i] \right]\right) \\
&=~ \lambda^{-1}\left( \eta~ \mathsf{L}\left[ [1, \pi_i] \right] \right)  \tag{\textbf{[EL.1]} + Lem.\ref{lstable-lem}.(\ref{lstable-lem.pre})} \\
&=~ \lambda^{-1}\left( \eta \left[1,  \mathsf{L}\left[\pi_i\right] \right] \right) \tag*{\textbf{[EL.2]}} \\
&=~ \lambda^{-1}\left( \eta \left[1,  \pi_i \right] \right) \tag*{\textbf{[L.3]}} \\
&=~ \lambda^{-1}\left( \lambda(\pi_1) \left[1,  \pi_i \right] \right) \\
&=~ \lambda^{-1}\left( \lambda(\pi_1\pi_i) \right)  \tag{\ref{CCC1}}\\
&=~ \pi_1\pi_i
\end{align*}
\noindent \textbf{[L.4]}: $\mathsf{L}^{C}\left[ \langle f, g \rangle \right] = \left \langle \mathsf{L}^{C}[f], \mathsf{L}^{C}[g] \right \rangle$ \\ \\
\noindent Recall that in any Cartesian closed category, we always have that $[C, A \times B] \cong [C, A] \times [C,B]$. So let $\theta_{C,A,B}: [C,A] \times [C,B] \to [C, A \times B]$ be the natural isomorphism defined as follows: 
\[ \theta := \lambda\left( \langle (1 \times \pi_0) \epsilon, (1 \times \pi_1) \epsilon \rangle \right) \]
with inverse $\theta^{-1}_{C,A,B}: [C, A \times B] \to [C,A] \times [C,B]$ defined as follows: 
\[ \theta^{-1}_{C,A,B} := \langle [1, \pi_0], [1, \pi_1] \rangle \]
To not overload notation, we will omit the subscripts of $\theta$ and $\theta^{-1}$. We first compute that: 
\begin{align*}
\mathsf{L}[\theta^{-1}] &=~ \mathsf{L}\left[ \langle [1, \pi_0], [1, \pi_1] \rangle \right] \\
&=~ \left \langle \mathsf{L}\left[ [1, \pi_0] \right], \mathsf{L}\left[ [1, \pi_1] \right] \right \rangle \tag*{\textbf{[L.4]}} \\
&=~  \left \langle [1, \mathsf{L}\left[ \pi_0 \right]] ,  [1, \mathsf{L}\left[ \pi_1 \right]]   \right \rangle \tag*{\textbf{[EL.2]}} \\
&=~ \langle [1, \pi_0], [1, \pi_1] \rangle \tag*{\textbf{[L.3]}}
\end{align*}
Therefore, $\theta^{-1}$ is $\mathsf{L}$-linear. Since $\theta^{-1}$ is an isomorphism, by Lemma \ref{lstable-lem}.(\ref{lstable-lem.iso}), it follows that $\theta$ is also $\mathsf{L}$-linear. Next observe that in any Cartesian closed category, the following equalities holds \cite[Part I, Section 2]{lambek1988introduction}:
\begin{equation}\label{lambdapair}\begin{gathered} \lambda(\langle f, g \rangle) = \left \langle \lambda(f), \lambda(g) \right \rangle \theta \qquad \qquad \left \langle \lambda^{-1}(h), \lambda^{-1}(k) \right \rangle = \lambda^{-1}\left(\left \langle h, k \right \rangle \theta \right)
 \end{gathered}\end{equation}
 which follows from $[C, A \times B] \cong [C,A] \times [C,B]$. As such, we can compute that: 
\begin{align*}
\mathsf{L}^{C}\left[ \langle f, g \rangle \right] &=~ \lambda^{-1}\left( \mathsf{L}[\lambda(\langle f, g \rangle)]\right) \\
&=~ \lambda^{-1}\left( \mathsf{L}\left[ \langle \lambda(f), \lambda(g) \rangle \theta \right]\right) \tag{\ref{lambdapair}} \\
&=~ \lambda^{-1}\left( \mathsf{L}\left[ \langle \lambda(f), \lambda(g) \rangle \right] \theta \right)  \tag{$\theta$ is $\mathsf{L}$-linear + Lem.\ref{lstable-lem}.(\ref{lstable-lem.post})} \\
&=~  \lambda^{-1}\left(  \left \langle \mathsf{L}\left[\lambda(f) \right], \mathsf{L}\left[\lambda(g) \right] \right \rangle \theta \right) \tag*{\textbf{[L.4]}} \\
&=~ \left \langle \lambda^{-1}\left( \mathsf{L}[\lambda(f)]\right), \lambda^{-1}\left( \mathsf{L}[\lambda(g)]\right) \right \rangle  \tag{\ref{lambdapair}} \\
&=~ \left \langle \mathsf{L}^{C}[f], \mathsf{L}^{C}[g] \right \rangle
\end{align*}
\noindent \textbf{[L.5]}: $\mathsf{L}^{C}[\langle \pi_0, f \rangle g] = \langle \pi_0, \mathsf{L}^{C}[f] \rangle~ \mathsf{L}^{C}\left[\left \langle \pi_0, \pi_1 + \langle \pi_0, 0 \rangle f \right \rangle g\right]$ \\ \\
\noindent First note that in a Cartesian closed left additive category, since post-composition preserves the additive structure, it follows that we always have the following equalities: 
\begin{equation}\label{[]-add}\begin{gathered} [f,0] = 0 \qquad \qquad [f,g+h] = [f,g] + [f,h] \end{gathered}\end{equation}
Next, note that by \textbf{[EL.1]}, $\eta$ and $\mu$ are $\mathsf{L}$-linear. So in particular, by Lemma \ref{lstable-lem}.(\ref{lstable-lem.add}), $\eta$ and $\mu$ are also additive. Also, recall that the monad identities are: 
\begin{equation}\label{monadidentities}\begin{gathered} \mu \mu = [1, \mu] \mu \qquad \qquad \eta \mu = 1 = [1, \eta] \mu \end{gathered}\end{equation}
Lastly, note that we have the following equality in any Cartesian closed category: 
\begin{equation}\label{CCC2}\begin{gathered} \lambda(\langle \pi_0, f \rangle g) = \lambda(f) [1,\lambda(g)] \mu \qquad \qquad \lambda^{-1}\left( h [1, k] \mu \right) = \langle \pi_0, \lambda^{-1}(h) \rangle \lambda^{-1}(k) \end{gathered}\end{equation}
which follows from the fact the Kleisli category of $E^C$ is isomorphic to the simple slice category $\mathbb{X}[C]$  \cite[Part I, Section 7]{lambek1988introduction}. Therefore, we compute: 
\begin{align*}
\mathsf{L}^{C}[\langle \pi_0, f \rangle g] &=~ \lambda^{-1}\left( \mathsf{L}[\lambda\left(\langle \pi_0, f \rangle g \right)]\right) \\
&=~  \lambda^{-1}\left( \mathsf{L}\left[ \lambda(f) [1, \lambda(g)] \mu \right]\right)  \tag{\ref{CCC2}}\\
&=~  \lambda^{-1}\left( \mathsf{L}\left[ \lambda(f) [1, \lambda(g)]  \right] \mu \right) \tag{\textbf{[EL.1]} + Lem.\ref{lstable-lem}.(\ref{lstable-lem.post})} \\
&=~  \lambda^{-1}\left( \mathsf{L}\left[ \lambda(f) \right] \mathsf{L}\left[ \left(1 + 0 \lambda(f) \right) [1, \lambda(g)]  \right] \mu \right) \tag*{\textbf{[L.5]}} \\ 
&=~  \lambda^{-1}\left( \mathsf{L}\left[ \lambda(f) \right] \mathsf{L}\left[ \left([1, \eta]\mu + 0 \lambda\left(f \right) \eta \mu \right) [1, \lambda(g)]  \right] \mu \right) \tag{\ref{monadidentities}} \\
&=~  \lambda^{-1}\left( \mathsf{L}\left[ \lambda(f) \right] \mathsf{L}\left[ \left([1,\eta] + 0 \lambda\left( f \right) \eta \right) \mu [1, \lambda(g)]  \right] \mu \right)   \tag{$\mu$ is additive} \\
&=~  \lambda^{-1}\left( \mathsf{L}\left[ \lambda(f) \right] \mathsf{L}\left[ \left([1,\eta] + 0 \eta [1, \lambda(f)] \right) \mu [1, \lambda(g)]  \right] \mu \right)  \tag{Naturality of $\eta$} \\
&=~  \lambda^{-1}\left( \mathsf{L}\left[ \lambda(f) \right] \mathsf{L}\left[ \left([1,\eta] + 0 [1, \lambda(f)] \right) \mu [1, \lambda(g)]  \right] \mu \right) \tag{$\eta$ is additive} \\
&=~  \lambda^{-1}\left( \mathsf{L}\left[ \lambda(f) \right] \mathsf{L}\left[ \left([1,\eta] + [1,0] [1, \lambda(f)] \right) \mu [1, \lambda(g)]  \right] \mu \right)  \tag{\ref{[]-add}} \\
&=~ \lambda^{-1}\left( \mathsf{L}\left[ \lambda(f) \right] \mathsf{L}\left[ \left([1,\eta] +[1,  0 \lambda(f)] \right) \mu [1, \lambda(g)]  \right] \mu \right)   \tag{Functoriality of $[1,-]$}\\
&=~ \lambda^{-1}\left( \mathsf{L}\left[ \lambda(f) \right] \mathsf{L}\left[ \left[1, \eta +  0 \lambda(f) \right]  \mu [1, \lambda(g)]  \right] \mu \right)   \tag{\ref{[]-add}} \\
&=~ \lambda^{-1}\left( \mathsf{L}\left[ \lambda(f) \right] \mathsf{L}\left[ \left[1, \eta +  \lambda\left((1 \times 0) f\right) \right]  \mu [1, \lambda(g)]  \right] \mu \right)   \tag{\ref{CCC1}}\\
&=~ \lambda^{-1}\left( \mathsf{L}\left[ \lambda(f) \right] \mathsf{L}\left[ \left[1, \eta +  \lambda\left( \langle\pi_0, 0 \rangle f\right) \right]  \mu [1, \lambda(g)]  \right] \mu \right)  \\
&=~ \lambda^{-1}\left( \mathsf{L}\left[ \lambda(f) \right] \mathsf{L}\left[ \left[1, \lambda(\pi_1) +  \lambda\left( \langle\pi_0, 0 \rangle f\right) \right]  \mu [1, \lambda(g)]  \right] \mu \right)  \\
&=~  \lambda^{-1}\left(  \mathsf{L}\left[ \lambda(f) \right] \mathsf{L}\left[ \left[ 1,  \lambda\left( \pi_1 + \langle \pi_0, 0 \rangle f \right)\right] \mu ~[1, \lambda(g)] \right] \mu  \right) \\
&=~  \lambda^{-1}\left(  \mathsf{L}\left[ \lambda(f) \right] \mathsf{L}\left[ \left[ 1,  \lambda\left( \pi_1 + \langle \pi_0, 0 \rangle f \right)\right] \left[1,  [1, \lambda(g)] \right] \mu \right] \mu  \right) \tag{Naturality of $\mu$} \\
&=~  \lambda^{-1}\left(  \mathsf{L}\left[ \lambda(f) \right] \mathsf{L}\left[ \left[ 1,  \lambda\left( \pi_1 + \langle \pi_0, 0 \rangle f \right)\right] \left[1,  [1, \lambda(g)] \right] \right] \mu \mu  \right) \tag{\textbf{[EL.1]} + Lem.\ref{lstable-lem}.(\ref{lstable-lem.post})} \\
&=~  \lambda^{-1}\left(  \mathsf{L}\left[ \lambda(f) \right] \mathsf{L}\left[ \left[ 1,  \lambda\left( \pi_1 + \langle \pi_0, 0 \rangle f \right)  [1, \lambda(g)] \right] \right] \mu \mu  \right) \tag{Functoriality of $[1,-]$}\\
&=~  \lambda^{-1}\left(  \mathsf{L}\left[ \lambda(f) \right] \left[ 1, \mathsf{L}\left[ \lambda\left( \pi_1 + \langle \pi_0, 0 \rangle f \right)  [1, \lambda(g)] \right] \right] \mu \mu  \right) \tag*{\textbf{[EL.2]}} \\
&=~  \lambda^{-1}\left(  \mathsf{L}\left[ \lambda(f) \right] \left[ 1, \mathsf{L}\left[ \lambda\left( \pi_1 + \langle \pi_0, 0 \rangle f \right)  [1, \lambda(g)] \right] \right] [1,\mu]\mu  \right) \tag{\ref{monadidentities}} \\
&=~  \lambda^{-1}\left(  \mathsf{L}\left[ \lambda(f) \right] \left[ 1, \mathsf{L}\left[ \lambda\left( \pi_1 + \langle \pi_0, 0 \rangle f \right)  [1, \lambda(g)] \right] \mu \right] \mu  \right) \tag{Functoriality of $[1,-]$}\\
&=~  \lambda^{-1}\left(  \mathsf{L}\left[ \lambda(f) \right] \left[ 1, \mathsf{L}\left[ \lambda\left( \pi_1 + \langle \pi_0, 0 \rangle f \right)  [1, \lambda(g)] \mu \right] \right] \mu  \right) \tag{\textbf{[EL.1]} + Lem.\ref{lstable-lem}.(\ref{lstable-lem.post})} \\
&=~  \lambda^{-1}\left(  \mathsf{L}\left[ \lambda(f) \right] \left[ 1, \mathsf{L}\left[\lambda\left( \left \langle \pi_0, \pi_1 + \langle \pi_0, 0 \rangle f \right \rangle g \right) \right] \right] \mu  \right) \tag{\ref{CCC2}}\\
&=~ \langle \pi_0,  \lambda^{-1}\left( \mathsf{L}[\lambda(f)]\right) \rangle~  \lambda^{-1}\left( \mathsf{L}\left[\lambda\left( \left \langle \pi_0, \pi_1 + \langle \pi_0, 0 \rangle f \right \rangle g \right) \right]\right) \tag{\ref{CCC2}}\\
&=~ \langle \pi_0, \mathsf{L}^{C}[f] \rangle~ \mathsf{L}^{C}\left[\left \langle \pi_0, \pi_1 + \langle \pi_0, 0 \rangle f \right \rangle g\right]
\end{align*}
\noindent \textbf{[L.6]}: $\mathsf{L}^{C}\left[\mathsf{L}^{C}[f]\right] = \mathsf{L}^{C}[f]$
\begin{align*}
\mathsf{L}^{C}\left[\mathsf{L}^{C}[f]\right] &=~ \lambda^{-1}\left( \mathsf{L}\left[\lambda\left(\lambda^{-1}\left( \mathsf{L}[\lambda(f)]\right) \right) \right]\right) \\
&=~ \lambda^{-1}\left( \mathsf{L}\left[\mathsf{L}[\lambda(f)] \right]\right) \\
&=~  \lambda^{-1}\left( \mathsf{L}[\lambda(f)]\right)  \tag*{\textbf{[L.6]}} \\
&=~ \mathsf{L}^{C}[f]
\end{align*}
\noindent \textbf{[L.7]}: $\mathsf{L}^C_1[\mathsf{L}^C_0[f]] =  \mathsf{L}^C_0[\mathsf{L}^C_1[f]]$ \\ \\
\noindent Recall that in any Cartesian closed category, we always have that $[A,[C,B]] \cong [C \times A, B]$. So let $\phi_{A,C,B}: [A,[C,B]] \to [C \times A, B]$ be the natural isomorphism defined as follows: 
\[ \phi_{A,C,B} = \lambda\left( \alpha^{-1} (1 \times \epsilon) \epsilon \right) \]
%with inverse $\phi^{-1}_{C,A,B}: [A \times C, B] \to [C,[A,B]]$ defined as follows: 
%\[ \phi^{-1}_{C,A,B} := \lambda\left( \lambda( \alpha \epsilon) \right) \]
 As before, as to not overload notation, we will omit the subscripts of $\phi$. We first note that we could have also expressed $\phi$ is terms of $\mu$ and $[-,-]$ as follows: 
\begin{equation}\label{phimu}\begin{gathered}\phi = \left[ \pi_1, [\pi_0, 1] \right] \mu
 \end{gathered}\end{equation}
Therefore, we can compute that: 
\begin{align*}
\mathsf{L}[\phi] &=~ \mathsf{L}\left[ \left[\pi_1, [\pi_0,1] \right] \mu \right] \tag{\ref{phimu}} \\
&=~  \mathsf{L}\left[ \left[\pi_1, [\pi_0,1] \right]\right] \mu  \tag{\textbf{[EL.1]} + Lem.\ref{lstable-lem}.(\ref{lstable-lem.post})} \\
&=~  \left[ \pi_1, \mathsf{L}\left[ [\pi_0,1] \right]\right] \mu\tag*{\textbf{[EL.2]}} \\
&=~  \left[ \pi_1,  [\pi_0,\mathsf{L}[1]] \right] \mu\tag*{\textbf{[EL.2]}} \\
&=~  \left[ \pi_1,  [\pi_0,1] \right] \mu\tag*{\textbf{[L.3]}} \\
&=~ \phi \tag{\ref{phimu}} 
\end{align*}
So $\phi$ is $\mathsf{L}$-linear. Next observe that for a map $f: C \times (A \times B) \to D$ (or a map $g: B \to [A, [C,D]]$) we can apply the curry operator (or uncurry operator) twice and the following equalities hold in any Cartesian closed category (which we leave to the reader to check for themselves): 
\begin{equation}\label{phialpha}\begin{gathered}\lambda(\lambda(f)) \phi =  \lambda(\alpha^{-1}f) \qquad \qquad \lambda^{-1}\left( \lambda^{-1}\left( g \right) \right)  = \alpha  \lambda^{-1}(g \phi) 
 \end{gathered}\end{equation}
 \begin{equation}\label{phibeta}\begin{gathered}\lambda\left( \tau \lambda(f) \right) \phi = \lambda(\beta^{-1}f) \qquad \qquad  \lambda^{-1}\left( \tau \lambda^{-1}\left( g \right) \right) = \beta \lambda^{-1}\left( g \phi \right) 
 \end{gathered}\end{equation}
 As such, we can compute the following:  
\begin{align*}
\mathsf{L}^C_0[f] &=~ \beta \mathsf{L}^{C \times B}[\beta^{-1}f] \\
&=~ \beta \lambda^{-1}\left( \mathsf{L}[\lambda(\beta^{-1}f)] \right) \\
&=~ \beta \lambda^{-1}\left( \mathsf{L}[\lambda\left( \tau \lambda(f) \right) \phi] \right)  \tag{\ref{phibeta}}  \\
&=~ \beta \lambda^{-1}\left( \mathsf{L}[\lambda\left( \tau \lambda(f) \right) ] \phi \right) \tag{$\phi$ is $\mathsf{L}$-linear + Lem.\ref{lstable-lem}.(\ref{lstable-lem.post})} \\
&=~ \lambda^{-1}\left( \tau \lambda^{-1}\left( \mathsf{L}[\lambda\left( \tau \lambda(f) \right) ]  \right) \right) \tag{\ref{phibeta}}  \\
&=~ \lambda^{-1}\left( \mathsf{L}_0[\lambda(f)] \right) \\ \\
\mathsf{L}^C_1[f] &=~  \alpha \mathsf{L}^{C \times A}[\alpha^{-1}f] \\ 
&=~ \alpha \lambda^{-1}\left( \mathsf{L}[\lambda(\alpha^{-1}f)] \right) \\
&=~ \alpha \lambda^{-1}\left( \mathsf{L}[\lambda(\lambda(f)) \phi] \right)  \tag{\ref{phialpha}}  \\
&=~ \alpha \lambda^{-1}\left( \mathsf{L}[\lambda(\lambda(f))] \phi \right) \tag{$\phi$ is $\mathsf{L}$-linear + Lem.\ref{lstable-lem}.(\ref{lstable-lem.post})} \\
&=~ \lambda^{-1}\left( \lambda^{-1}\left( \mathsf{L}[\lambda(\lambda(f))] \right) \right)  \tag{\ref{phialpha}}  \\
&=~\lambda^{-1}\left(  \mathsf{L}_1[\lambda(f)] \right)
\end{align*}
So we have the following equalities: 
\begin{equation}\label{L0L1exp}\begin{gathered} \mathsf{L}^C_0[f] =  \lambda^{-1}\left( \mathsf{L}_0[\lambda(f)] \right) \quad \quad \quad \mathsf{L}^C_1[f] =  \lambda^{-1}\left( \mathsf{L}_1[\lambda(f)] \right)
 \end{gathered}\end{equation}
Then we have that: 
\begin{align*}
\mathsf{L}^C_1[\mathsf{L}^C_0[f]] &=~ \lambda^{-1}\left(  \mathsf{L}_1\left[\lambda\left(  \lambda^{-1}\left( \mathsf{L}_0[\lambda(f)]  \right) \right) \right] \right) \tag{\ref{L0L1exp}} \\
&=~ \lambda^{-1}\left(  \mathsf{L}_1\left[\mathsf{L}_0[\lambda(f)]   \right] \right) \\
&=~ \lambda^{-1}\left(  \mathsf{L}_0\left[\mathsf{L}_1[\lambda(f)]   \right] \right) \tag*{\textbf{[EL.3]}} \\
&=~ \lambda^{-1}\left(  \mathsf{L}_0\left[\lambda\left(  \lambda^{-1}\left( \mathsf{L}_1[\lambda(f)]  \right) \right) \right] \right) \\
&=~ \mathsf{L}^C_0[\mathsf{L}^C_1[f]] \tag{\ref{L0L1exp}} 
\end{align*} 
\noindent \textbf{[L.8]}: $(h \times 1)\mathsf{L}^{C^\prime}[f]  = \mathsf{L}^{C}[(h \times 1)f]$ \\ \\
We first observe that for any map $h$: 
\begin{align*}
\mathsf{L}[h,1] &=~ [h, \mathsf{L}[1] ] \tag*{\textbf{[EL.2]}} \\
&=~ [h,1] \tag*{\textbf{[L.3]}} 
\end{align*}
Therefore, $[h,1]$ is $\mathsf{L}$-linear. 
\begin{align*}
(h \times 1)\mathsf{L}^{C^\prime}[f]  &=~ (h \times 1) \lambda^{-1}\left( \mathsf{L}[\lambda(f)]\right) \\
&=~  \lambda^{-1}\left(\mathsf{L}[\lambda(f)] [h,1]\right) \tag{\ref{CCC1}}\\
&=~ \lambda^{-1}\left(\mathsf{L}[\lambda(f) [h,1] ] \right)  \tag{$[h,1]$ is $\mathsf{L}$-linear + Lem.\ref{lstable-lem}.(\ref{lstable-lem.post})} \\
&=~ \lambda^{-1}\left(\mathsf{L}[\lambda\left( (h \times 1) f \right) ] \right) \tag{\ref{CCC1}}\\
&=~ \mathsf{L}^{C}[(h \times 1)f]
\end{align*}
\noindent \textbf{[L.ev]}: $\mathsf{L}^C[\epsilon] = \epsilon$
\begin{align*}
\mathsf{L}^C[\epsilon]  &=~  \lambda^{-1}\left( \mathsf{L}[\lambda(\epsilon)]\right) \\
&=~  \lambda^{-1}\left( \mathsf{L}[1]\right) \\
&=~  \lambda^{-1}\left( 1 \right) \tag*{\textbf{[L.3]}} \\
&=~ \epsilon 
\end{align*}
So we conclude that $\mathsf{L}^C$ is a closed system of linearizing combinators. We also have that: 
\begin{align*}
\langle 0,1 \rangle \mathsf{L}^\top[\pi_1 f] &=~ \langle 0,1 \rangle \lambda^{-1}\left( \mathsf{L}[\lambda(\pi_1 f)]\right) \\
&=~ \langle 0,1 \rangle \lambda^{-1}\left( \mathsf{L}\left[\lambda(\pi_1) [1,f] \right]\right) \tag{\ref{CCC1}}\\
&=~ \langle 0,1 \rangle \lambda^{-1}\left( \mathsf{L}\left[\eta [1,f] \right]\right) \\
&=~ \langle 0,1 \rangle \lambda^{-1}\left(\eta ~\mathsf{L}\left[ [1,f] \right]\right) \tag{\textbf{[EL.1]} + Lem.\ref{lstable-lem}.(\ref{lstable-lem.pre})} \\
&=~ \langle 0,1 \rangle \lambda^{-1}\left(\eta ~\left[ 1, \mathsf{L}[f] \right]\right) \tag*{\textbf{[EL.2]}} \\
&=~ \langle 0,1 \rangle \lambda^{-1}\left(\lambda(\pi_1) \left[ 1, \mathsf{L}[f] \right]\right)  \\
&=~ \langle 0,1 \rangle \lambda^{-1}\left(\lambda\left(\pi_1 \mathsf{L}[f]\right) \right)   \tag{\ref{CCC1}}\\
&=~ \langle 0,1 \rangle \pi_1 \mathsf{L}[f] \\
&=~ \mathsf{L}[f] 
\end{align*}
Therefore, $\mathsf{L}[f] = \langle 0,1 \rangle \mathsf{L}^\top[\pi_1 f]$ and so $\mathsf{L}$ is precisely the induced linearizing combinator from Proposition \ref{lemL1}. Next we must show that $\mathsf{D}_\mathsf{L}$ is a differential combinator which also satisfies {\bf [CD.$\lambda$]} (or equivalently {\bf [CD.ev]}). However, we have that: 
\begin{align*}
\mathsf{D}_\mathsf{L}[f] &=~ \lambda^{-1}\left( \mathsf{L}[\lambda(\oplus_A f)] \right) \\
&=~ \mathsf{L}^A[\oplus_A f] 
\end{align*}
Therefore, $\mathsf{D}_\mathsf{L}$ is precisely the induced differential combinator from Proposition \ref{LDprop}. Furthermore, by Proposition \ref{LDclosedprop}.(\ref{LDclosedprop.ii}), $\mathsf{D}_\mathsf{L}$ satisfies {\bf [CD.$\lambda$]} (or equivalently {\bf [CD.ev]}). So we conclude that a Cartesian left additive category with an exponential linearizing combinator is a Cartesian closed differential category. 
\end{proof} 

We conclude this paper by stating the second main result of this paper. 

\begin{theorem}\label{finalthm} For a Cartesian closed left additive category $\mathbb{X}$, there is a bijective correspondence between:
\begin{enumerate}[{\em (i)}]
\item \label{finalthm.i} Differential combinators $\mathsf{D}$ on $\mathbb{X}$ which satisfy {\bf [CD.$\lambda$]} (or equivalently {\bf [CD.ev]}); 
\item \label{finalthm.ii} Closed systems of linearizing combinators $\mathsf{L}^C$ on $\mathbb{X}$;
\item \label{finalthm.iii} Exponentiable linearizing combinators $\mathsf{L}$ on $\mathbb{X}$. 
\end{enumerate}
Therefore, a Cartesian closed differential category is precisely a Cartesian closed left additive category equipped with a exponentiable linearizing combinator or equivalently a Cartesian closed left additive category equipped with a closed system of linearizing combinators. 
\end{theorem} 
\begin{proof} That (\ref{finalthm.i}) and (\ref{finalthm.ii}) are in bijective correspondence follows from Theorem \ref{DLthm}, Proposition \ref{LDclosedprop}.(\ref{LDclosedprop.ii}), and Corollary \ref{cor:6.10}.(\ref{cor:6.10.i}). On the other hand, that (\ref{finalthm.ii}) and (\ref{finalthm.iii}) are in bijective correspondence follows form Proposition \ref{LDclosedprop}.(\ref{LDclosedprop.i}) and Proposition \ref{LDexpprop}.(\ref{LDexpprop.i}) (which when put together shows that their respective constructions are inverses of each other). \end{proof}

\section{Concluding Remarks}\label{sec:conclusion}

The main purpose of this paper is to establish in detail an alternative axiomatization for Cartesian differential categories using a system of linearizing combinators.  This was motivated by the existing techniques of Goodwillie's functor calculus and, in particular, the example of the abelian functor calculus, which focused on the processes of linearization and Taylor approximation rather than differentiation per se \cite{bauer2018directional}. Regarding the abelian functor calculus, a question which now should be answered is whether, in fact, its linearization combinator is exponentiable and if $\mathsf{HoAbCat}_\mathsf{Ch}$ is a Cartesian closed differential category. 

While an alternative axiomatization for Cartesian differential categories is, of course, always of theoretical interest, in this case it was motivated by a practical example in which the alternative axiomatization using linearization arose quite naturally.  Notably the axiomatization presented here gives an algebraic face to the classical relationship between linear approximation and differential.  However, the weakness of this alternative axiomatization should not be overlooked.  The problem is that one needs to assume partial linearization at the outset: this is a significant requirement. On top of the required equalities which must be established, checking that linearizing works in context increases the overhead for checking that one has a Cartesian differential category.  In this regard the total differential combinator has a clear advantage.  Example \ref{counter-example} of $\mathcal{C}^1\text{-}\mathsf{DIFF}$, however, indicates an important aspect of linearization: it can exist for functions which are {\em not\/} infinitely differentiable and these are definitely in the purview of classical analysis.  This suggests that linearization could play a significant role in providing a broader categorical approach for non-smooth analysis.

It is worth emphasizing the discussion at the end of Section \ref{system-sec}. In the development of differential categories, tensor differential categories have always had a guiding role: they provide an important source of examples of Cartesian differential categories by applying the coKleisli construction  (indeed, even tangent categories can often be produced by applying the coEilenberg-Moore construction \cite{cockett_et_al:LIPIcs:2020:11660}).  Thus, it is worth understanding how linearization appears in tensor differential categories.   Somewhat surprisingly it is the correspondence between deriving transformations and the coderelictions for \emph{tensor} differential categories.  This correspondence becomes, under translation into the coKleisli category, the correspondence between differential combinators and systems of linearizing combinators: 
\begin{table}[!h]
\begin{center}
\begin{tabular}{|c|c|}
\hline
$\otimes$-differential categories & Cartesian differential categories \\
\hline 
Deriving transformations & Differential combinators $\mathsf{D}$ \\
$\mathsf{d}: \oc A \otimes A \to \oc A$ & $\infer{\mathsf{D}[f]: A \times A \to B}{f: A \to B}$ \\
\hline 
Coderelictions  & Linearizing Combinators $\mathsf{L}$ \\
$\eta: A \to \oc A$ & $\infer{\mathsf{L}[f]: A \to B}{f: A \to B}$  \\
\hline
\end{tabular}
\end{center}
\label{default}
\end{table}%

Linearizing combinators, thus, should also provide equivalent axiomatizations for generalizations of Cartesian differential categories including generalized Cartesian differential categories \cite{cruttwell2017cartesian}, differential restriction categories \cite{cockett2011differential}, and even tangent categories \cite{cockett2014differential}.  In each setting the precise form that linearization takes needs to be developed: hopefully this development, centred as it is on Cartesian differential categories, will be a useful guide.

% One way forward here is to examine the lifting of coderiliction and differential to the E-M category. JS answer: YES! 

\bibliographystyle{spmpsci}      % mathematics and physical sciences
\bibliography{lincombib}   % name your BibTeX data base
\end{document}